\documentclass[11pt]{article} 

\usepackage[utf8]{inputenc} 

\usepackage{geometry} 
\usepackage{graphicx} 
\usepackage{graphics}

\usepackage{booktabs} 
\usepackage{array} 
\usepackage{paralist} 
\usepackage{verbatim} 
\usepackage{subfig} 
\usepackage{amssymb}
 \usepackage{amsthm}
\usepackage{amsmath,amssymb,amsopn,amsfonts,mathrsfs,amsbsy,amscd}
\usepackage{longtable}
\usepackage[dvipsnames]{xcolor}
\usepackage{ulem}
\usepackage{float}
\usepackage{tabularx}
\usepackage{caption}
\usepackage{bbm}

\usepackage{fancyhdr} 
\usepackage{sectsty}
\allsectionsfont{\sffamily\mdseries\upshape} 

\usepackage[nottoc,notlof,notlot]{tocbibind} 
\usepackage[titles,subfigure]{tocloft} 

\newtheorem{Theo}{Theorem}[section]
\newcommand{\Z}{\mathbbm{Z}}

\title{On the minimum number of Fox colorings of knots}

\author{H. Abchir\\ {\scriptsize Hassan II University. 
		Casablanca, Morocco.}\\\scriptsize{e-mail: hamid.abchir@univh2c.ma}\\M. Elhamdadi\\\scriptsize{University of South Florida, Tampa FL.}\\\scriptsize{e-mail: emohamed@usf.edu}\\S. Lamsifer\\\scriptsize{Hassan II University. 
		Casablanca, Morocco.}\\\scriptsize{e-mail: soukainalamsifer@gmail.com}}

\begin{document}
\maketitle


\begin{abstract} 
	 We investigate Fox colorings of knots that are $17$-colorable.  Precisely, we prove that any $17$-colorable knot has a diagram such that exactly $6$ among the seventeen colors are assigned to the arcs of the diagram.
\end{abstract}
\section{Introduction}
Knot theory has been a very active area of research for more than a century.  In the last fifty years, Knot theory has been successfully applied to chemistry and molecular biology.
Recently some algebraic structures, related to quandles which generalize Fox $n$-colorability, have been used to classify topological structures of proteins were introduced in \cite{ADEM} and labelled bondles.  The $3$-coloring invariant is the simplest invariant that distinguishes the trefoil knot from the trivial knot.  The idea of $3$-coloring and more generally $m$-coloring was developed by R. Fox around 1960 (see \cite{fox1962quick}).  He introduced a diagrammatic definition of colorability of a knot by $\Z_m$ (the integers modulo $m$). Precisely, for any natural number $m$, a knot diagram is said to be 
$m$-colorable if we can assign to each of its arcs an element of $\Z_m$, called the color of that arc, such that, at each crossing, the sum of the colors of the under-arcs is twice the color assigned to the over arc modulo $m$ (see Figure~\ref{Fig.0} below). 
A knot is said to be $m$-colorable if it has an $m$-colorable diagram. For obvious reasons $m$ will be restricted to the odd primes.  A coloring that uses only one color is usually called a trivial coloring.  For an explicit example of a Fox $3$-coloring of the knot $8_{19}$ consult example 60 on page 82 of \cite{EN}. 

Let $p$ be an odd prime integer. Let $K$ be a $p$-colorable knot and let $C_p(K)$ denote the minimal number of colors needed to color a diagram of $K$. 
The problem of finding the minimum number of colors for $p$-colorable knots with primes up to 13 was investigated by many authors.  In 2009, S. Satoh showed in \cite{satoh20095} that 
 $C_5(K)=4$. In 2010, K. Oshiro proved that $C_7(K)=4$ \cite{oshiro2010any}. In 2016, T. Nakamura, Y. Nakanishi and S. Satoh showed in \cite{satoh2016} that $C_{11}(K)=5$. In 2017, M. Elhamdadi and J. Kerr \cite{elhamdadi2016fox} and independently F. Bento and P. Lopes \cite{bento2017minimum} proved that $C_{13}(K)=5$.  In what follows we investigate the case of the prime number $p=17$. 

\newcolumntype{D}[1]{>{\centering}p{#1}}
\begin{center}
\begin{tabular}{|D{3cm}|D{3cm}|}\hline
p & $C_p(K)$  \tabularnewline\hline
3 & 3 \tabularnewline\hline
5 & 4 \tabularnewline\hline
7 & 4 \tabularnewline\hline
11 & 5\tabularnewline\hline
13 & 5 \tabularnewline\hline
\end{tabular}
\end{center}

First, using the result of T. Nakamura, Y. Nakanishi and S. Satoh \cite{satoh2013} which states that for any knot $K$ and any prime $p$, $C_p(K)\ge\lfloor\log_2 p\rfloor +2$, we obtain that $C_{17}(K)\ge 6$. The main result of this article is to show that $C_{17}(K)= 6$.  It is still an open question whether for any $p$-colorable knot $K$, the equality $C_p(K)  =   \lfloor\log_2 p\rfloor +2$ holds.  In addition to the results already known, regarding the problem of finding the minimum number of colors for $p$-colorable knots, the result of this article reinforces the validity that equality holds.
\section{Any $17$-colorable knot can be colored by six colors}
\noindent
Through this article we will adopt the same notations as in \cite{elhamdadi2016fox}. So we will use $\{a|b|c\}$ to denote a crossing, as in Figure \ref{Fig.0} where $b$ is the color of the over-arc and $a$ and $c$ are the colors of the under-arcs with $a+c= 2b$ modulo $17$. When the crossing is of the type $\{a|a|a\}$ (trivial coloring), we will omit over and under-arcs and draw them crossing each other.

\begin{figure}[H]
\begin{center}
	\includegraphics[scale=0.17]{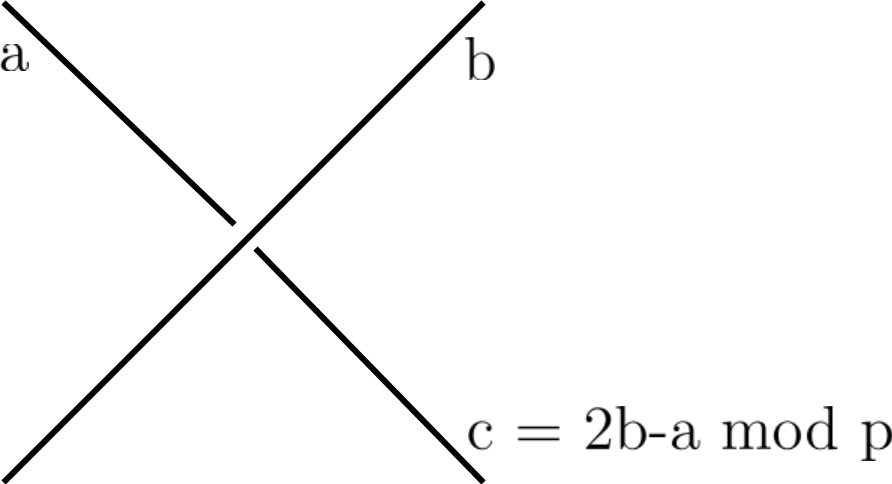}
	\caption{The coloring $\{a|b|c\}$.\label{Fig.0}}
        \end{center}
\end{figure}
Our main result is the following      
\begin{Theo}\label{MainThm}
Any 17-colorable knot has a 17-colored diagram with exactly six colors.
\end{Theo}
\begin{proof}
  Let $D$ be a non-trivially 17-colored knot diagram of a knot $K$. We will show that the integers $\{0,2,3,4,8,12\}$ are enough to color $K$. To do this, we will proceed by steps. At the step number $i$ we will prove that one can do without the color $c_i$, which is the $i$-th number in the ordered list $\lbrace 16, 15, 9, 10, 6, 7, 5, 1, 11, 14, 13\rbrace$.\\ We will start by proving that we can modify $D$ to get an equivalent colored diagram $D_1$ where the color $c_1=16$ is not used. The step $i$, $i\ge 2$, consists in showing that if one begins with a colored diagram $D_{i-1}$ in which none of the already discarded colors $\{c_1,\dots ,c_{i-1}\}$ is used, then one can modify $D_{i-1}$ to get a new equivalent colored diagram $D_i$ where none of the colors $\{c_1,\dots ,c_i\}$ appear.
    Note that any color $c$ can occur in $D$ in three ways:
    \begin{itemize}
    \item at a crossing of the form $\{c|c|c\}$,
    \item or on an over-arc at a crossing $\{a|c|2c-a\}$ for some color $a$
    \item or as the color of an under-arc that connects two crossings of the type $\{2a-c|a|c\}$ and $\{c|b|2b-c\}$ for some colors $a$ and $b$.
\end{itemize}
Then at each step, we will show that in each one of these three cases, one can modify the diagram such that the color $c$ will be eliminated.\\
In all the figures we will use, we denote by $c$ the color we want to drop.  To make things clear, we start by dealing with the first step when $c=16$. We will show that there is a non-trivially equivalent $17$-colored diagram $D_1$ with no arc colored by $16$.\\
\textbf{Case 1}: Assume that $D$ has a crossing of type $\{16|16|16\}$. Then $D$ will necessarily have one of the two crossings, $\{2a+1|a|16\}$ or $\{a|16|15-a\}$ for some $a\ne 16$. Since $2a+1\ne 16$ and $15-a\ne 16$, we deform the arc colored by $a$ as shown in Figure \ref{Fig.104} in the case of the first crossing, or as shown in Figure \ref{Fig.105} in the case of the second crossing. Each of those two deformations provides an equivalent diagram where the crossing $\{16|16|16\}$ disappeared.
\begin{figure}[H]
\centering
	\includegraphics[scale=0.132]{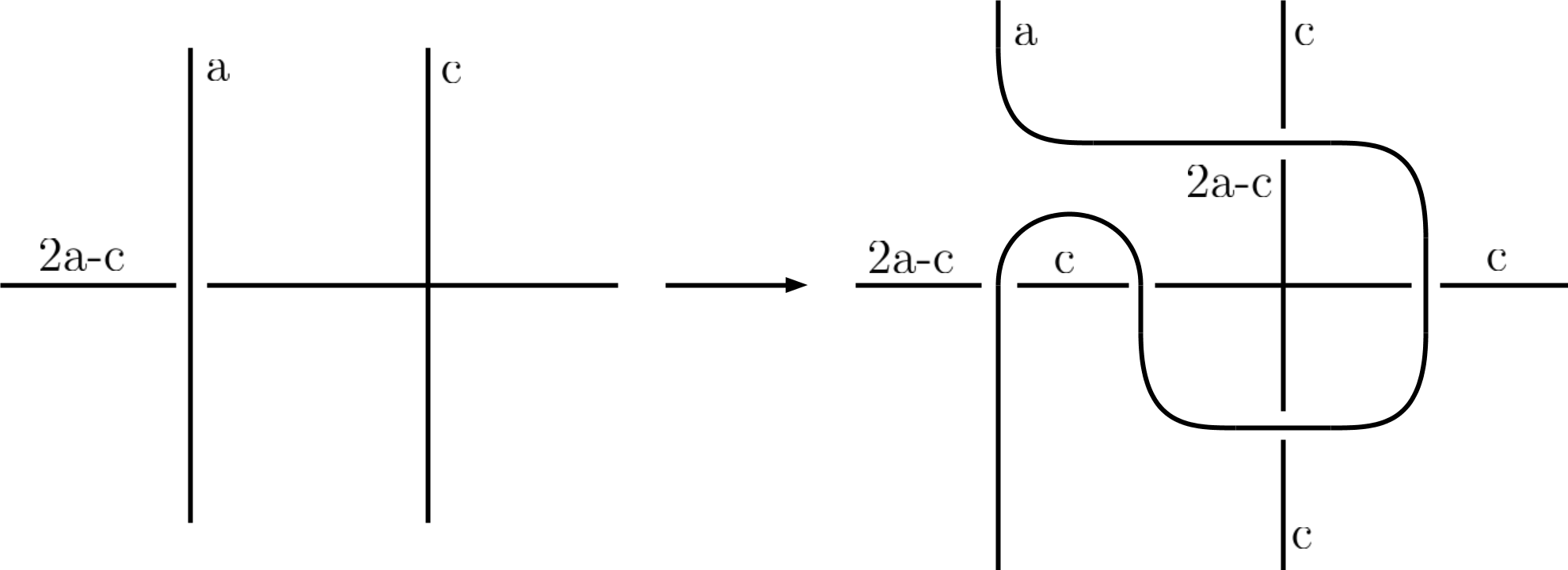} 
	\caption{Transformation of $\left\{c\vert c\vert c \right\}$ when $a$ is the color of an over-arc.\label{Fig.104}}
\end{figure}
In the case of the second crossing, we do the deformation described in Figure \ref{Fig.105}.
\begin{figure}[H]
\centering
	\includegraphics[scale=0.132]{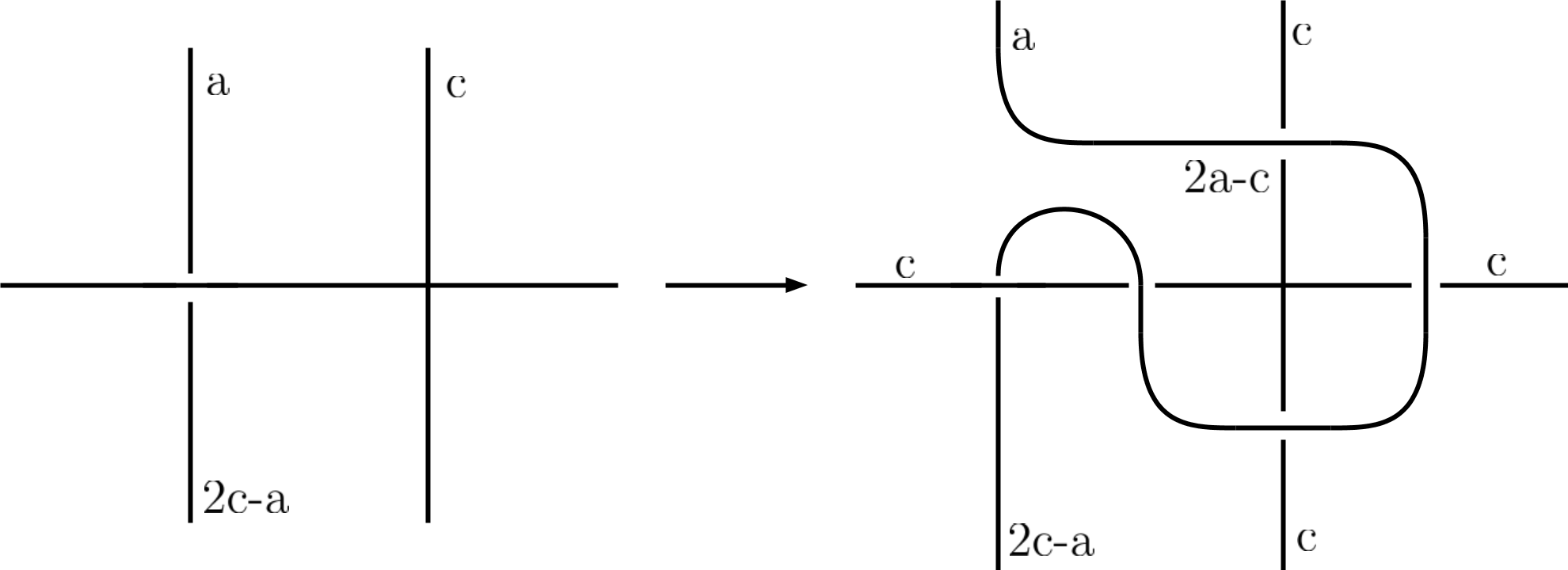} 
	\caption{Transformation of the crossing $\left\{c\vert c\vert c \right\}$ when $a$ is the color of an under-arc.\label{Fig.105}}
\end{figure}
\textbf{Case 2}: Assume that $D$ has a crossing whose over-arc has the color $16$, i.e. it is of the type $\{a|16|15-a\}$ for some $a\ne 16$. Then we deform $D$ as shown in Figure \ref{Fig.108}. We easily check that the generated colors $2a+1$ and $3a+2$ are both distinct from $16$. Furthermore there is no more over-arc with color $16$ in the region concerned by the modification.
\begin{figure}[H]
\centering
	\includegraphics[scale=0.131]{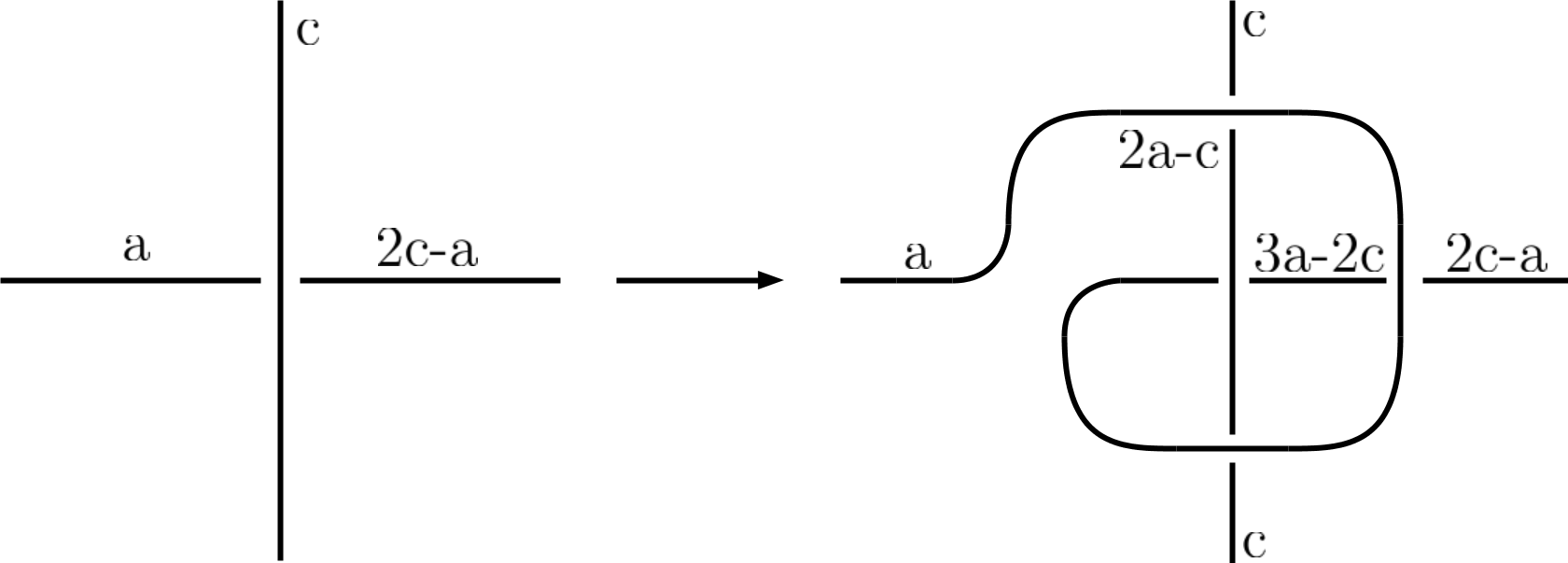} 
	\caption{\label{Fig.108}}
\end{figure}
\textbf{Case 3}: Assume that $D$ has a crossing whose under arc is colored by $16$. Then this under-arc will connect a crossing of the type $\{2a+1|a|16\}$ to a crossing of type $\{16|b|2b+1\}$ for some $a$ and $b$ distinct from $16$. If $a=b$, the deformation shown in Figure \ref{Fig.111} allows to eliminate the color $16$. If $a\neq b$, we do the deformation described in Figure \ref{Fig.115} and then the color $16$ disappears unless when $2a-b=16$ i.e. $b=2a+1$. In this case we apply to $D$ the transformation shown in Figure \ref{Fig.116}.
Finally, we get an equivalent diagram $D_1$ in which no arc has the color $16$.
\begin{figure}[H]
\centering
	\includegraphics[scale=0.115]{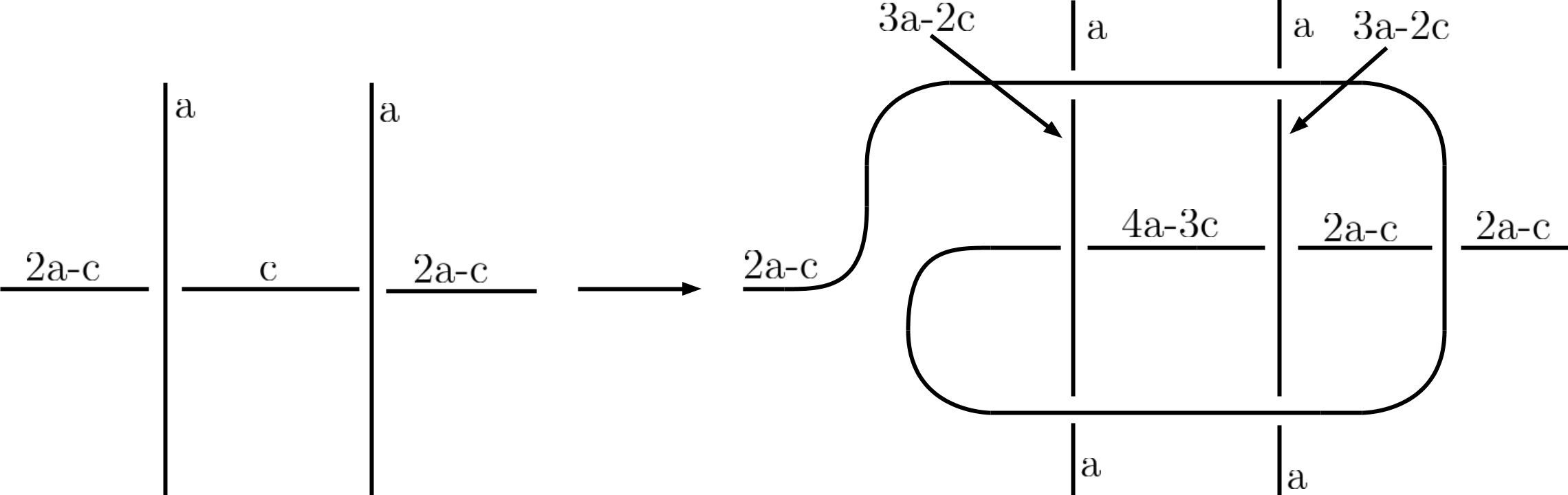} 
	\caption{\label{Fig.111}}
\end{figure}
\begin{figure}[H]
\centering
	\includegraphics[scale=0.115]{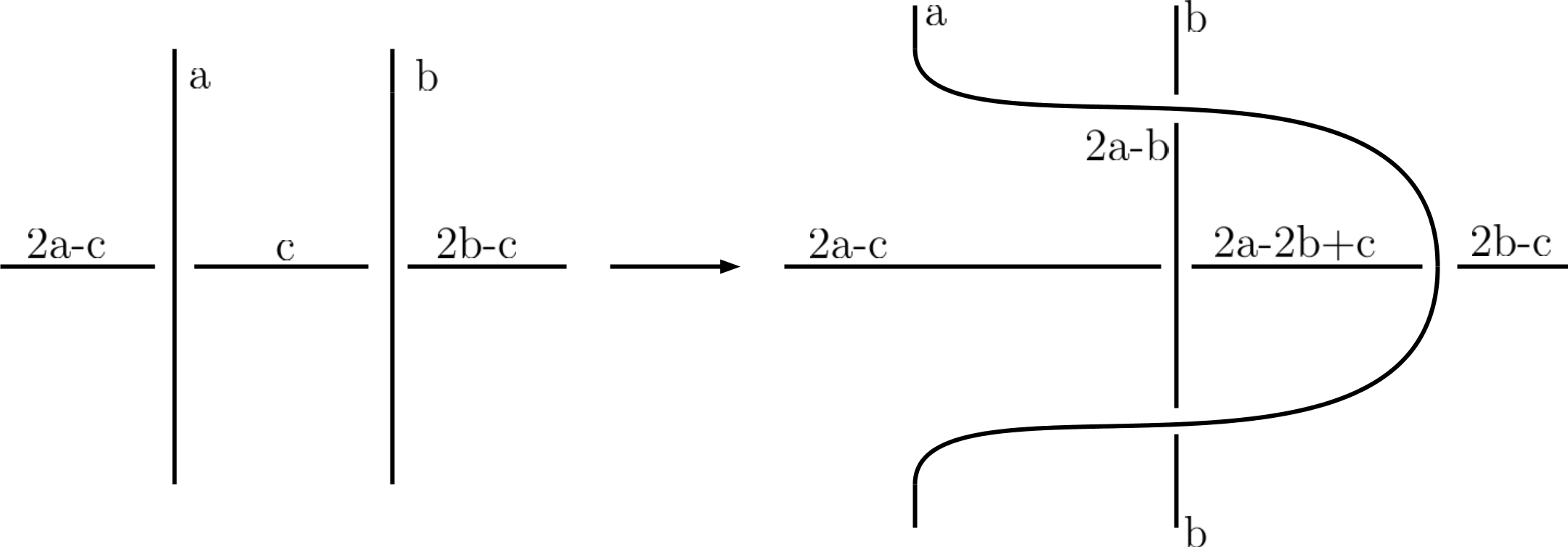} 
	\caption{\label{Fig.115}}
\end{figure}
\begin{figure}[H]
\centering
	\includegraphics[scale=0.115]{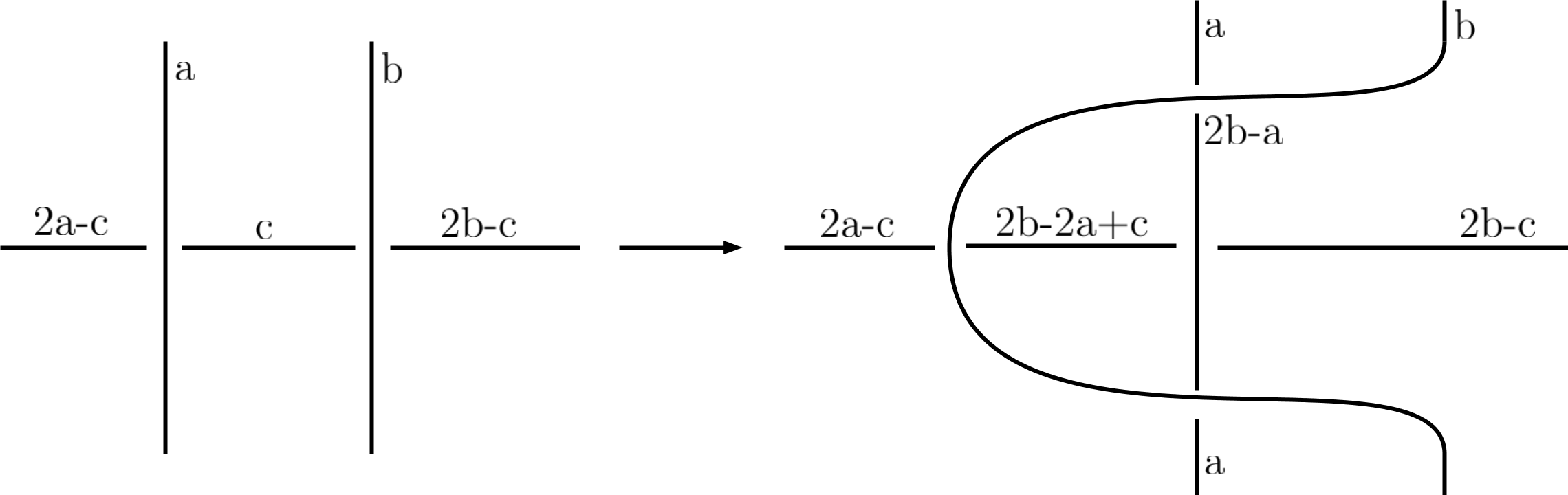} 
	\caption{\label{Fig.116}}
\end{figure}

Now we will deal with a general step $i$, $i\ge 2$. Assume that we have a diagram $D_{i-1}$ that is equivalent to $D$ where the colors $\{c_1,\dots ,c_{i-1}\}$ are not used. We want to show that there exists an equivalent colored diagram $D_i$ which does not use colors $\{c_1,\dots ,c_{i-1},c_i\}$. Here, $c_i$ will be denoted by $c$ as in the figures. Like in the first step, we will consider the three cases:\\
\noindent{\textbf{Case 1}} 
Assume that $D_{i-1}$ has a crossing of the type $\left\{c\vert c\vert c \right\}$. Then there exists a crossing of type $\left\{2a-c\vert a\vert c \right\}$ or $\left\{a\vert c \vert 2c-a \right\}$ for some $a$ distinct from $c$ and $a\notin\{c_1,\dots ,c_{i-1}\}$. In the case of the first crossing we deform the arc colored by $a$ as indicated in Figure \ref{Fig.104} which results in the crossing $\left\{c\vert c\vert c\right\}$ disappearing.\\
In the case of the second crossing, we do the deformation described in Figure \ref{Fig.105}.
The obtained color $2a-c$ will be different from $c$ and $c_k$ iff $a\neq c$ and $a\neq 9(c+c_{k})$, for each $k$ such that $1\le k\le i-1$.\\
If $a=9(c+c_k)$ we resolve the problem by making the deformation of Figure~\ref{Fig.106}, unless if $(c,c_k)=(7,16)$ or $(c,c_k)=(7,15)$ which occur in the sixth step (i.e. $i=6$). For those cases we will apply to the diagram $D_{i-1}=D_5$
one of the deformations described in the Figure~\ref{Fig.107} according to the value of $a$. So, we eliminate all crossings of the type $\left\{c\vert c\vert c \right\}$.
\begin{figure}[H]
\centering
	\includegraphics[scale=0.123]{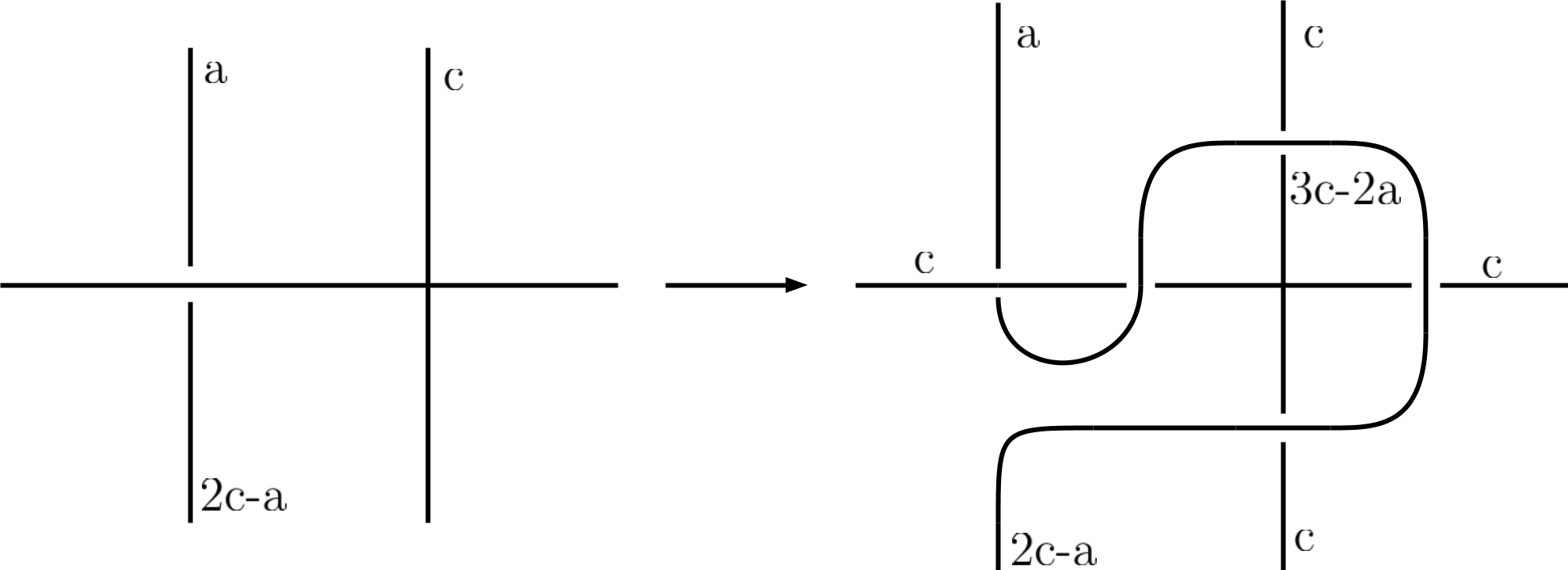} 
	\caption{\label{Fig.106}}
\end{figure}
\begin{figure}[H]
\centering
	\includegraphics[scale=0.103]{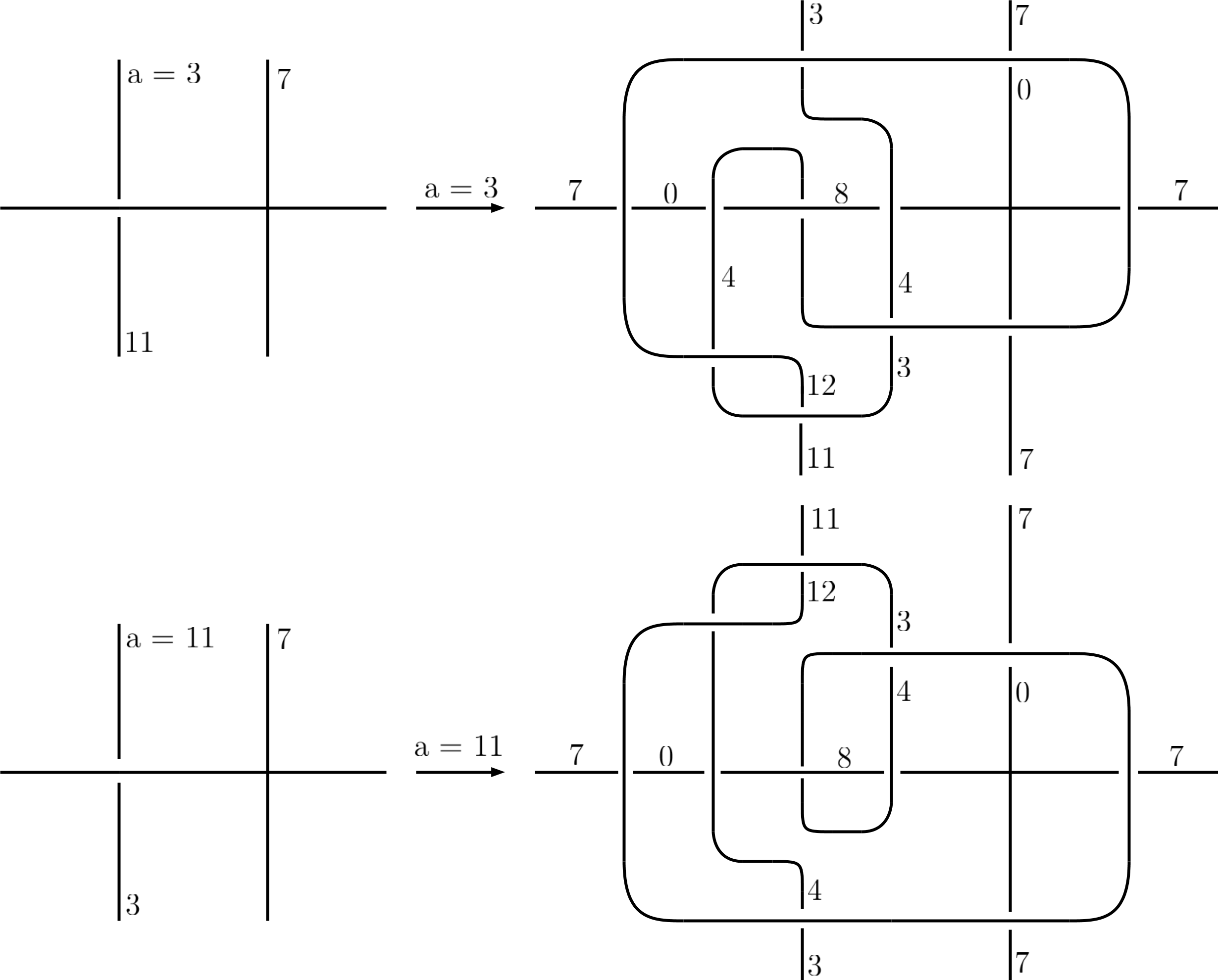} 
	\caption{\label{Fig.107}}
\end{figure}

\noindent{\textbf{Case 2}} Assume that $D_{i-1}$ has a crossing whose over-arc is of color $c=c_i$, i.e. it is of the type $\left\{a\vert c\vert 2c-a \right\}$ for some $a$ different from $c$ and $c_k$, for each $k$, $1\le k\le i-1$. We deform the diagram $D_{i-1}$ as shown in Figure~\ref{Fig.108}.\\
This deformation provides the two new colors $2a-c$ and $3a-2c$, which are different from $c$ and $c_k$ iff $a\ne c$, $a\ne 9(c+c_k)$ and $a\ne 6(c_k+2c)$.
If $a=9(c+c_k)$ or $a= 6(c_k+2c)$ for some $k$, the deformation of Figure~\ref{Fig.109} resolves the problem except when $(c,c_k)=(7,16)$ or $(c,c_k)=(7,15)$ wich occur in the sixth step (i.e. $i=6$). For the two remaining cases we resolve the problem by applying to $D_{i-1}=D_5$ one of the deformations described Figure~\ref{Fig.110} according to the value of $a$.
\begin{figure}[H]
\centering
	\includegraphics[scale=0.132]{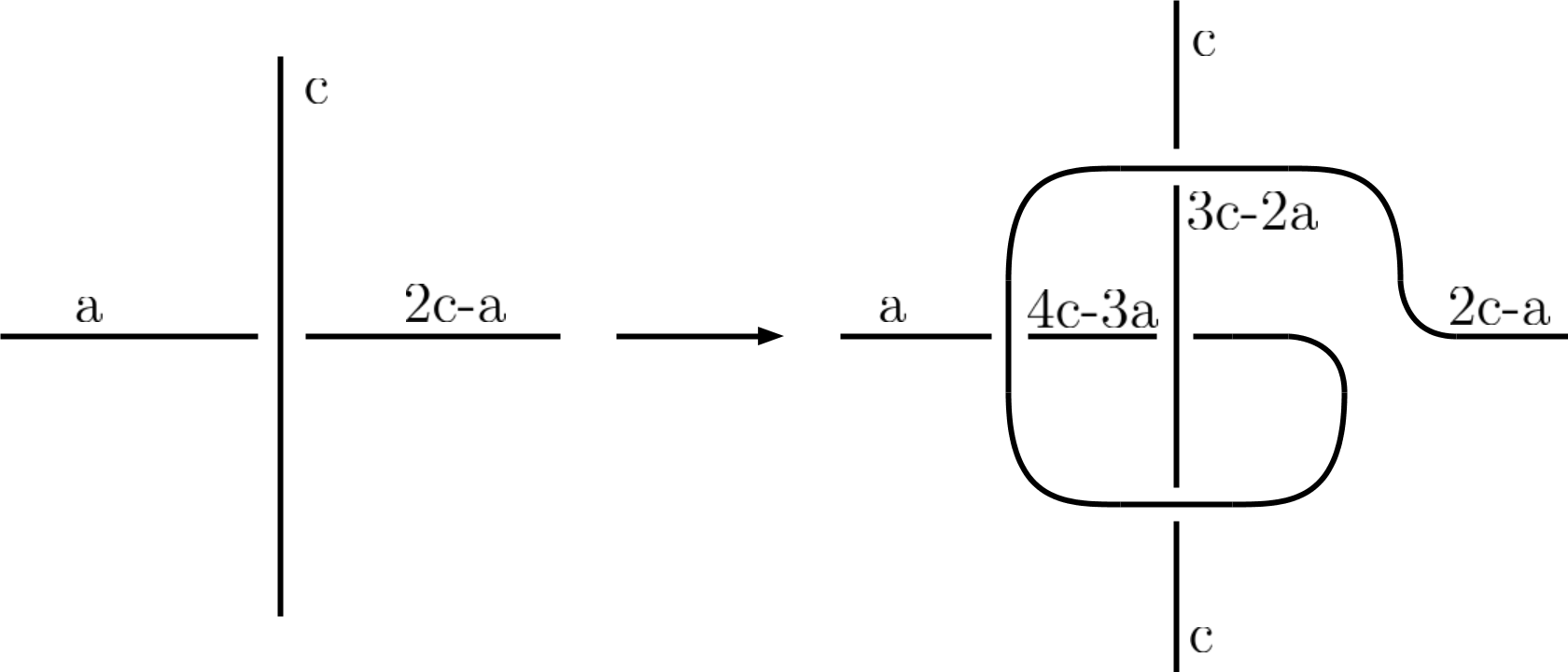} 
	\caption{\label{Fig.109}}
\end{figure}
\begin{figure}[H]
\centering
	\includegraphics[scale=0.11]{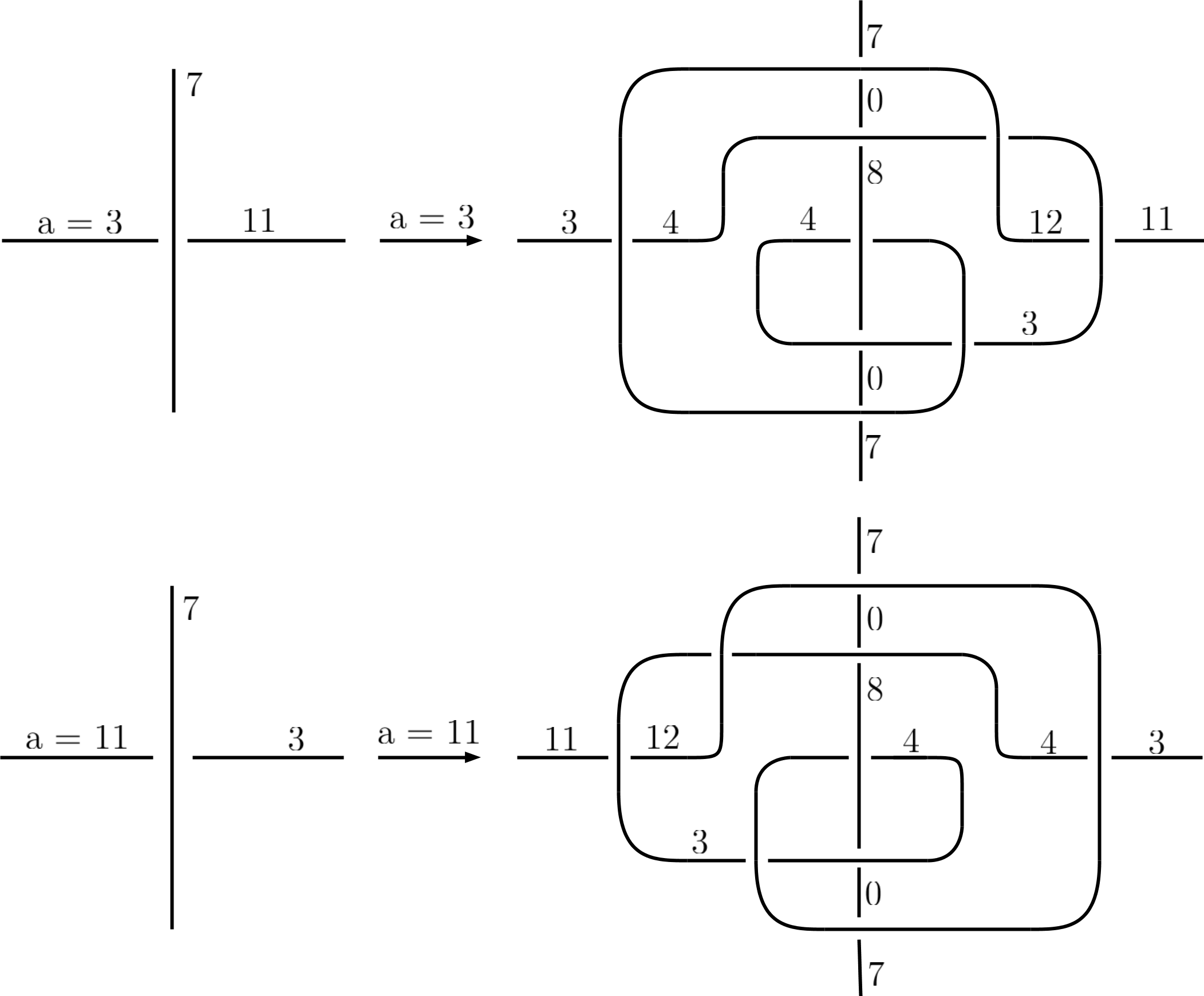} 
	\caption{\label{Fig.110}}
\end{figure}
\noindent{\textbf{Case 3}} Assume that $D_{i-1}$ has a crossing whose under-arc is colored by $c=c_i$. Then $c$ connects two crossings of the type  $\left\{2a-c\vert a\vert c \right\}$ and  $\left\{c\vert b\vert 2b-c \right\}$ for some $a$ and $b$ both distinct from $c$ and $c_k$, for each $k$, $1\le k\le i-1$.\\
\underline{\textbf{If $a=b$}}, we apply to the diagram $D_{i-1}$ the deformation shown in Figure~\ref{Fig.111}. We get the two new colors $3a-2c$ and $4a-3c$. They are different from $c$ and $c_k$ iff $a\ne c$ $a\neq 6(c_k+2c)$ and $a\neq 13(c_k+3c)$, for each $k$, $1\le k\le i-1$.\\
For the remaining cases, if $a=6(c_k+2c)$ or $a=13(c_k+3c)$ (obviously $a\neq c$ and $a\neq c_k$), some other transformations are required. They are listed in the following table.\\
\scriptsize{
\begin{tabular}[t]{|D{0.4cm}|D{1cm}|D{1.6cm}|D{1.6cm}|}\hline
Step&$(c,c_k)$& $a=6(c_k + 2c)$& Required deformation\tabularnewline\hline
$2$&$(15,16)$ & $4$& Fig. \ref{Fig.112}  \tabularnewline\hline
$3$&$(9,16)$ & $0$& Fig. \ref{Fig.114} \tabularnewline\cline{2-4}
&$(9,15)$ & $11$ &Fig. \ref{Fig.112} \tabularnewline\hline
$4$&$(10,16)$ & $12$& Fig. \ref{Fig.112}  \tabularnewline\cline{2-4}
&$(10,15)$& $6$& Fig. \ref{Fig.61}  \tabularnewline\hline
$5$&$(6,10)$ &$13$& Fig. \ref{Fig.112} \tabularnewline\hline
$6$&$(7,15)$ &$4$& Fig. \ref{Fig.114} \tabularnewline\cline{2-4}
&$(7,9)$ &$2$& Fig. \ref{Fig.61} \tabularnewline\cline{2-4}
&$(7,6)$ &$1$& Fig. \ref{Fig.66} \tabularnewline\hline
$7$&$(5,16)$ &$3$ & Fig. \ref{Fig.67} \tabularnewline\cline{2-4}
&$(5,10)$ &$1$& Fig. \ref{Fig.62}  \tabularnewline\cline{2-4}
&$(5,6)$ &$11$& Fig. \ref{Fig.68}  \tabularnewline\cline{2-4}
&$(5,7)$ &$0$& Fig. \ref{Fig.69}  \tabularnewline\hline
$9$&$(11,7)$ &$4$& Fig. \ref{Fig.75}  \tabularnewline\hline
$11$&$(13,14)$ &$2$& Fig. \ref{Fig.83}  \tabularnewline\hline
\end{tabular}}
\scriptsize{\begin{tabular}[t]{|D{0.4cm}|D{1cm}|D{1.8cm}|D{1.6cm}|}\hline
Step&$(c,c_k)$& $a=13(c_k + 3c)$& Required deformation \tabularnewline\hline
$2$&$(15,16)$ & $11$& Fig. \ref{Fig.113} \tabularnewline\hline
$3$&$(9,15)$ & $2$& Fig. \ref{Fig.113}  \tabularnewline\hline
$4$&$(10,16)$ & $3$& Fig. \ref{Fig.113}  \tabularnewline\cline{2-4}
&$(10,15)$ & $7$ & Fig. \ref{Fig.63}  \tabularnewline\cline{2-4}
&$(10,9)$ &$14$& Fig. \ref{Fig.113}   \tabularnewline\hline
$5$&$(6,16)$ &$0$ & Fig. \ref{Fig.113} \tabularnewline\cline{2-4}
&$(6,15)$ &$4$& Fig. \ref{Fig.61}  \tabularnewline\cline{2-4}
&$(6,10)$ &$7$& Fig. \ref{Fig.62}\tabularnewline\hline
$6$&$(7,16)$ &$5$& Fig. \ref{Fig.64}  \tabularnewline\cline{2-4}
&$(7,10)$ &$12$ & Fig. \ref{Fig.65} \tabularnewline\hline
$7$&$(5,16)$ &$12$& Fig. \ref{Fig.62}   \tabularnewline\hline
$8$&$(1,15)$ & $13$& Fig. \ref{Fig.70} \tabularnewline\cline{2-4}
&$(1,7)$ & $11$& Fig. \ref{Fig.71}  \tabularnewline\cline{2-4}
&$(1,5)$ &$2$& Fig. \ref{Fig.72}  \tabularnewline\hline
$9$&$(11,15)$ & $12$& Fig. \ref{Fig.73}  \tabularnewline\cline{2-4}
&$(11,6)$ &$14$& Fig. \ref{Fig.74}   \tabularnewline\hline
$10$&$(14,9)$ &$0$& Fig. \ref{Fig.76}  \tabularnewline\cline{2-4}
&$(14,10)$ &$13$& Fig. \ref{Fig.77}   \tabularnewline\cline{2-4}
&$(14,7)$ &$8$& Fig. \ref{Fig.78}  \tabularnewline\hline
$11$&$(13,10)$ &$8$& Fig. \ref{Fig.85}  \tabularnewline\cline{2-4}
&$(13,11)$ &$4$ & Fig. \ref{Fig.84}  \tabularnewline\hline
\end{tabular}
}
\captionof{table}{\scriptsize{List of the remaining cases at each step and the corresponding deformations.}}
\begin{figure}[H]
\centering
	\includegraphics[scale=0.113]{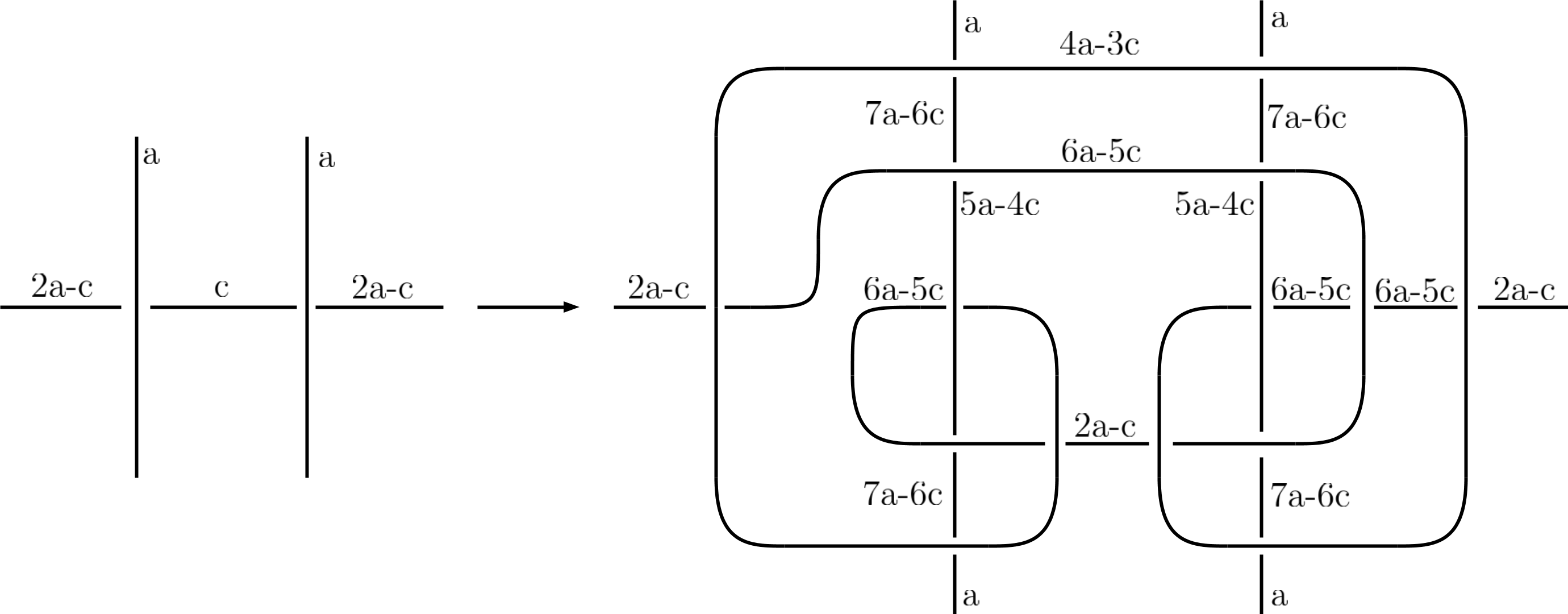} 
	\caption{\label{Fig.112}}
\end{figure}

\begin{figure}[H]
\centering
	\includegraphics[scale=0.113]{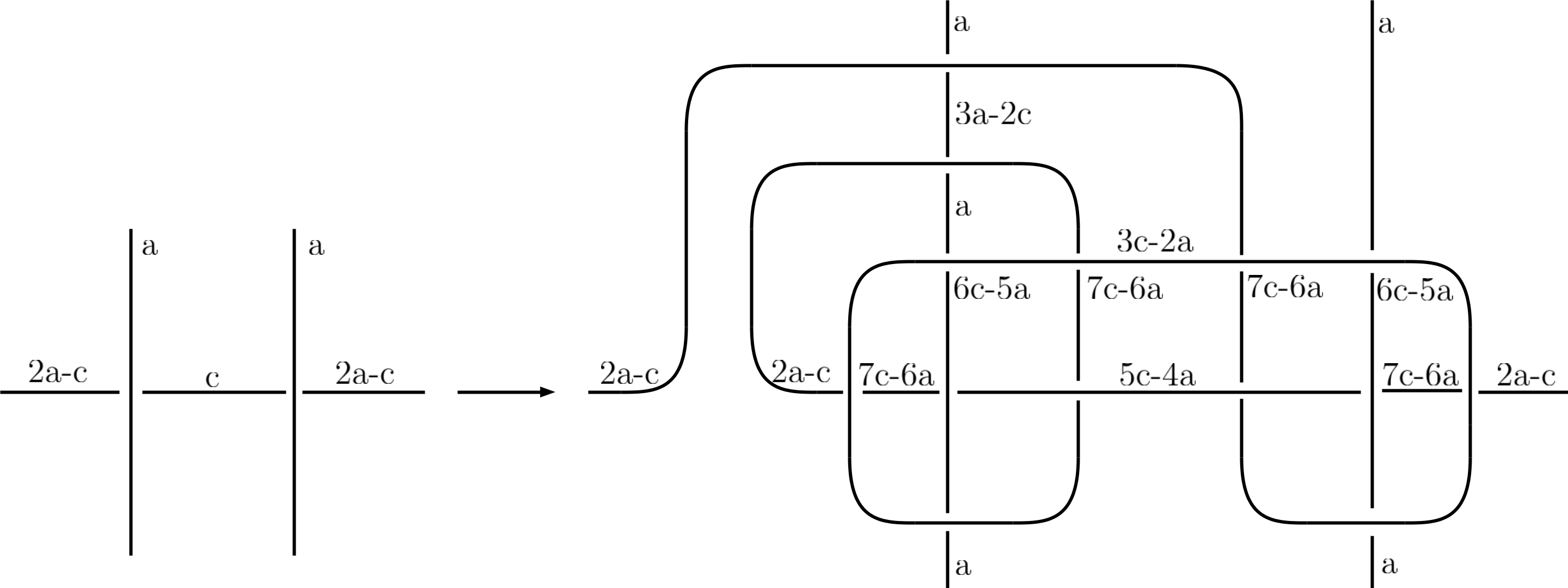} 
	\caption{\label{Fig.113}}
\end{figure}

\begin{figure}[H]
\centering
	\includegraphics[scale=0.12]{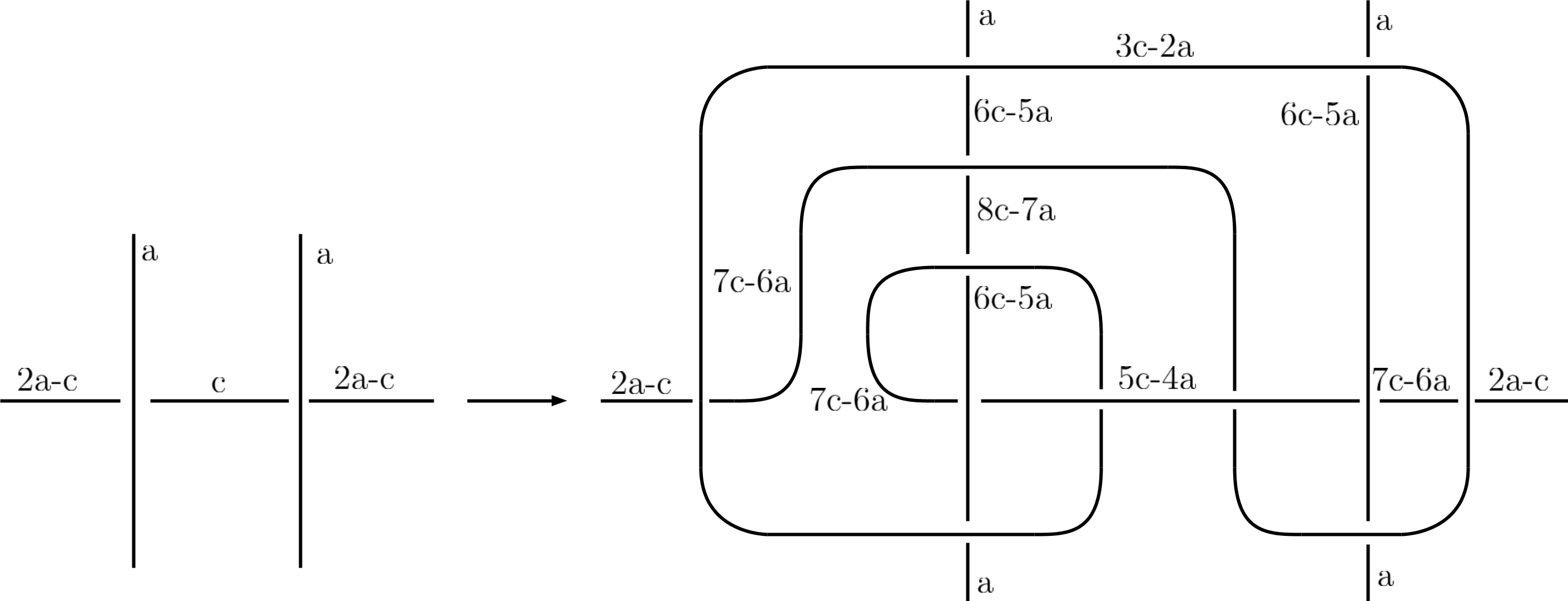} 
	\caption{\label{Fig.114}}
\end{figure}

\begin{figure}[H]
\centering
	\includegraphics[scale=0.12]{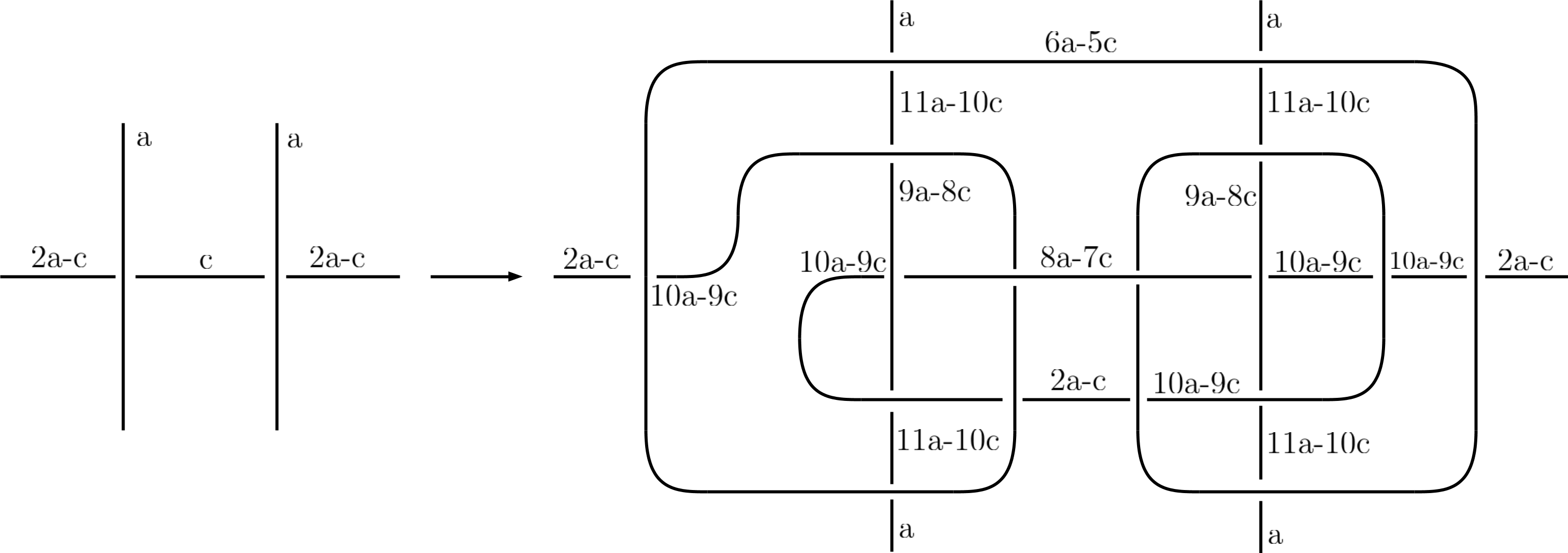} 
	\caption{\label{Fig.61}}
\end{figure}

\begin{figure}[H]
\centering
	\includegraphics[scale=0.12]{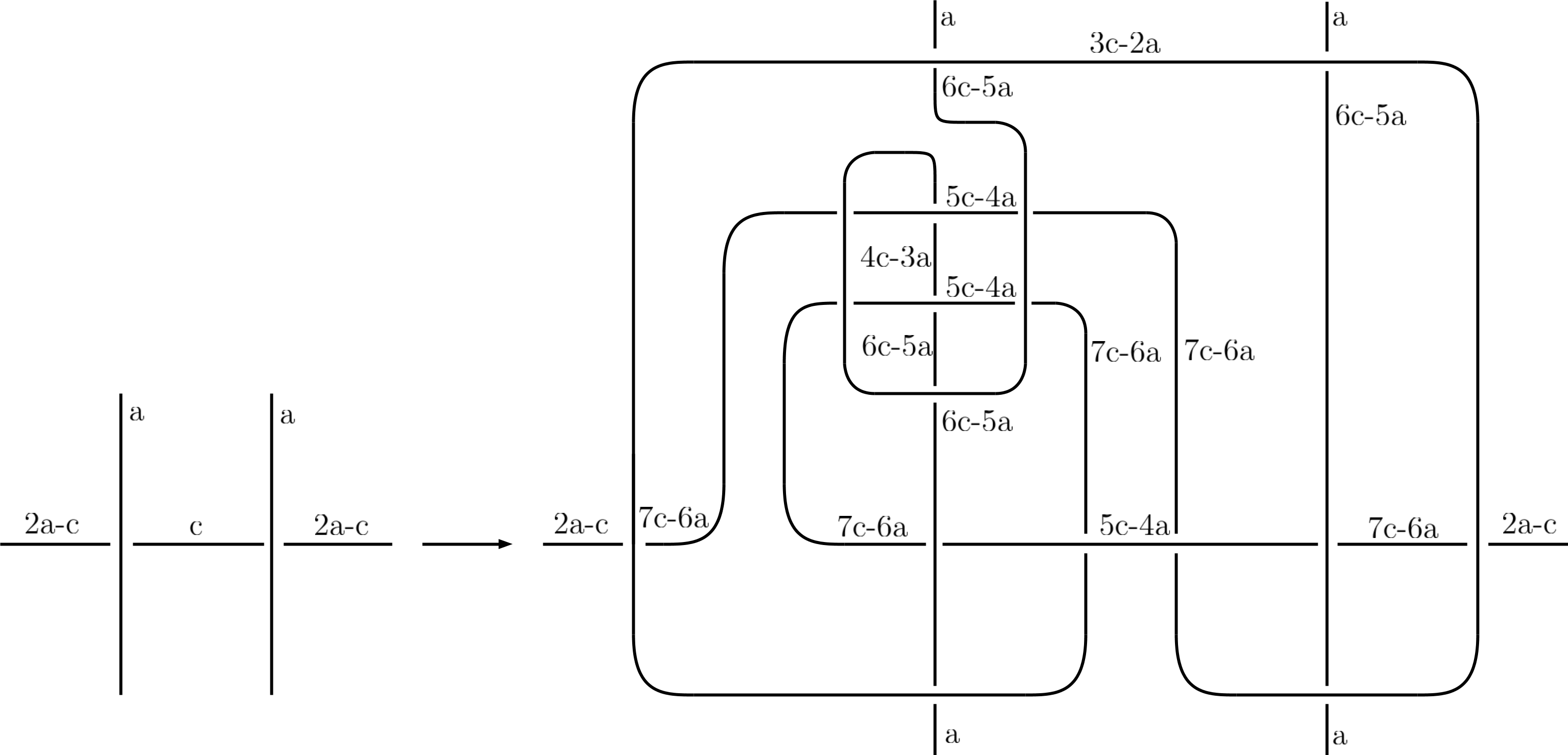} 
	\caption{\label{Fig.62}}
\end{figure}

\begin{figure}[H]
\centering
	\includegraphics[scale=0.098]{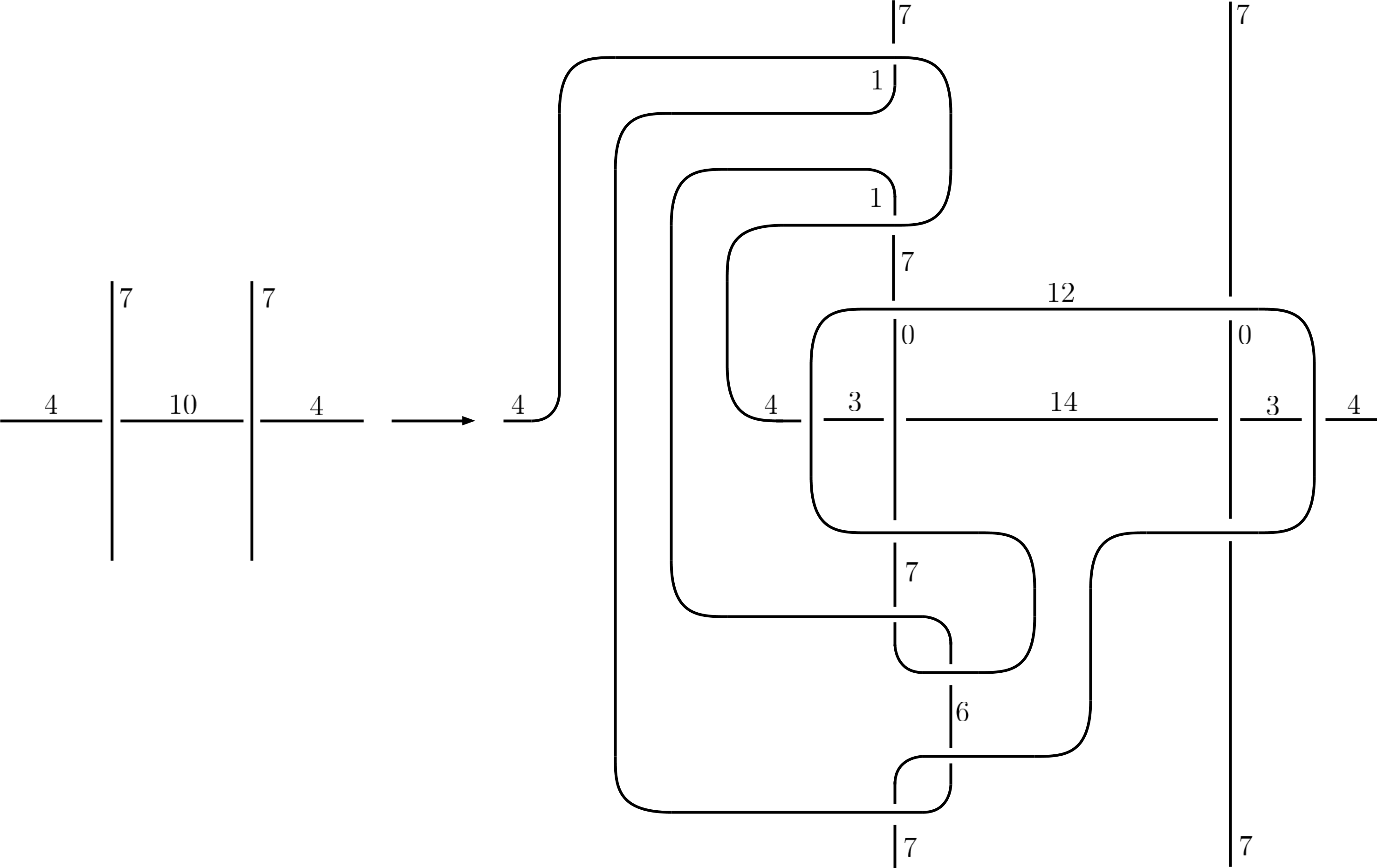} 
	\caption{\label{Fig.63}}
\end{figure}

\begin{figure}[H]
\centering
	\includegraphics[scale=0.1]{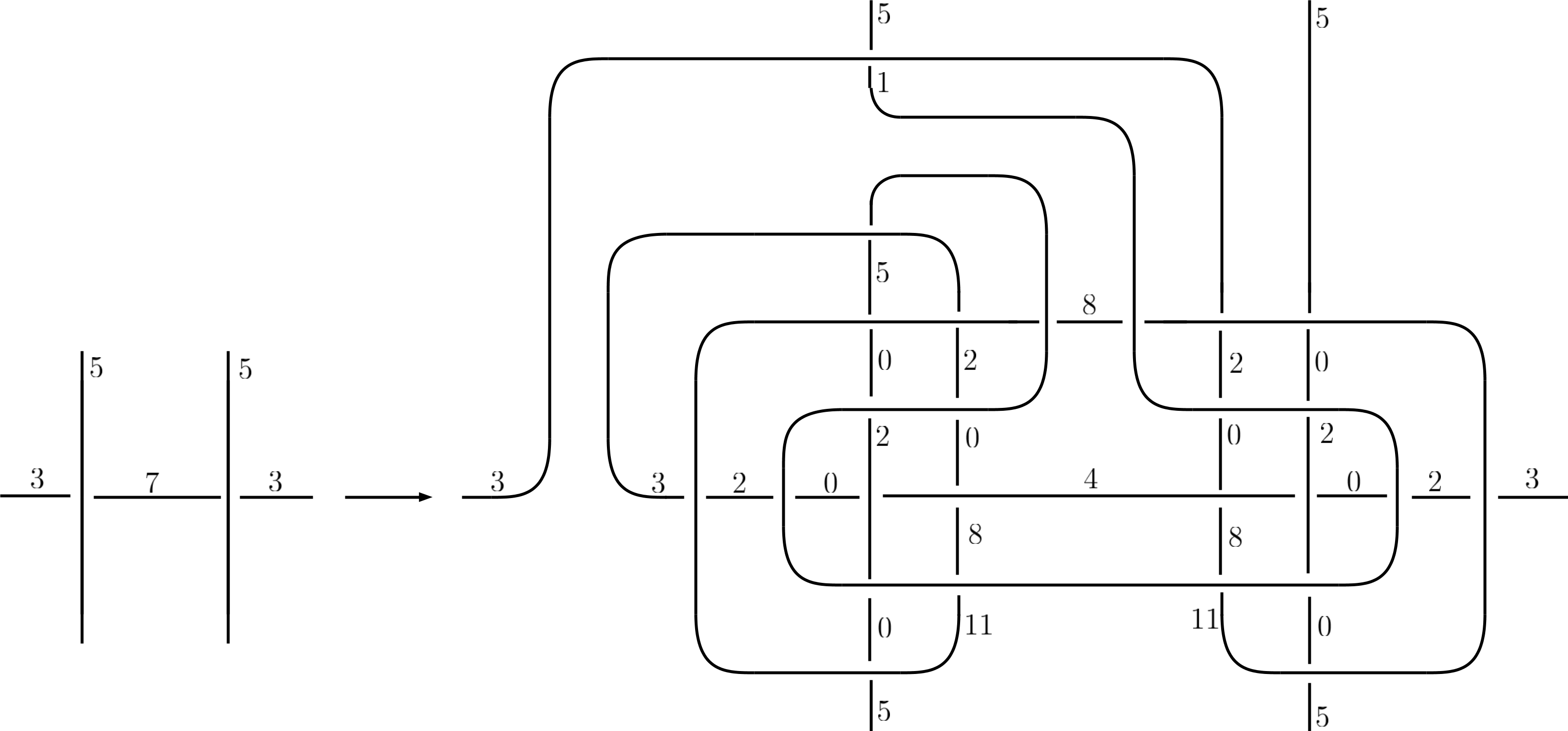} 
	\caption{\label{Fig.64}}
\end{figure}

\begin{figure}[H]
\centering
	\includegraphics[scale=0.1]{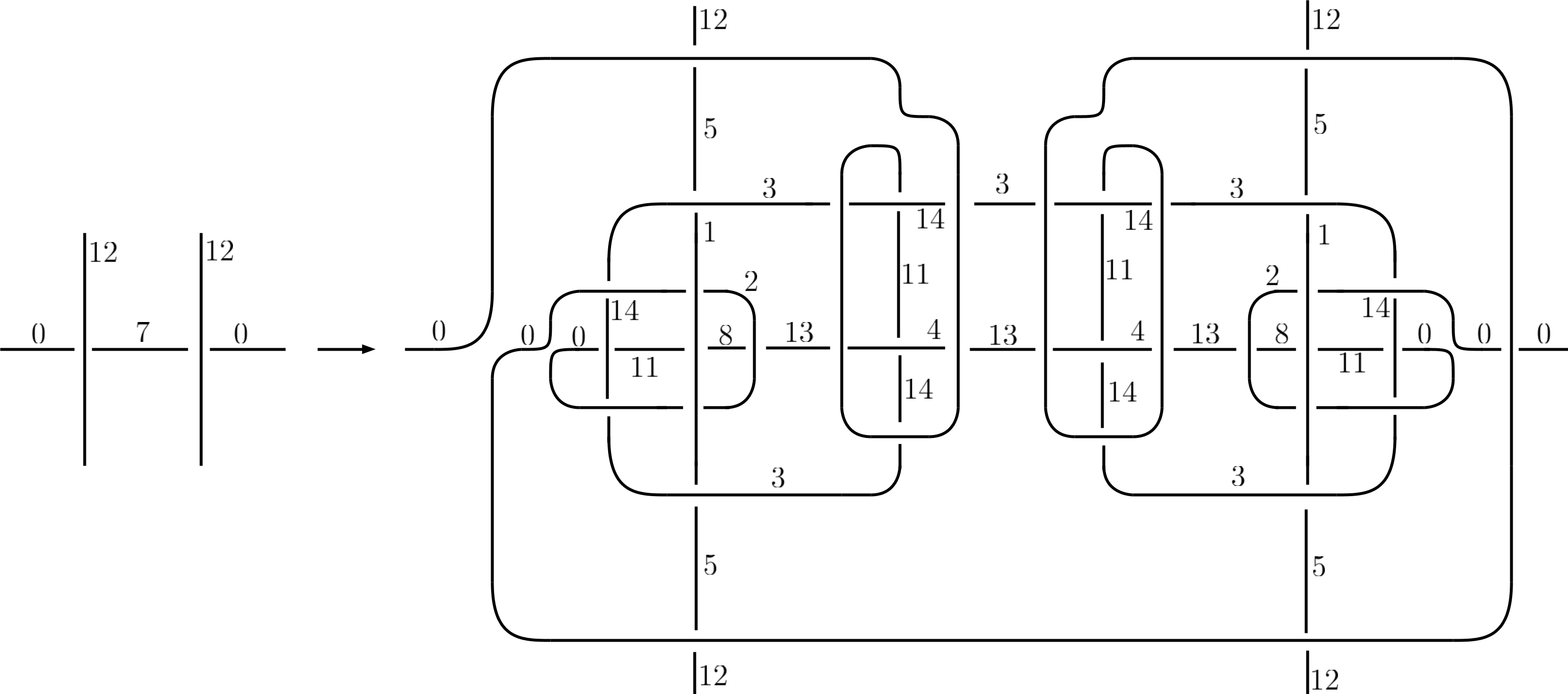} 
	\caption{\label{Fig.65}}
\end{figure}

\begin{figure}[H]
\centering
	\includegraphics[scale=0.1]{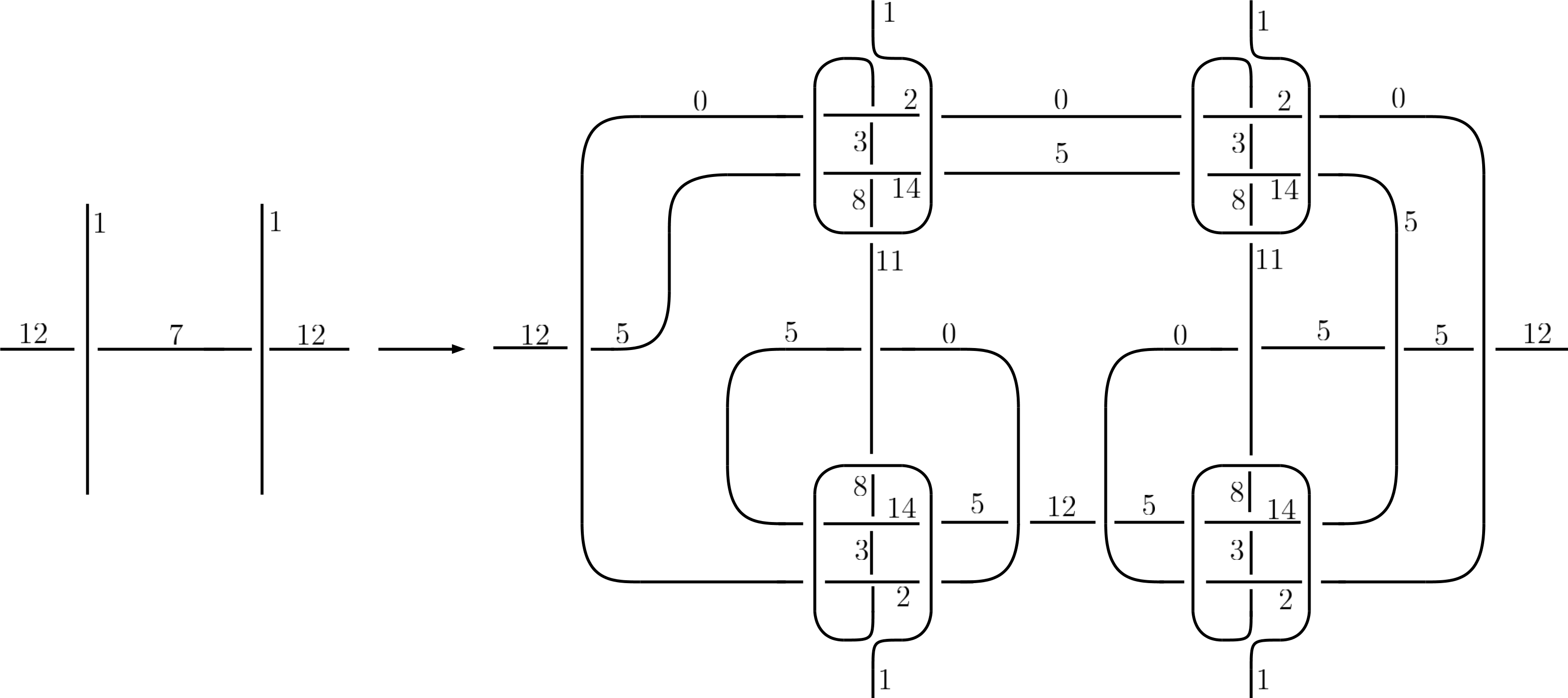} 
	\caption{\label{Fig.66}}
\end{figure}
\begin{figure}[H]
\centering
	\includegraphics[scale=0.1]{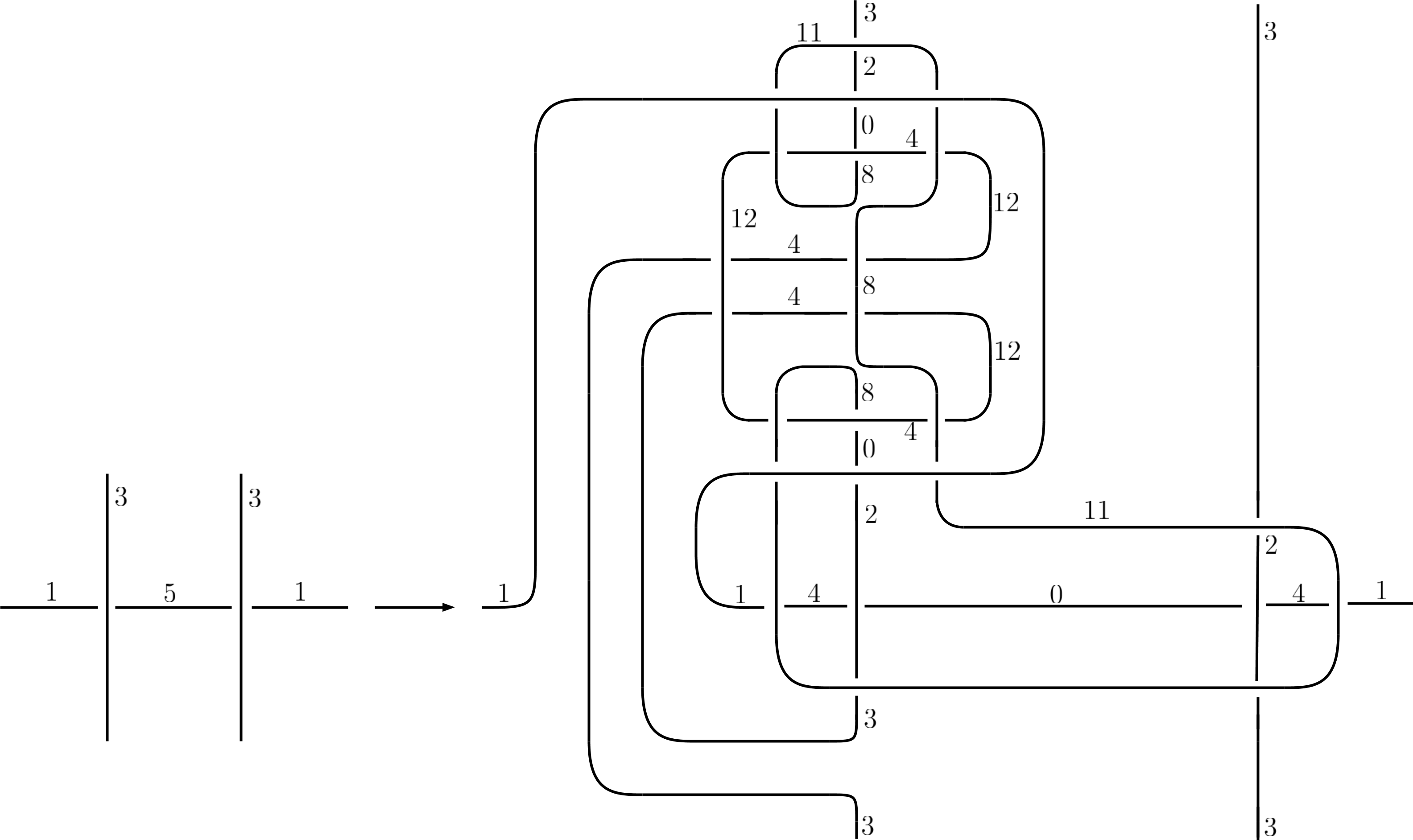} 
	\caption{\label{Fig.67}}
\end{figure}

\begin{figure}[H]
\centering
	\includegraphics[scale=0.104]{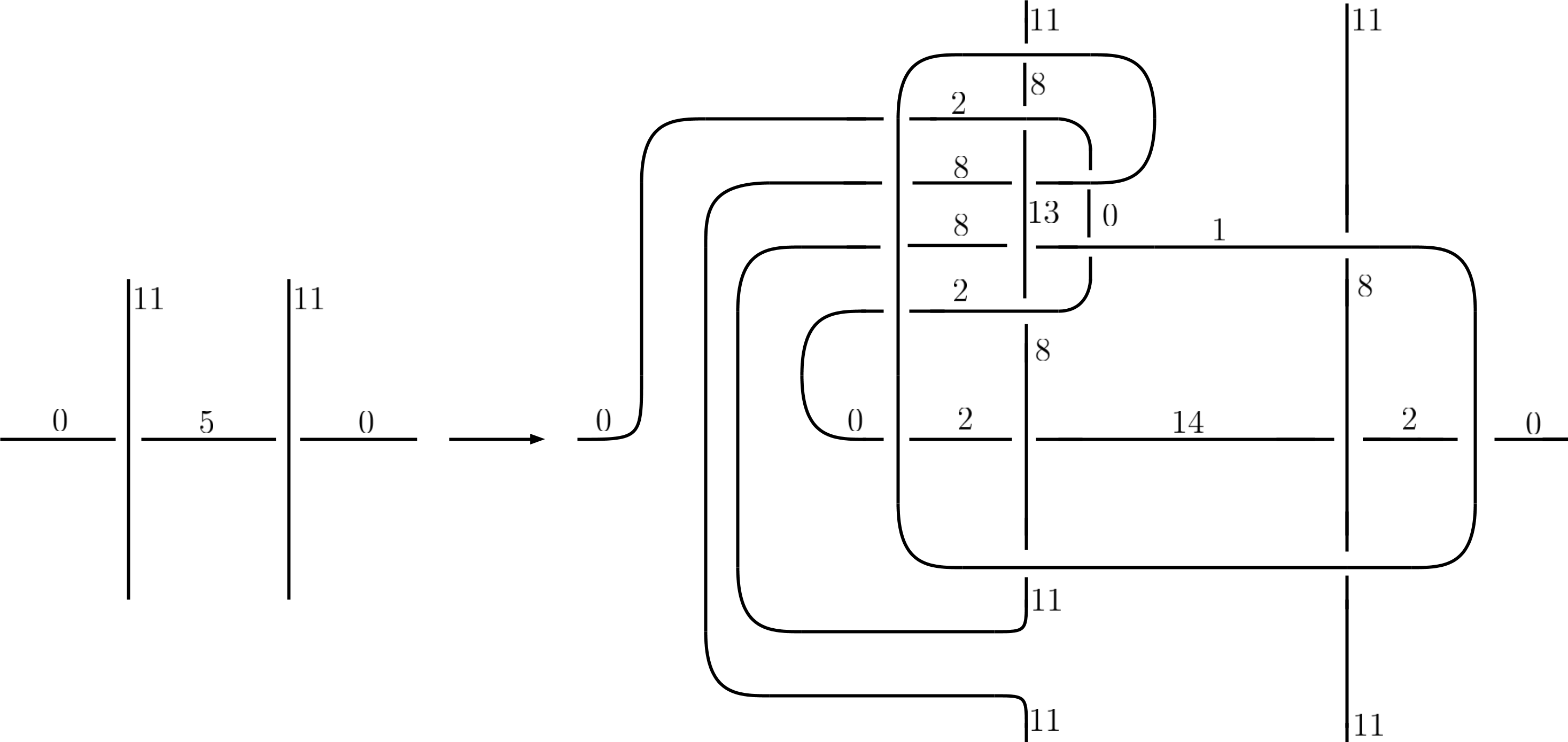} 
	\caption{\label{Fig.68}}
\end{figure}

\begin{figure}[H]
\centering
	\includegraphics[scale=0.09]{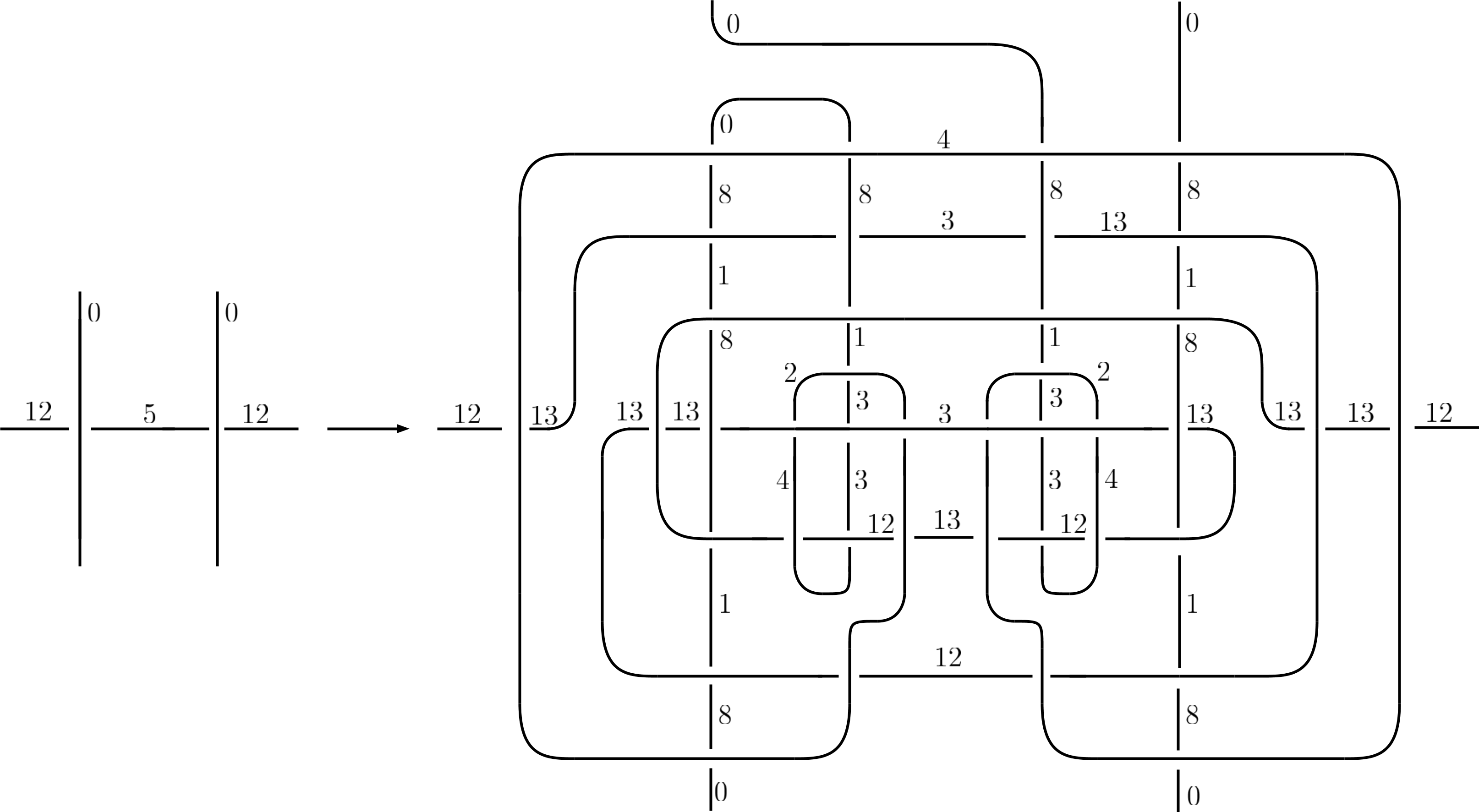} 
	\caption{\label{Fig.69}}
\end{figure}

\begin{figure}[H]
\centering
	\includegraphics[scale=0.09]{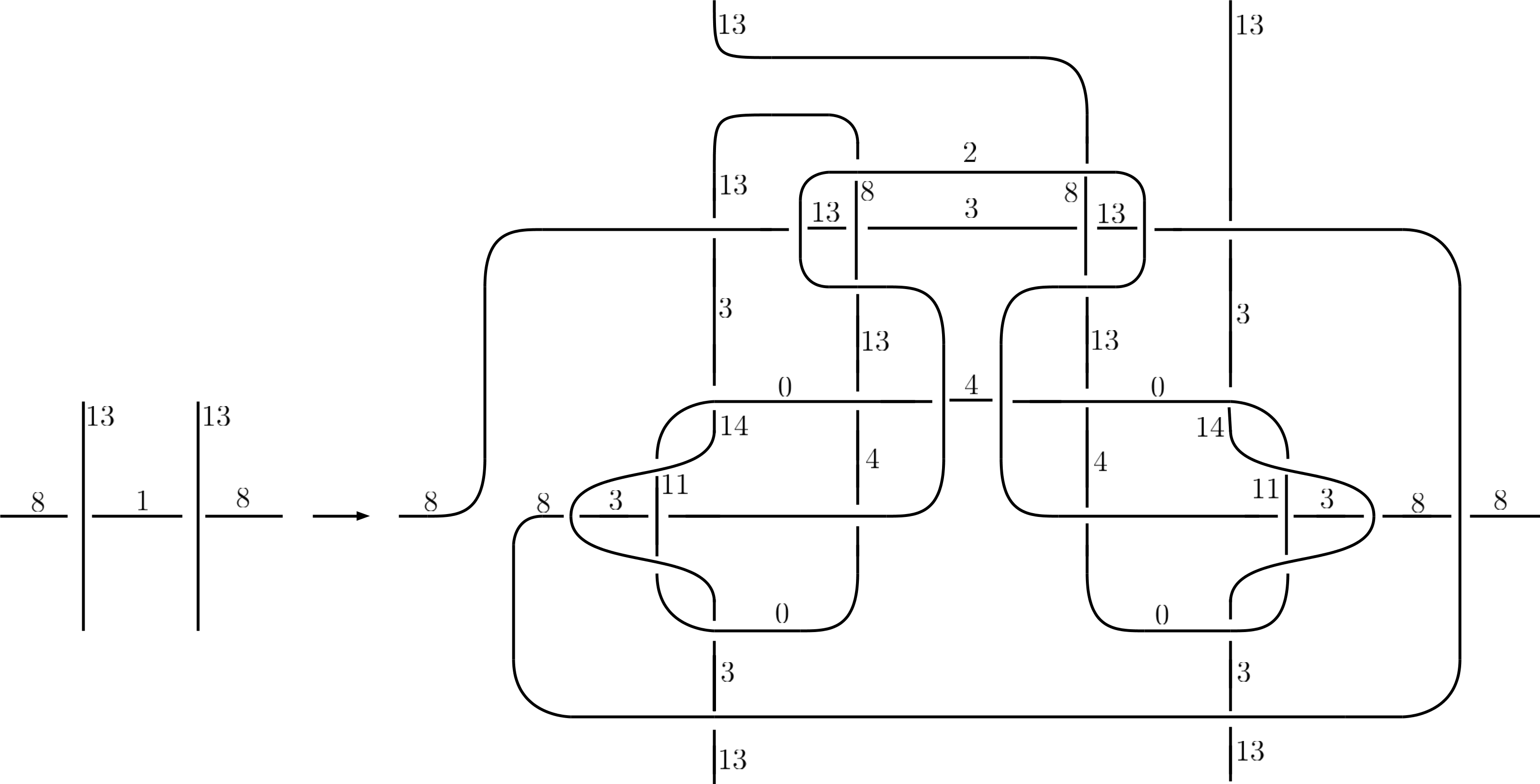} 
	\caption{\label{Fig.70}}
\end{figure}

\begin{figure}[H]
\centering
	\includegraphics[scale=0.095]{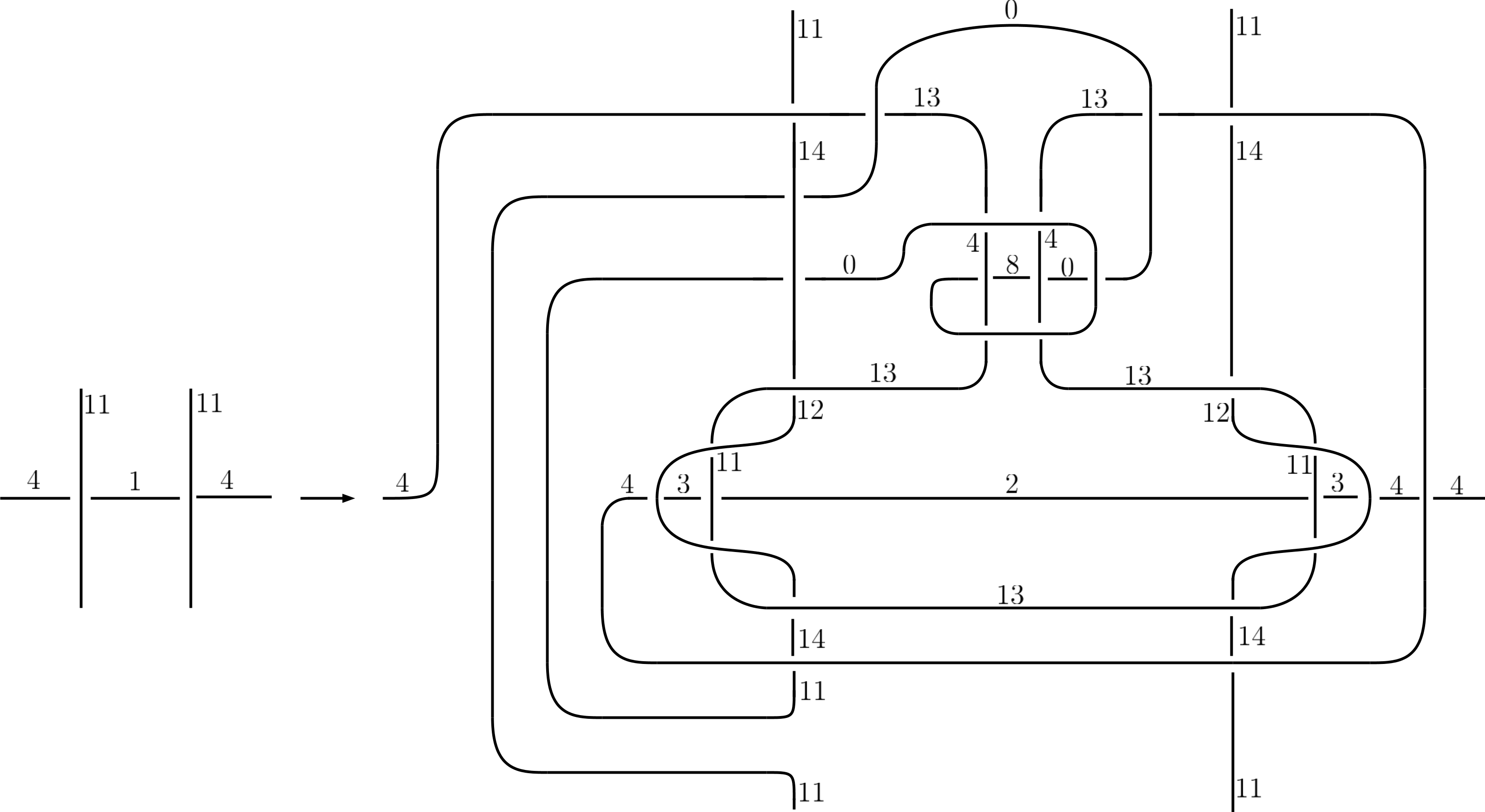} 
	\caption{\label{Fig.71}}
\end{figure}

\begin{figure}[H]
\centering
	\includegraphics[scale=0.1]{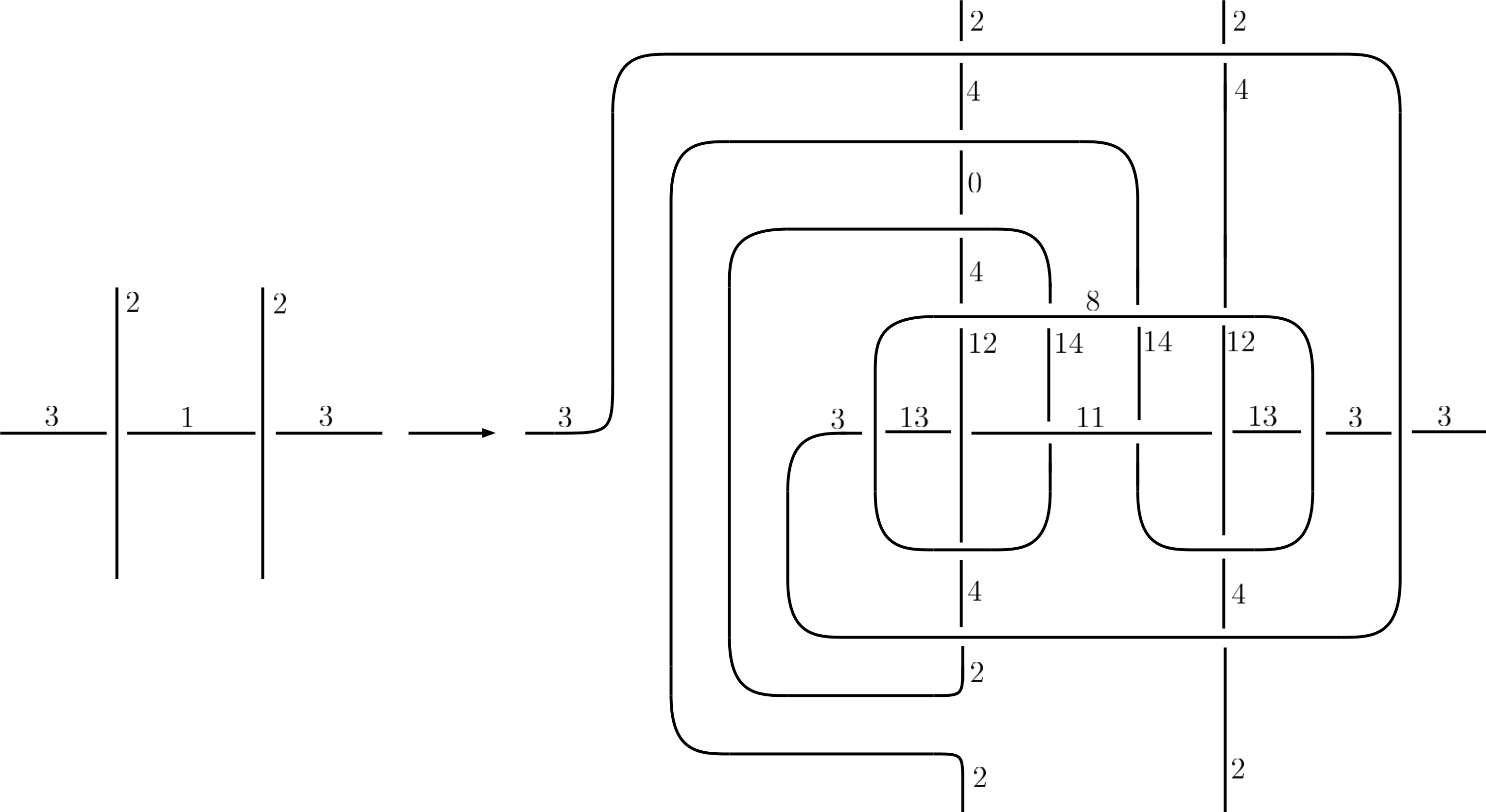} 
	\caption{\label{Fig.72}}
\end{figure}

\begin{figure}[H]
\centering
	\includegraphics[scale=0.104]{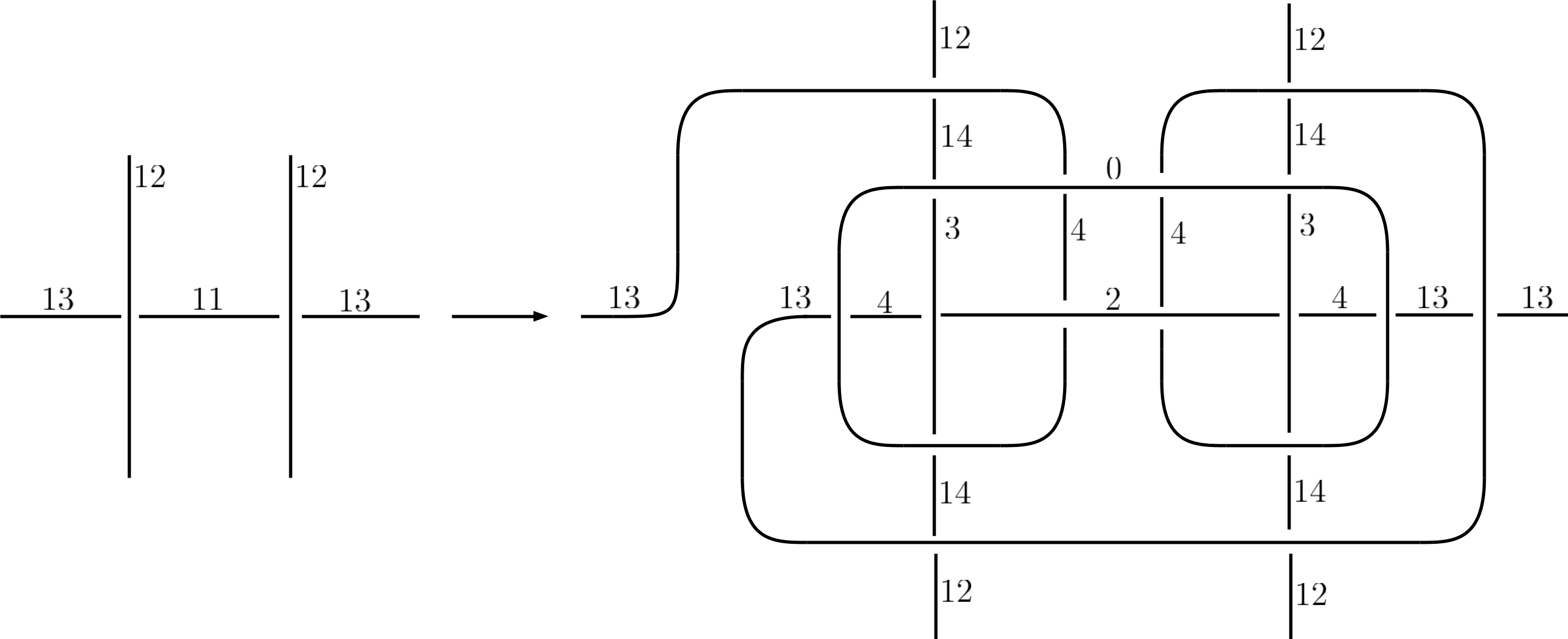} 
	\caption{\label{Fig.73}}
\end{figure}

\begin{figure}[H]
\centering
	\includegraphics[scale=0.1]{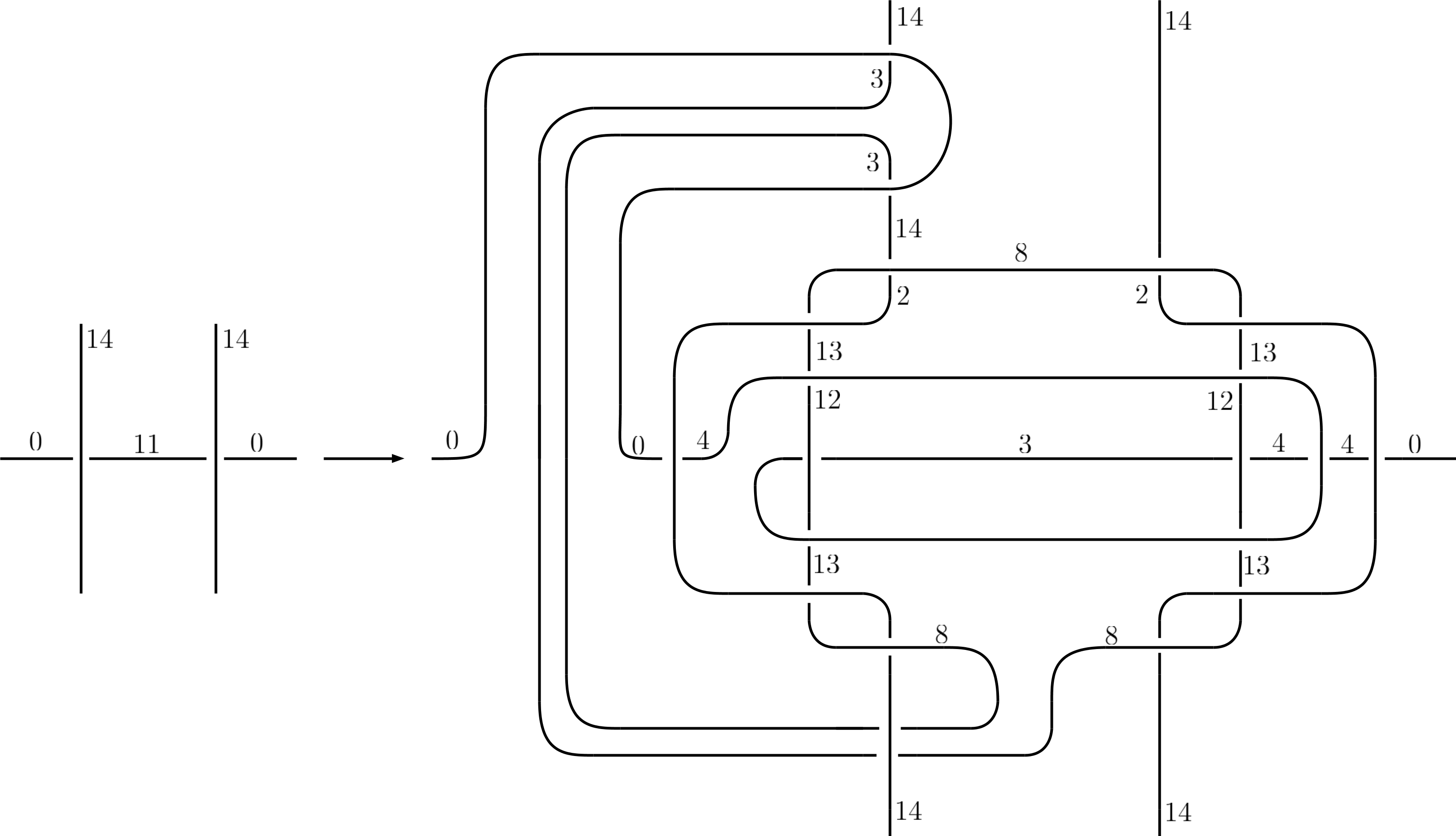} 
	\caption{\label{Fig.74}}
\end{figure}

\begin{figure}[H]
\centering
	\includegraphics[scale=0.081]{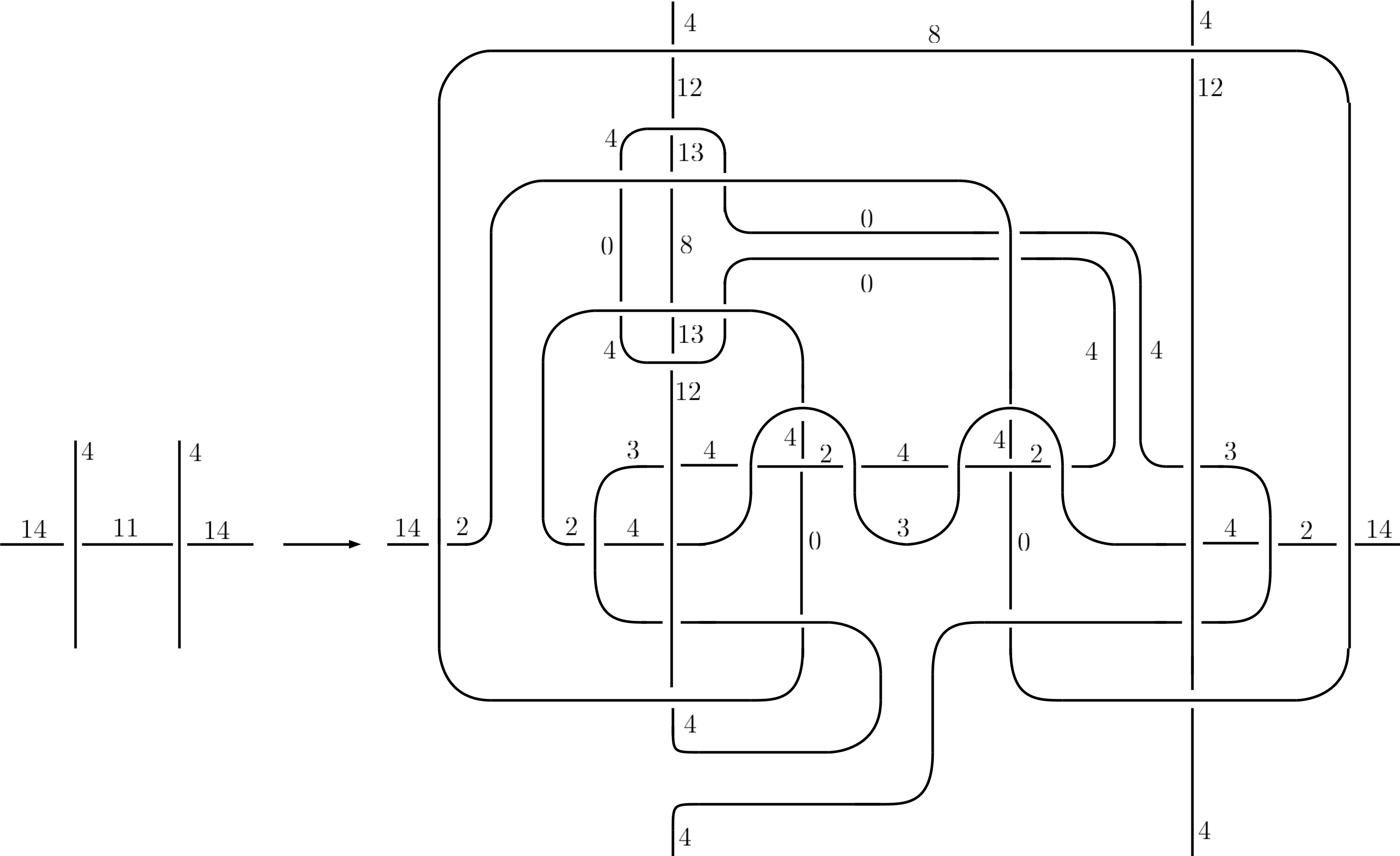} 
	\caption{\label{Fig.75}}
\end{figure}

\begin{figure}[H]
\centering
	\includegraphics[scale=0.081]{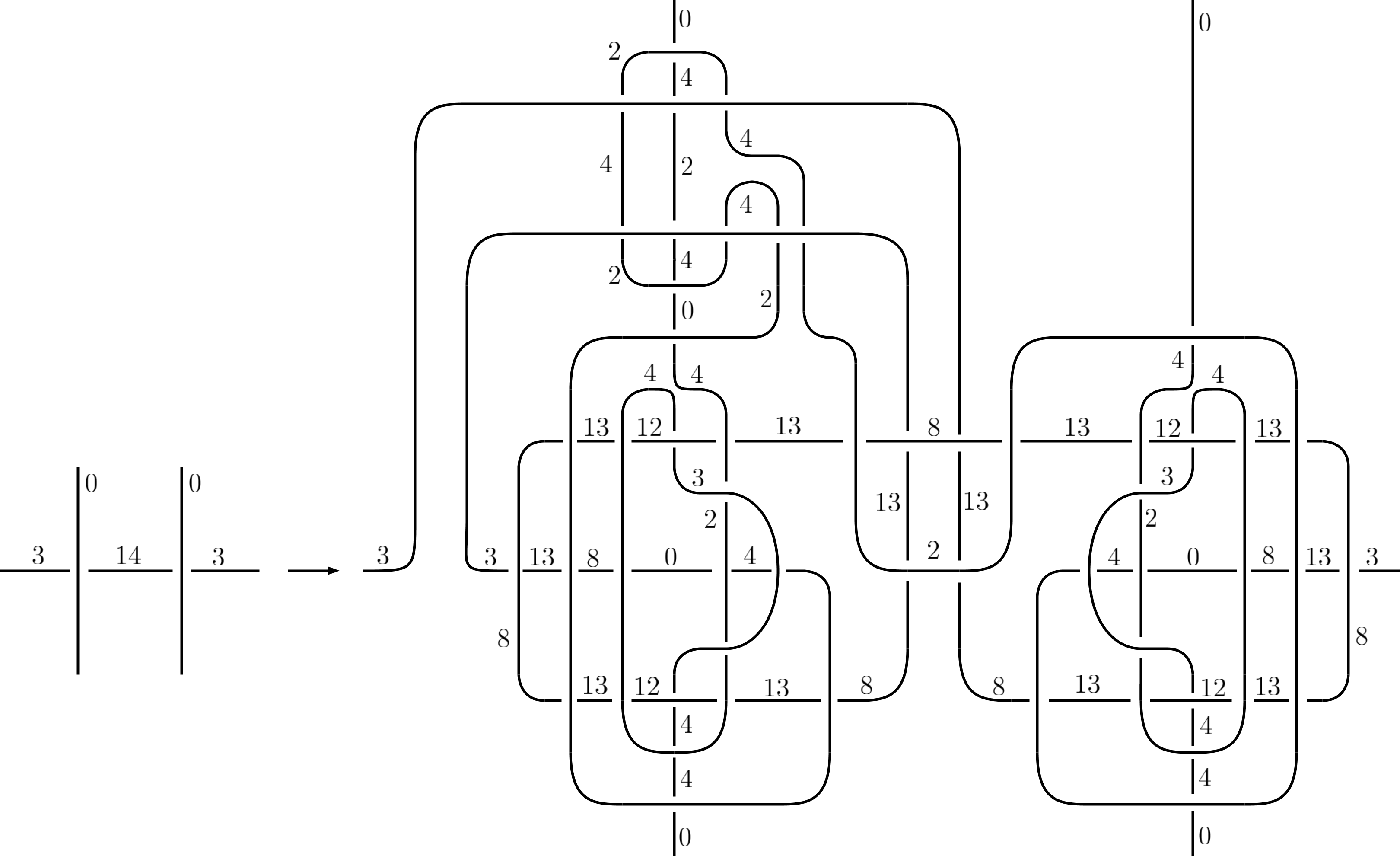} 
	\caption{\label{Fig.76}}
\end{figure}

\begin{figure}[H]
\centering
	\includegraphics[scale=0.081]{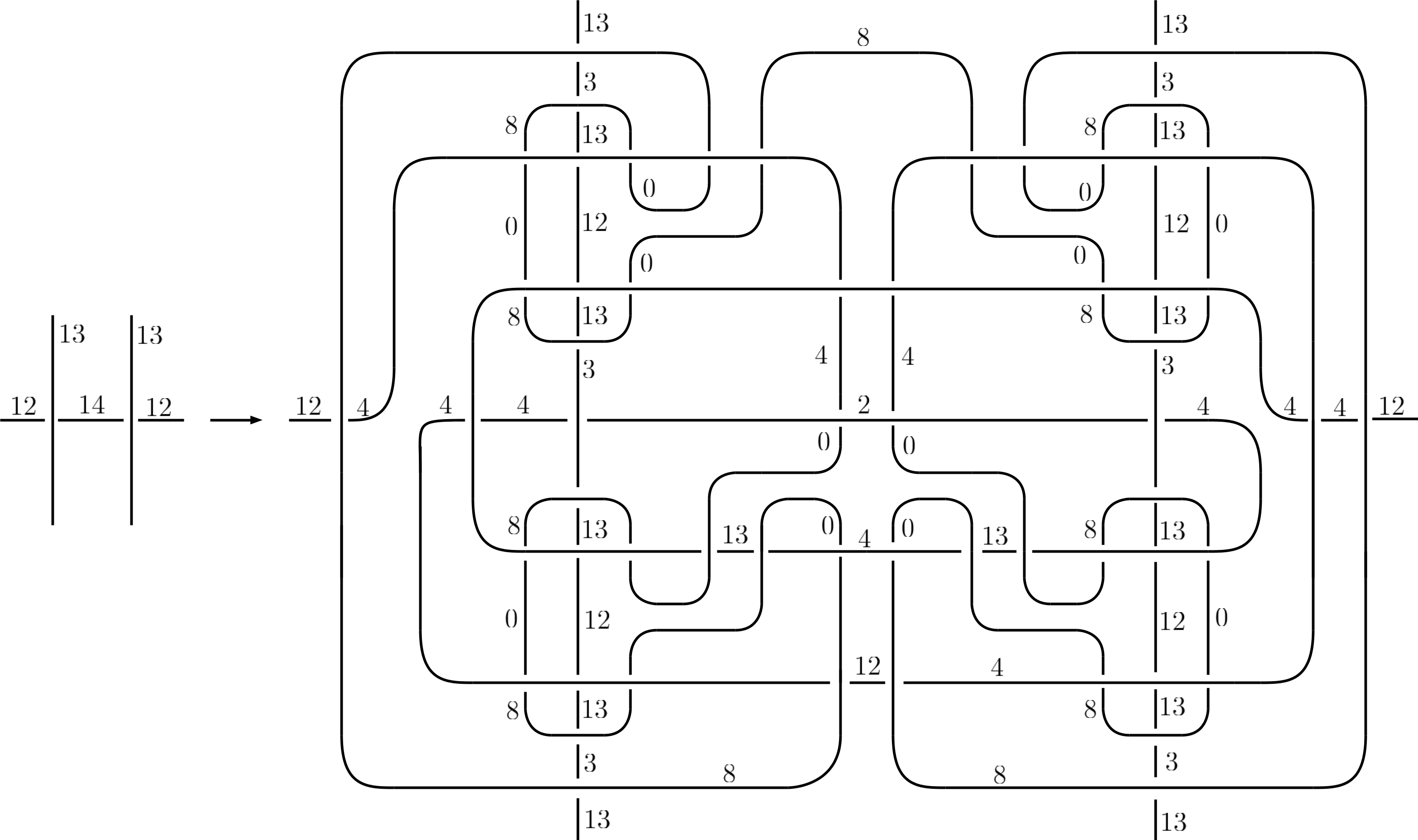} 
	\caption{\label{Fig.77}}
\end{figure}

\begin{figure}[H]
\centering
	\includegraphics[scale=0.11]{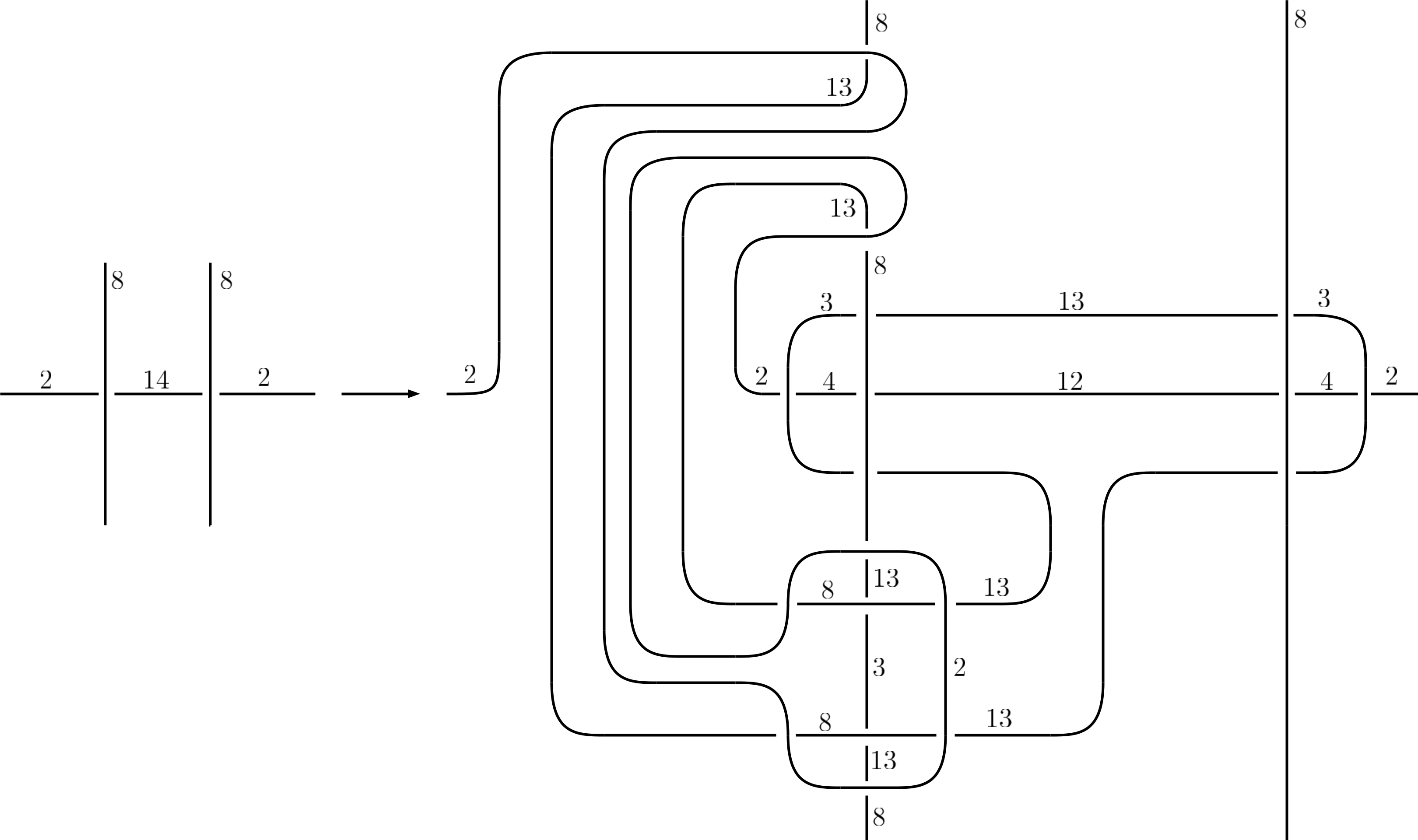} 
	\caption{\label{Fig.78}}
\end{figure}
\begin{figure}[H]
\centering
	\includegraphics[scale=0.12]{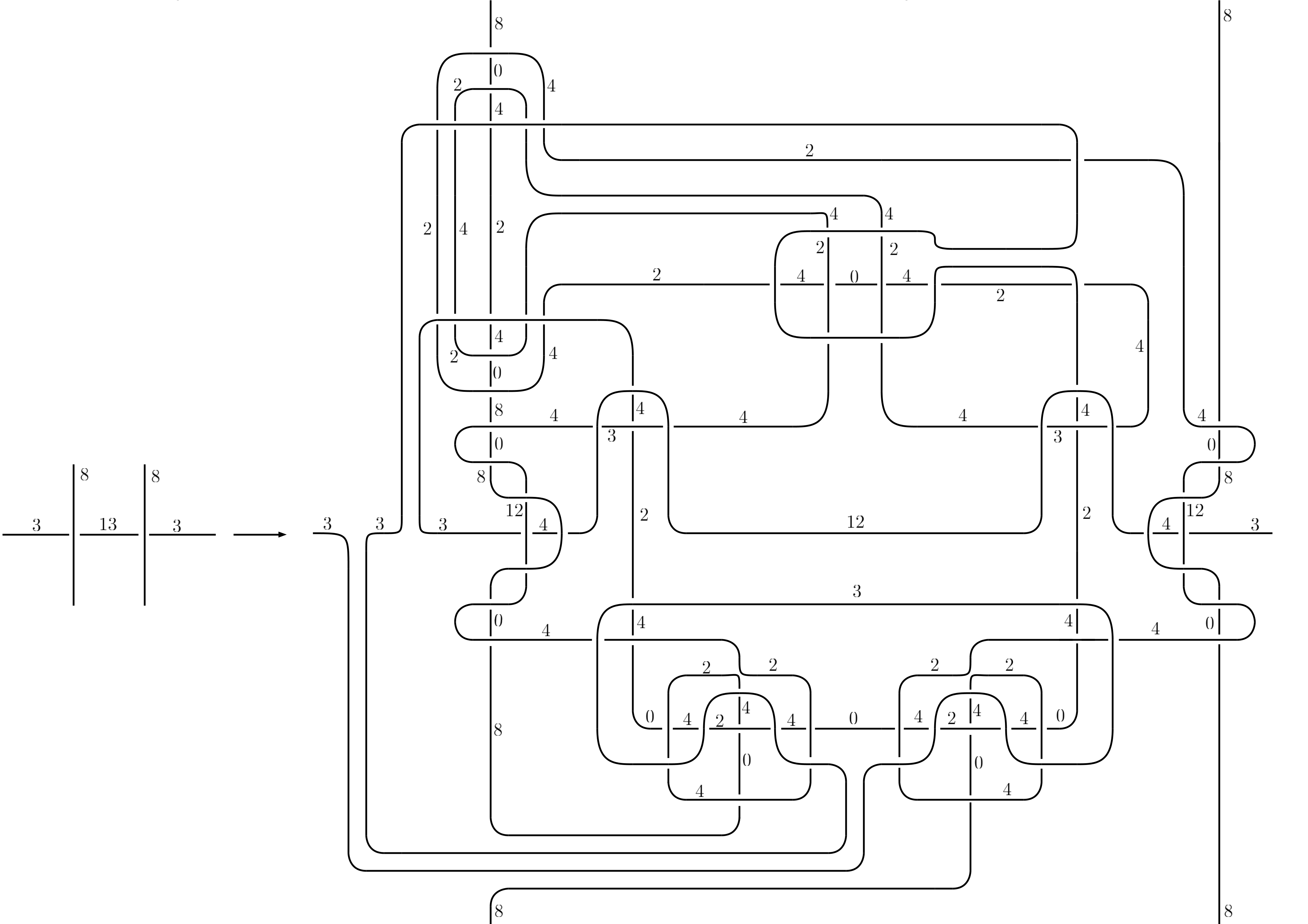} 
	\caption{\label{Fig.85}}
\end{figure}
\begin{figure}[H]
\centering
	\includegraphics[scale=0.096]{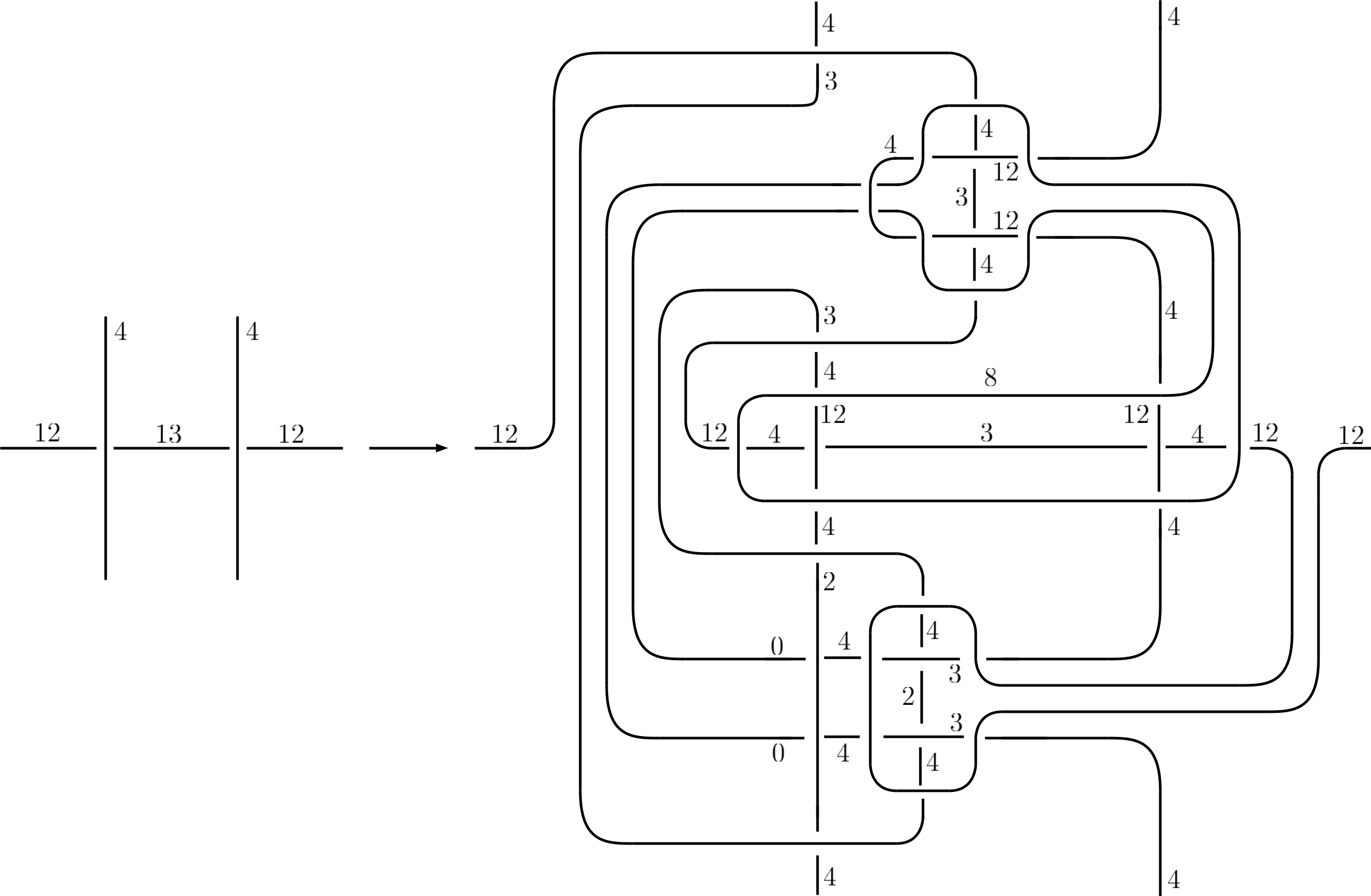} 
	\caption{\label{Fig.84}}
\end{figure}
\begin{figure}[H]
\centering
	\includegraphics[scale=0.089]{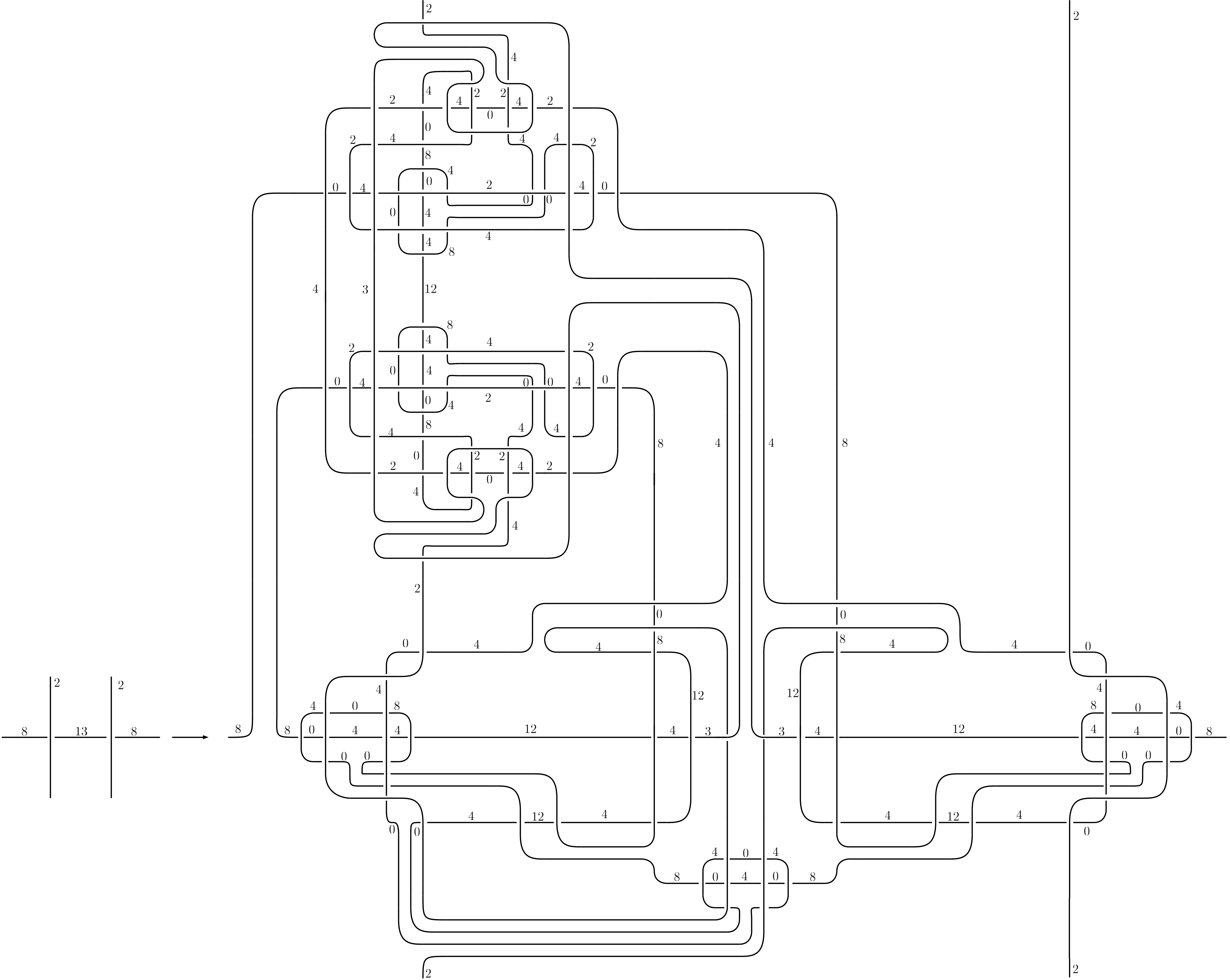} 
	\caption{\label{Fig.83}}
\end{figure}
\normalsize{
\underline{\textbf{If $a\neq b$}}, we do the deformation described in Figure~\ref{Fig.115}. We get the two new colors $2a-b$ and $2a-2b+c$. They are different from $c$ and $c_k$ iff $b\neq 2a-c$, $b\neq 2a-c_k$ and $b\ne a+9c-9c_k$, for each $k$, $1\le k\le i-1$. Then the color $c$ disappears and none of the the colors $c_k$ appears.\\
Now if $b=2a-c$ or $b=2a-c_k$ or $b=a+9c-9c_k$, then we apply to $D_{i-1}$ the transformation shown in Figure~\ref{Fig.116}. We obtain the new colors $2b-a$ and $2b-2a+c$. They are different from $c$, $c_k$ and $c_l$ where $l \neq k$ for each $l$,$k$, $1\le l,k\le i-1$,  iff $(a,b)$ is distinct from $(6c_k+12c, 12c_k+6c)$, $(6c+12c_k, 12c+6c_k)$, $(10c_k+8c, 2c_k-c)$, $(2c_k-c, 10c_k+8c)$, $(6c_l+12c_k,12c_l+6c_k)$, $(c_l+c_k-c,c_l+9c_k+8c)$ and $(9c_l+c_k+8c,c_l+c_k-c)$. when $(a,b)$ is one of those pairs, we will apply to the diagram $D_{i-1}$ different deformations which will be indicated in the following tables. Finally we get a diagram $D_i$ equivalent to $D_{i-1}$ in which no arc has the color $c_i$.\\
We remark that in all those cases, the colors $a$ and $b$ play symmetric roles. Then the adequate figures are similar. In such cases, we fill just one box in the table and the other is left blank. For example, in the first table, when $(c,c_k)=(15,16)$, we get $(a,b)=(4,10)$ and $(a,b)=(10,4)$. The deformation in Figure~\ref{Fig.145} allows to resolve the problem in the two cases in a similar way.}

\begin{center}
  \scriptsize{
\begin{tabular}{|D{0.5cm}|D{1cm}|D{3.3cm}|D{1.8cm}|D{3.5cm}|D{1.8cm}|}\hline
Step&$(c,c_k)$ & $(a,b)=(6c_k+12c, 12 c_k+6c)$& Required deformation & $(a,b)=(6c+12c_k, 12c+6c_k)$ & Required deformation \tabularnewline\hline
$2$&$(15,16)$ & $(4,10)$ && $(10,4)$ & Fig. \ref{Fig.145}  \tabularnewline\hline
$3$&$(9,16)$ & $(0,8)$ && $(8,0)$&Fig. \ref{Fig.145}  \tabularnewline\cline{2-6}
&$(9,15)$ & $(11,13)$ && $(13,11)$& Fig. \ref{Fig.145}  \tabularnewline\hline
$4$&$(10,16)$ & $(12,14)$ && $(14,12)$&Fig. \ref{Fig.86}  \tabularnewline\cline{2-6}
&$(10,15)$ & $(6,2)$ && $(2,6)$ &Fig. \ref{Fig.146} \tabularnewline\hline
$5$&$(6,10)$ & $(13,3)$ && $(3,13)$& Fig. \ref{Fig.145} \tabularnewline\hline
$6$&$(7,15)$ & $(4,1)$ &Fig. \ref{Fig.148}& $(1,4)$ &\tabularnewline\cline{2-6}
&$(7,9)$ & $(2,14)$ && $(14,2)$&Fig. \ref{Fig.146} \tabularnewline\cline{2-6}
&$(7,6)$ & $(1,12)$ && $(12,1)$&Fig. \ref{Fig.88} \tabularnewline\hline
$7$&$(5,16)$ & $(3,1)$ && $(1,3)$&Fig. \ref{Fig.f46}  \tabularnewline\cline{2-6}
&$(5,6)$ & $(11,0)$ && $(0,11)$&Fig. \ref{Fig.93} \tabularnewline\cline{2-6}
&$(5,7)$ & $(0,12)$ && $(12,0)$ &Fig. \ref{Fig.94}\tabularnewline\hline
$9$&$(11,7)$ & $(4,14)$ && $(14,4)$&Fig. \ref{Fig.99} \tabularnewline\hline
$11$&$(13,14)$ & $(2,8)$ &Fig. \ref{Fig.103}& $(8,2)$& \tabularnewline\hline
\end{tabular}
}
\captionof{table}{\scriptsize{Table of $(a,b)=(6c_k+12c, 12 c_k+6c)$ or $(a,b)=(6c+12c_k, 12c+6c_k)$.}}
\end{center}
\begin{center}
\scriptsize{
\begin{tabular}{|D{0.5cm}|D{1cm}|D{3cm}|D{1.5cm}|D{3cm}|D{1.5cm}|}\hline
Step&$(c,c_k)$ & $(a,b)=(10c_k+8c, 2c_k-c)$& Required deformation & $(a,b)=(2c_k-c, 10c_k+8c)$ & Required deformation \tabularnewline\hline
$2$&$(15,16)$ & $(8,0)$ && $(0,8)$ & Fig. \ref{Fig.145}  \tabularnewline\hline
$3$&$(9,16)$ & $(11,6)$ && $(6,11)$& Fig. \ref{Fig.145}  \tabularnewline\hline
$4$&$(10,16)$ & $(2,5)$ && $(5,2)$ &Fig. \ref{Fig.147} \tabularnewline\cline{2-6}
&$(10,9)$ & $(0,8)$ && $(8,0)$& Fig. \ref{Fig.145} \tabularnewline\hline
$5$&$(6,10)$ & $(12,14)$ &Fig. \ref{Fig.145}& $(14,12)$ &\tabularnewline\hline
$6$&$(7,6)$ & $(14,5)$ &Fig. \ref{Fig.145}& $(5,14)$& \tabularnewline\hline
$7$&$(5,15)$ & $(3,8)$ && $(8,3)$&Fig. \ref{Fig.145} \tabularnewline\cline{2-6}
&$(5,9)$ & $(11,13)$ && $(13,11)$&Fig. \ref{Fig.145} \tabularnewline\hline
\end{tabular}
}
\captionof{table}{\scriptsize{Table of $(a,b)=(10c_k+8c, 2c_k-c)$ or $(a,b)=(2c_k-c, 10c_k+8c)$.}}
\end{center}
\begin{center}\scriptsize{
\begin{tabular}{|D{0.4cm}|D{1.2cm}|D{3cm}|D{1cm}||D{1.2cm}|D{3cm}|D{1cm}|}\hline
\scriptsize{Step}&$(c,c_k,c_l)$ & \scriptsize{$(a,b)=(9c_l+c_k+8c, c_l+c_k-c)$}& \scriptsize{Required deformation} & $(c,c_k,c_l)$& $(a,b)=(c_l+c_k-c,c_l+9c_k+8c)$& \scriptsize{Required deformation}  \tabularnewline\hline
$3$&$(9,16,15)$ & $(2,5)$& & $(9,15,16)$ & $(5,2)$& Fig. \ref{Fig.145}  \tabularnewline\cline{2-7}
&$(9,15,16)$ & $(10,5)$ &Fig. \ref{Fig.145}& $(9,16,15)$ & $(5,10)$&  \tabularnewline\hline
$4$&$(10,9,15)$ & $(3,14)$ && $(10,15,9)$ & $(14,3)$ &Fig. \ref{Fig.147} \tabularnewline\cline{2-7}
&$(10,15,9)$ & $(6,14)$ && $(10,9,15)$ & $(14,6)$ &Fig. \ref{Fig.147} \tabularnewline\hline
$5$&$(6,16,10)$ & $(1,3)$ && $(6,10,16)$ & $(3,1)$&Fig. \ref{Fig.147}  \tabularnewline\cline{2-7}
&$(6,9,15)$ & $(5,1)$ && $(6,15,9)$ &$(1,5)$& Fig. \ref{Fig.145} \tabularnewline\hline
$6$&$(7,9,16)$ & $(5,1)$&Fig. \ref{Fig.86}& $(7,16,9)$ & $(1,5)$& \tabularnewline\cline{2-7}
&$(7,10,15)$ & $(14,1)$&Fig. \ref{Fig.89}&$(7,15,10)$ & $(1,14)$ & \tabularnewline\cline{2-7}
&$(7,9,10)$ & $(2,12)$ && $(7,10,9)$ & $(12,2)$&Fig. \ref{Fig.90} \tabularnewline\hline
$7$&$(5,16,7)$ & $(0,1)$ && $(5,7,16)$ & $(1,0)$& Fig. \ref{Fig.91}\tabularnewline\cline{2-7}
&$(5,7,16)$ & $(4,1)$ & &$(5,16,7)$ & $(1,4)$& Fig. \ref{Fig.95}\tabularnewline\cline{2-7}
&$(5,16,6)$ & $(8,0)$ &Fig. \ref{Fig.92}& $(5,6,16)$ & $(0,8)$& \tabularnewline\cline{2-7}
&$(5,6,16)$ & $(3,0)$ && $(5,16,6)$ & $(0,3)$ &Fig. \ref{Fig.91}\tabularnewline\cline{2-7}
&$(5,7,10)$ & $(1,12)$ &Fig. \ref{Fig.96}& $(5,10,7)$ & $(12,1)$& \tabularnewline\cline{2-7}
&$(5,10,7)$ & $(11,12)$ && $(5,7,10)$ & $(12,11)$&Fig. \ref{Fig.97} \tabularnewline\hline
$8$&$(1,9,5)$ & $(11,13)$ && $(1,5,9)$ & $(13,11)$ &Fig. \ref{Fig.98}\tabularnewline\hline
$9$&$(11,5,9)$ & $(4,3)$ && $(11,9,5)$ & $(3,4)$& Fig. \ref{Fig.146} \tabularnewline\cline{2-7}
&$(11,9,16)$ & $(3,14)$ && $(11,16,9)$ &  $(14,3)$&Fig. \ref{Fig.147} \tabularnewline\cline{2-7}
&$(11,10,15)$ & $(12,14)$ && $(11,15,10)$ & $(14,12)$&Fig. \ref{Fig.100}\tabularnewline\hline
$10$&$(14,16,15)$ & $(8,0)$ &Fig. \ref{Fig.101}& $(14,15,16)$ &  $(0,8)$&\tabularnewline\cline{2-7}
&$(14,9,5)$ & $(13,0)$ && $(14,5,9)$ &  $(0,13)$&Fig. \ref{Fig.147}\tabularnewline\hline
$11$&$(13,16,14)$ & $(8,0)$ &&$(13,14,16)$ &  $(0,8)$& Fig. \ref{Fig.146} \tabularnewline\cline{2-7}
&$(13,6,9)$ & $(4,2)$ &Fig. \ref{Fig.102}& $(13,9,6)$ &  $(2,4)$&\tabularnewline\hline
\end{tabular}
}
\captionof{table}{\scriptsize{Table of $(a,b)=(9c_l+c_k+8c, c_l+c_k-c)$ or $(a,b)=(c_l+c_k-c,c_l+9c_k+8c)$.}}
\end{center}
\begin{center}
  \scriptsize{
\begin{tabular}{|D{0.5cm}|D{3cm}|D{3.5cm}|D{3cm}|}\hline
Step&$(c,c_k,c_l)$ & $(a,b)=(6c_l+12c_k,12c_l+6c_k)$ & Required deformation \tabularnewline\hline
$6$&$(7,10,16)$ & $(12,14)$ & \tabularnewline\cline{2-4}
&$(7,16,10)$ & $(14,12)$& Fig. \ref{Fig.145} \tabularnewline\hline
$7$&$(5,10,6)$ &  $(3,13)$& \tabularnewline\cline{2-4}
&$(5,6,10)$ &  $(13,3)$ & Fig. \ref{Fig.146}\tabularnewline\hline
\end{tabular}
}
\captionof{table}{\scriptsize{Table of $(a,b)=(6c_l+12c_k,12c_l+6c_k)$.}}
\end{center}
\begin{figure}[H]
\centering
	\includegraphics[scale=0.096]{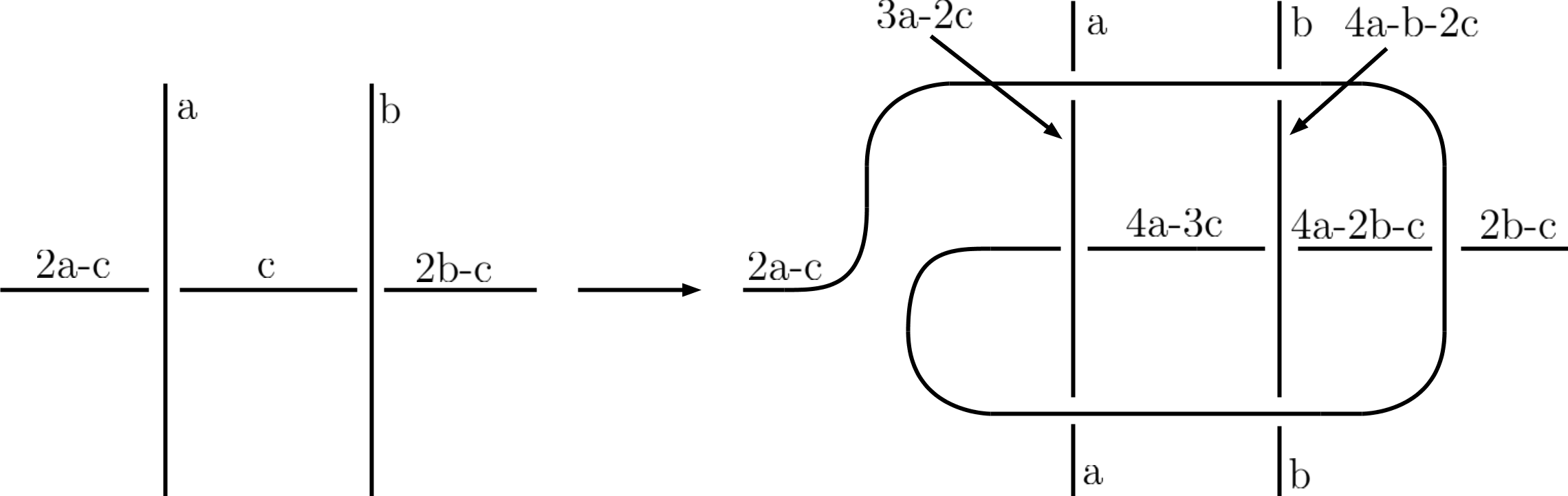} 
	\caption{\label{Fig.145}}
\end{figure}

\begin{figure}[H]
\centering
	\includegraphics[scale=0.096]{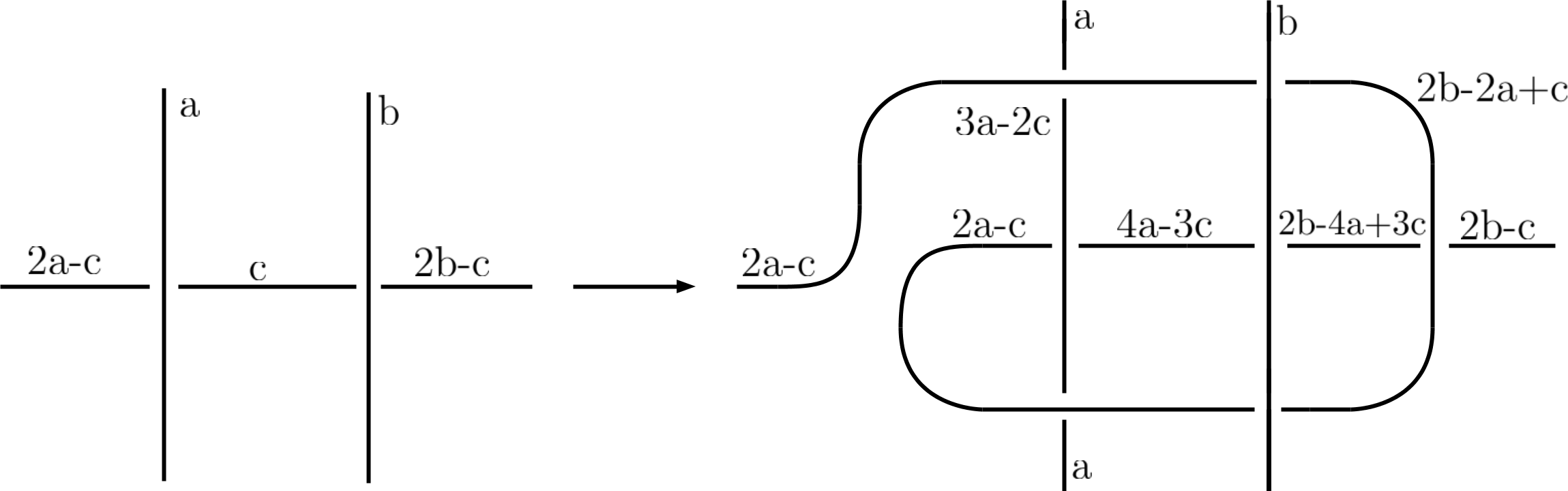} 
	\caption{\label{Fig.146}}
\end{figure}

\begin{figure}[H]
\centering
	\includegraphics[scale=0.12]{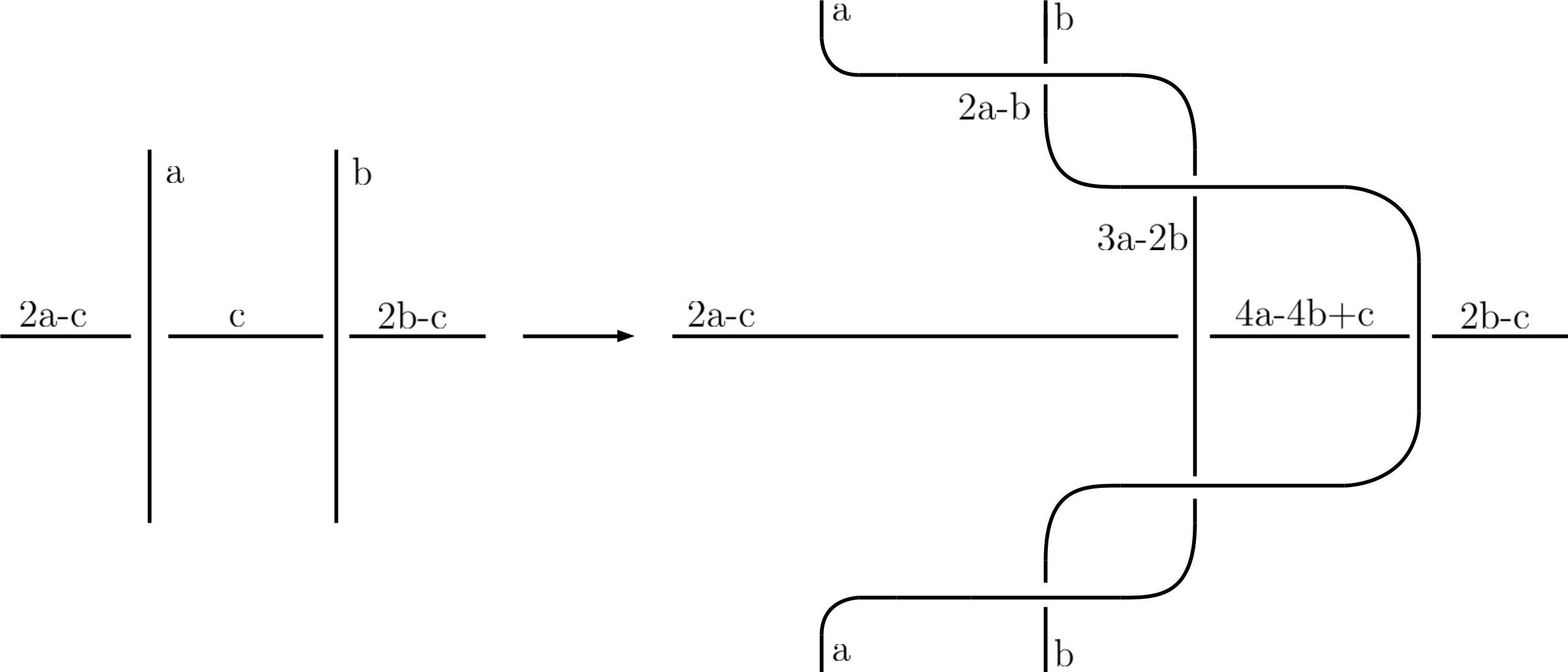} 
	\caption{\label{Fig.147}}
\end{figure}
\begin{figure}[H]
\centering
	\includegraphics[scale=0.12]{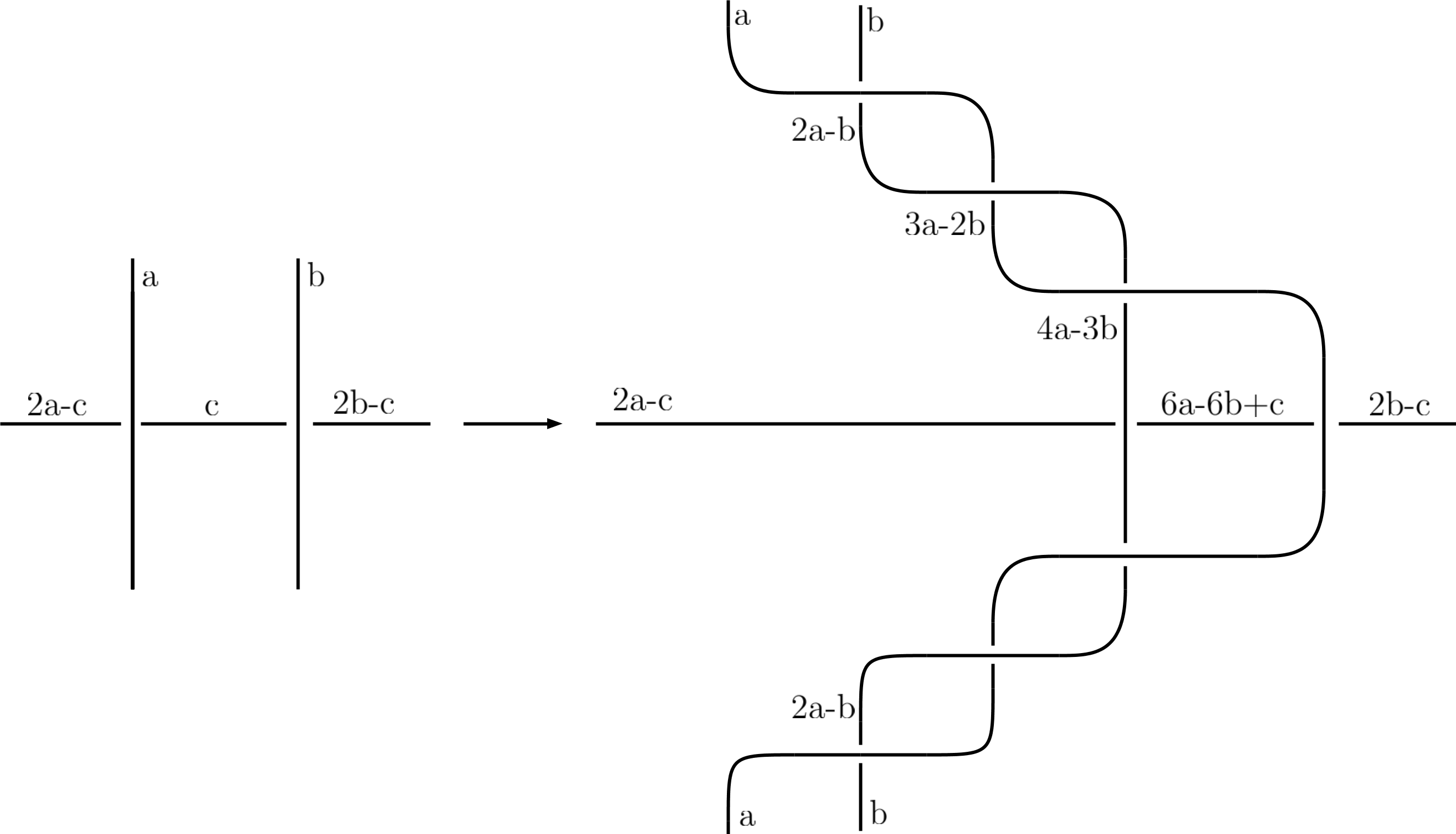} 
	\caption{\label{Fig.91}}
\end{figure}
\begin{figure}[H]
\centering
	\includegraphics[scale=0.12]{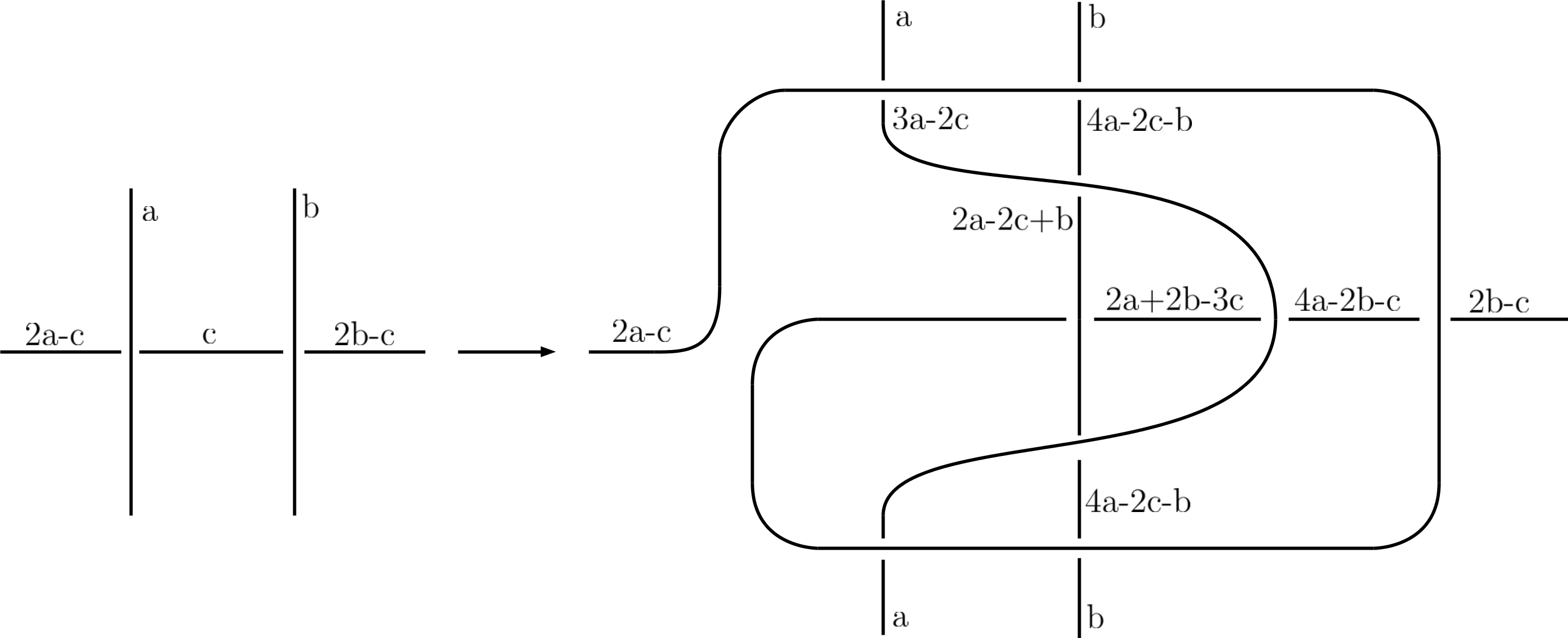} 
	\caption{\label{Fig.86}}	
\end{figure}
\begin{figure}[H]
\centering
	\includegraphics[scale=0.12]{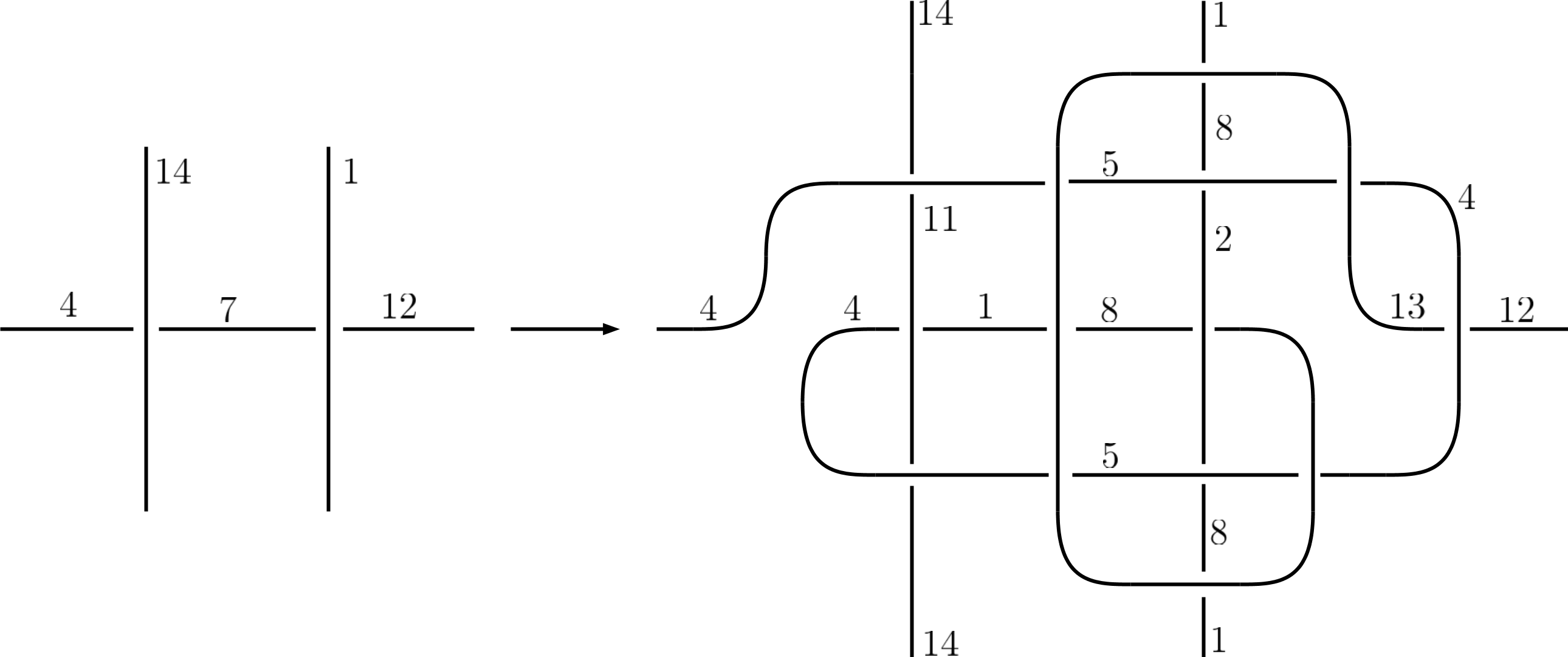} 
	\caption{\label{Fig.89}}
\end{figure}
\begin{figure}[H]
\centering
	\includegraphics[scale=0.12]{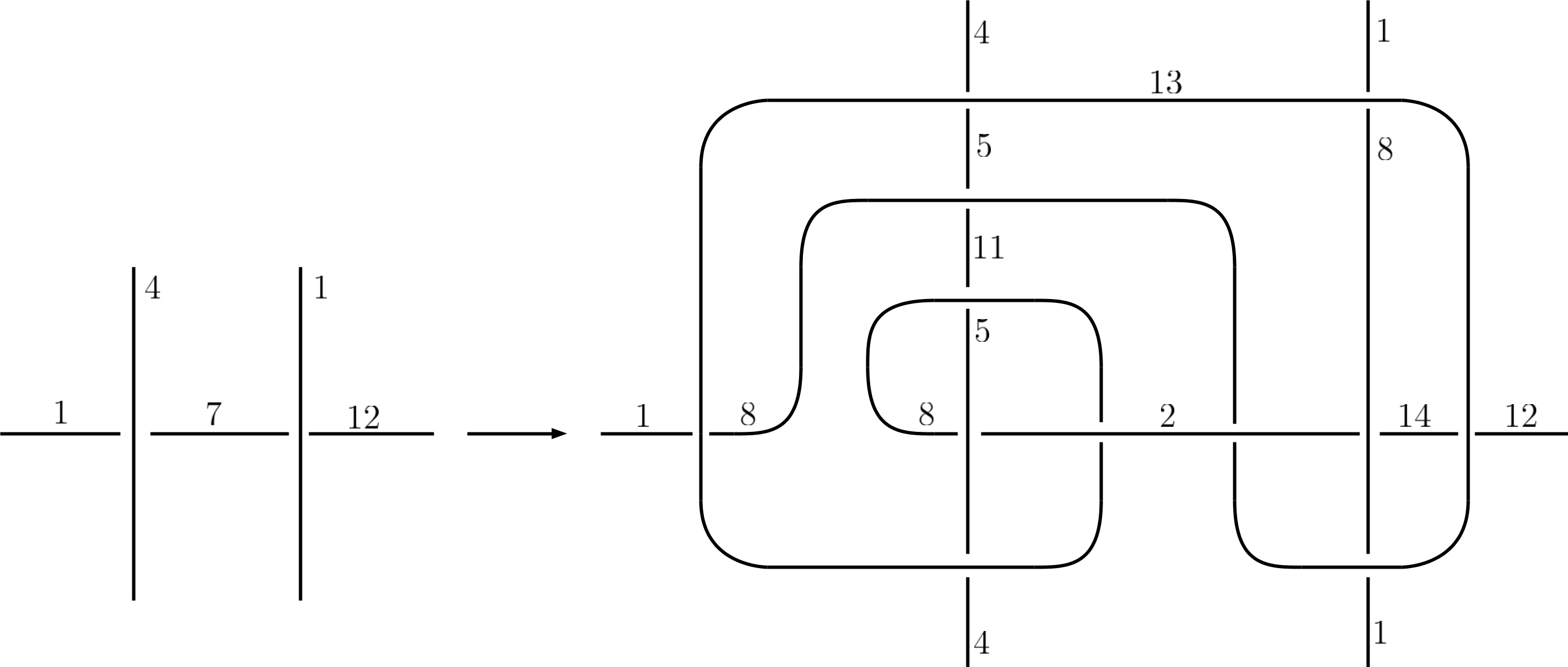} 
	\caption{\label{Fig.148}}
\end{figure}

\begin{figure}[H]
\centering
	\includegraphics[scale=0.12]{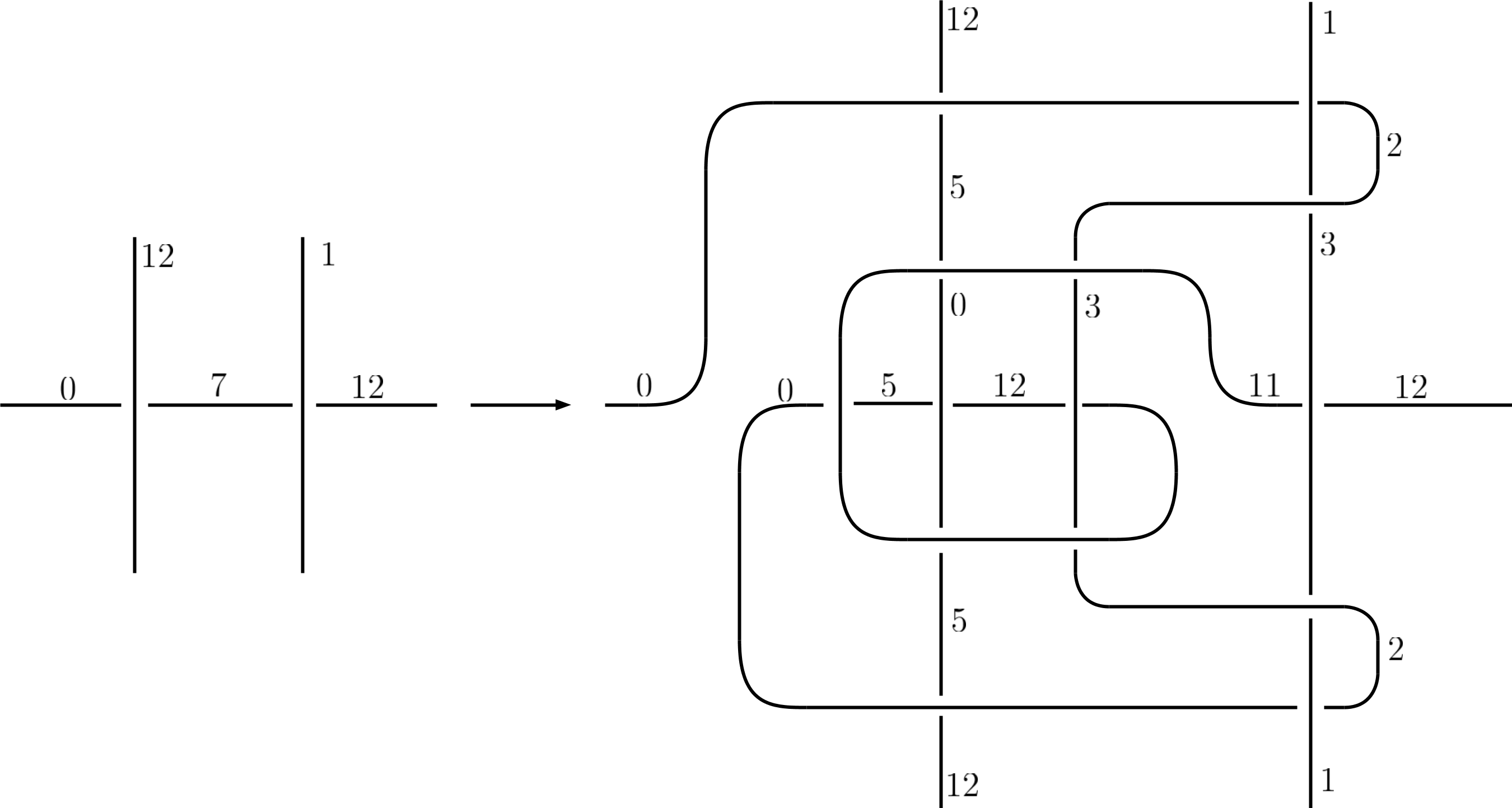} 
	\caption{\label{Fig.88}}
\end{figure}
\begin{figure}[H]
\centering
	\includegraphics[scale=0.11]{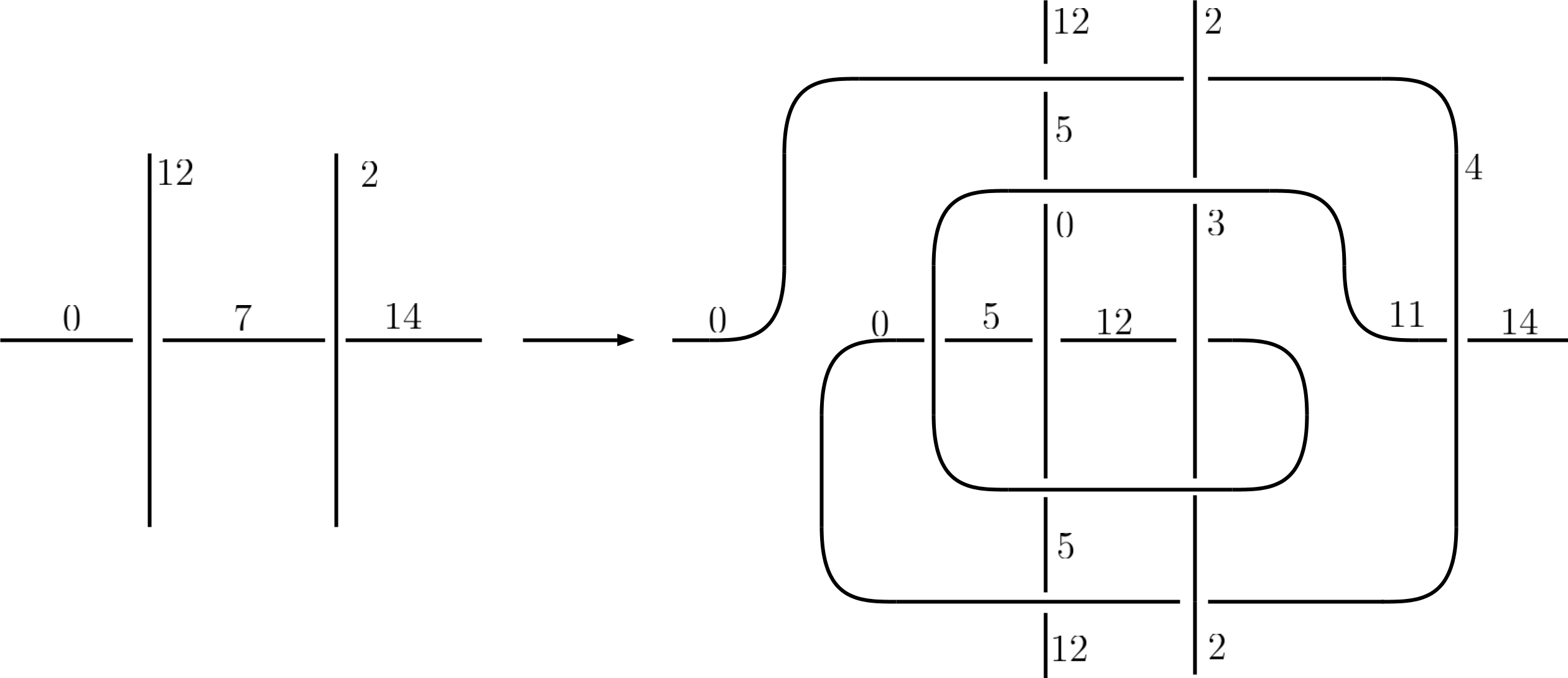} 
	\caption{\label{Fig.90}}
\end{figure}
\begin{figure}[H]
\centering
	\includegraphics[scale=0.11]{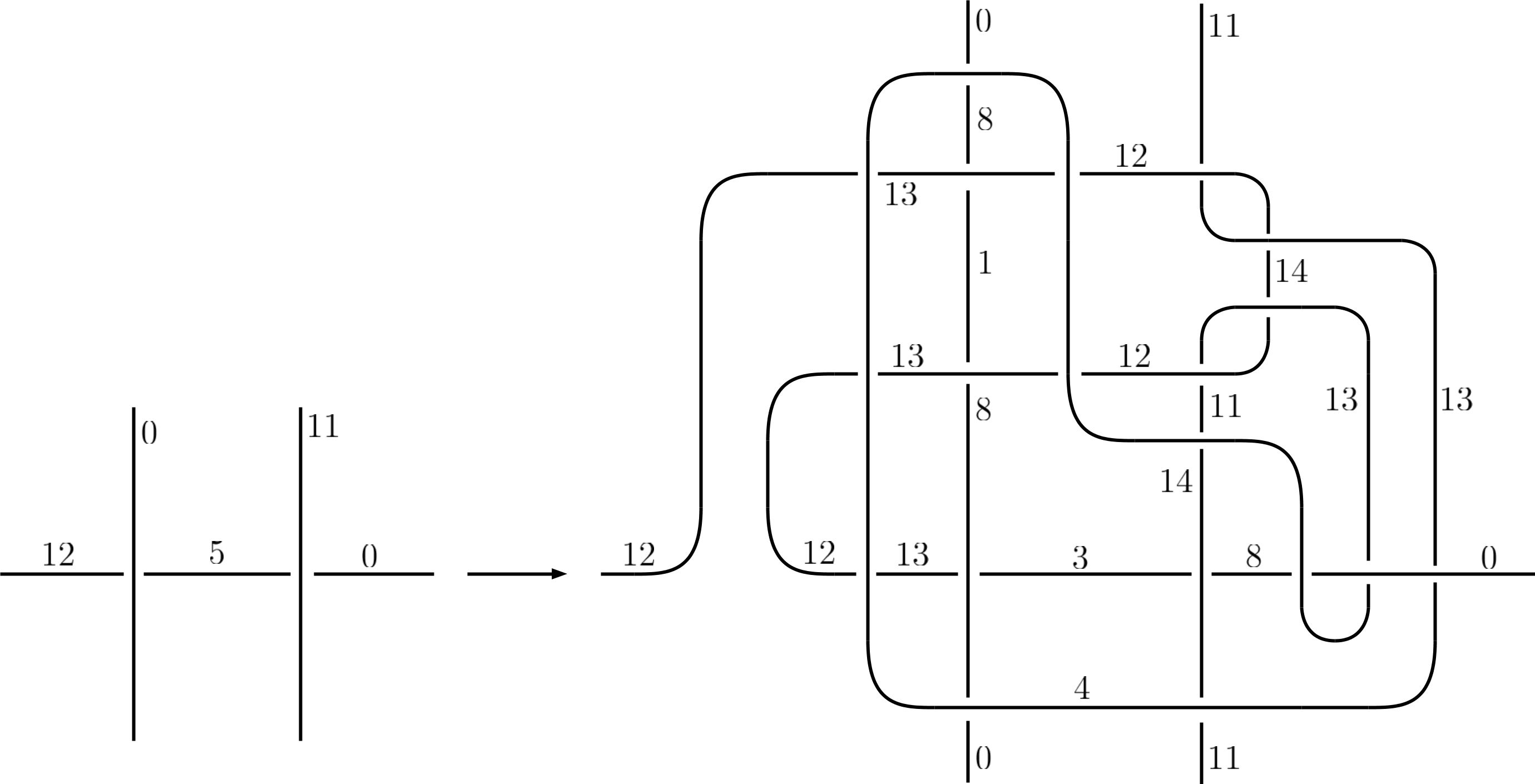} 
	\caption{\label{Fig.93}}
\end{figure}
\begin{figure}[H]
\centering
	\includegraphics[scale=0.1]{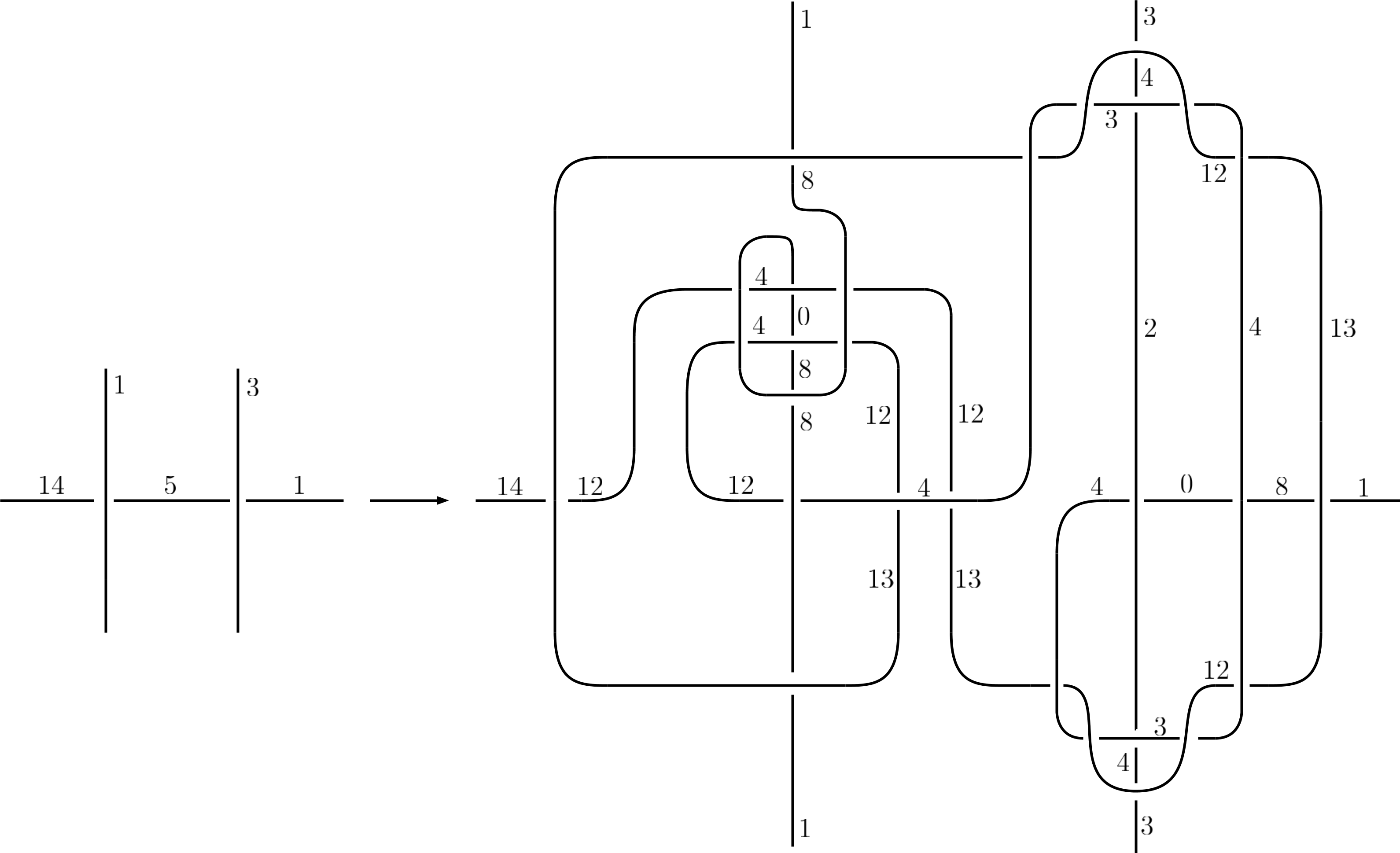} 
	\caption{\label{Fig.f46}}
\end{figure}
\begin{figure}[H]
\centering
	\includegraphics[scale=0.09]{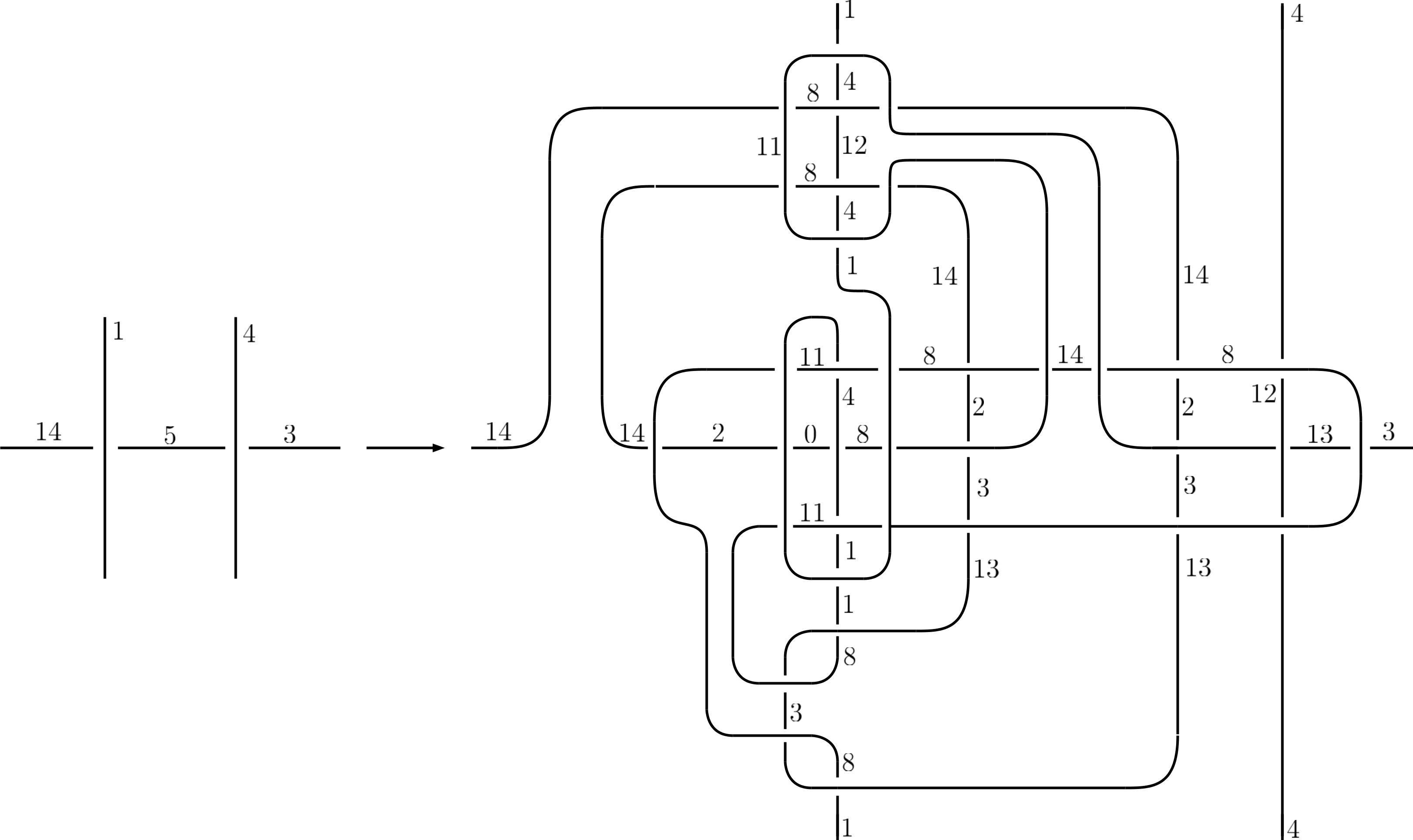} 
	\caption{\label{Fig.95}}
\end{figure}
\begin{figure}[H]
\centering
	\includegraphics[scale=0.09]{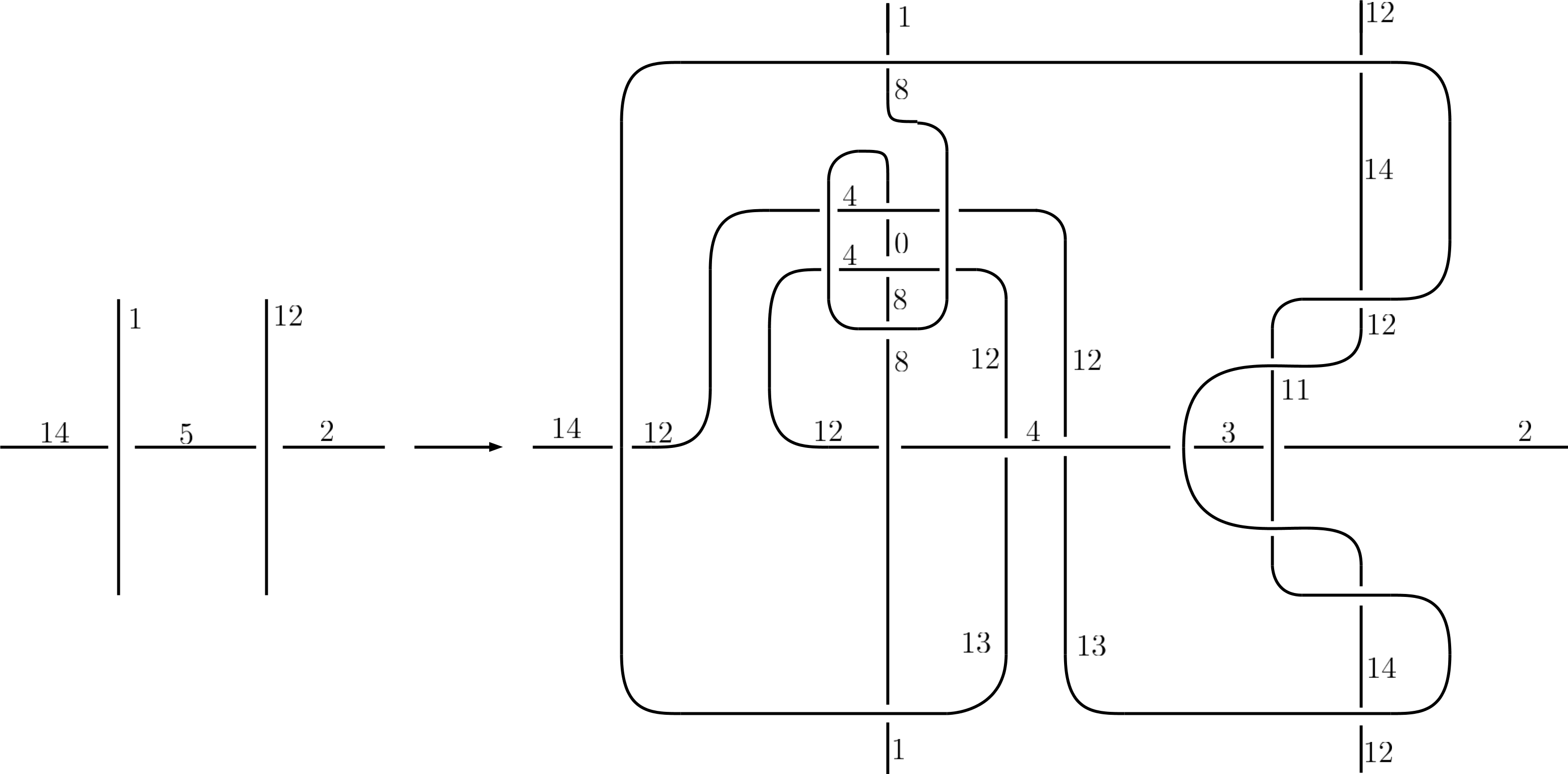} 
	\caption{\label{Fig.96}}
\end{figure}
\begin{figure}[H]
\centering
	\includegraphics[scale=0.09]{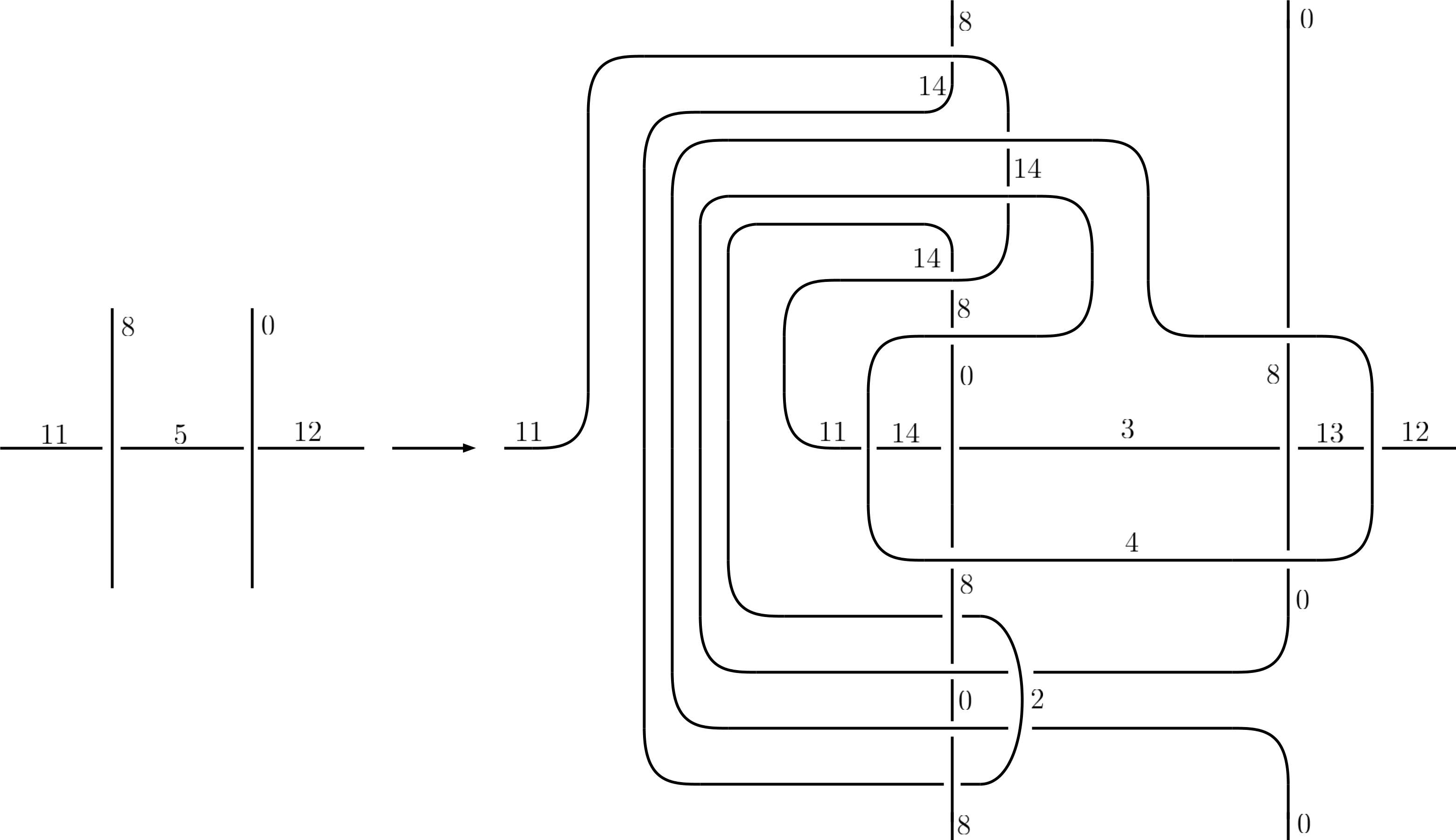} 
	\caption{\label{Fig.92}}
\end{figure}

\begin{figure}[H]
\centering
	\includegraphics[scale=0.09]{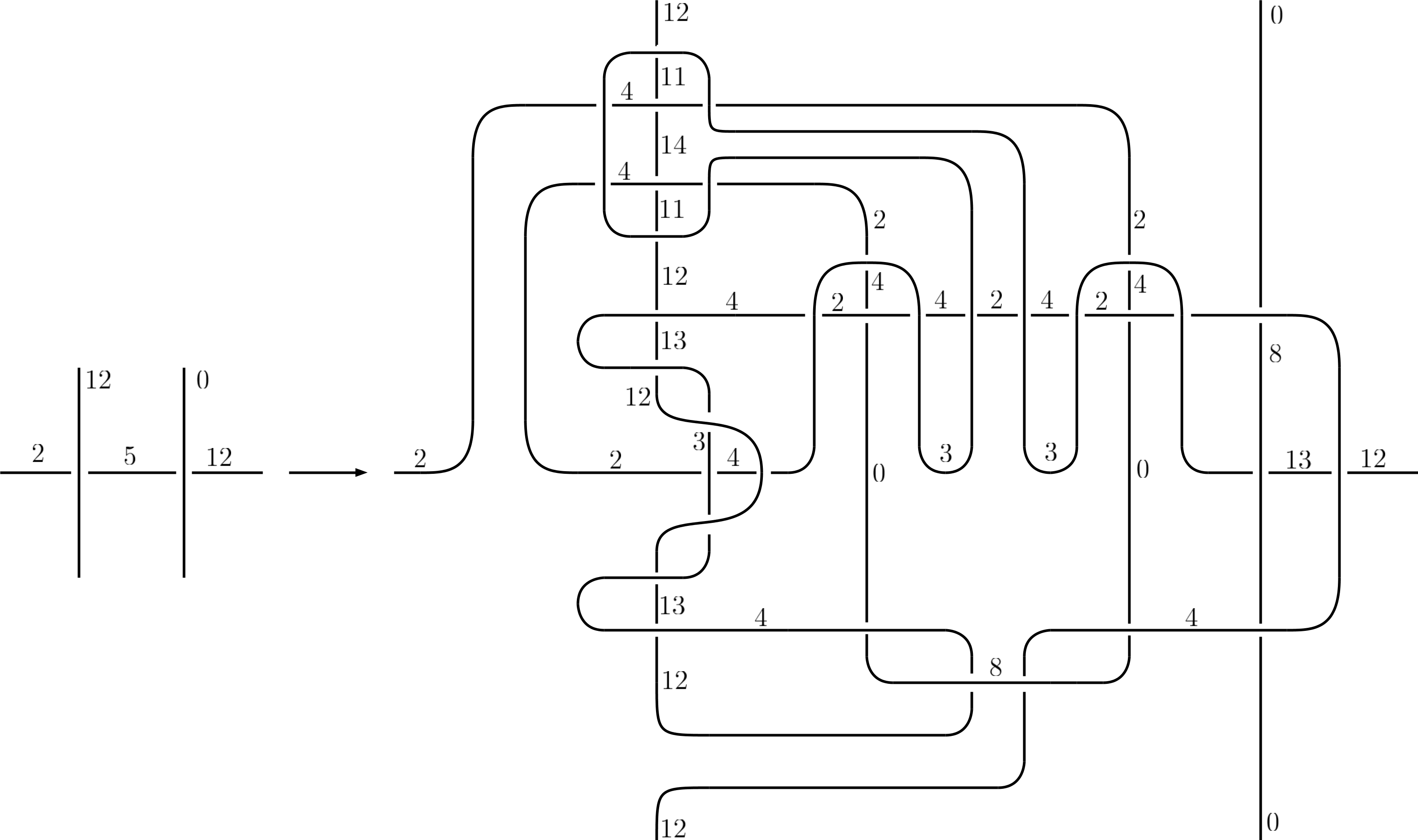} 
	\caption{\label{Fig.94}}
\end{figure}

\begin{figure}[H]
\centering
	\includegraphics[scale=0.09]{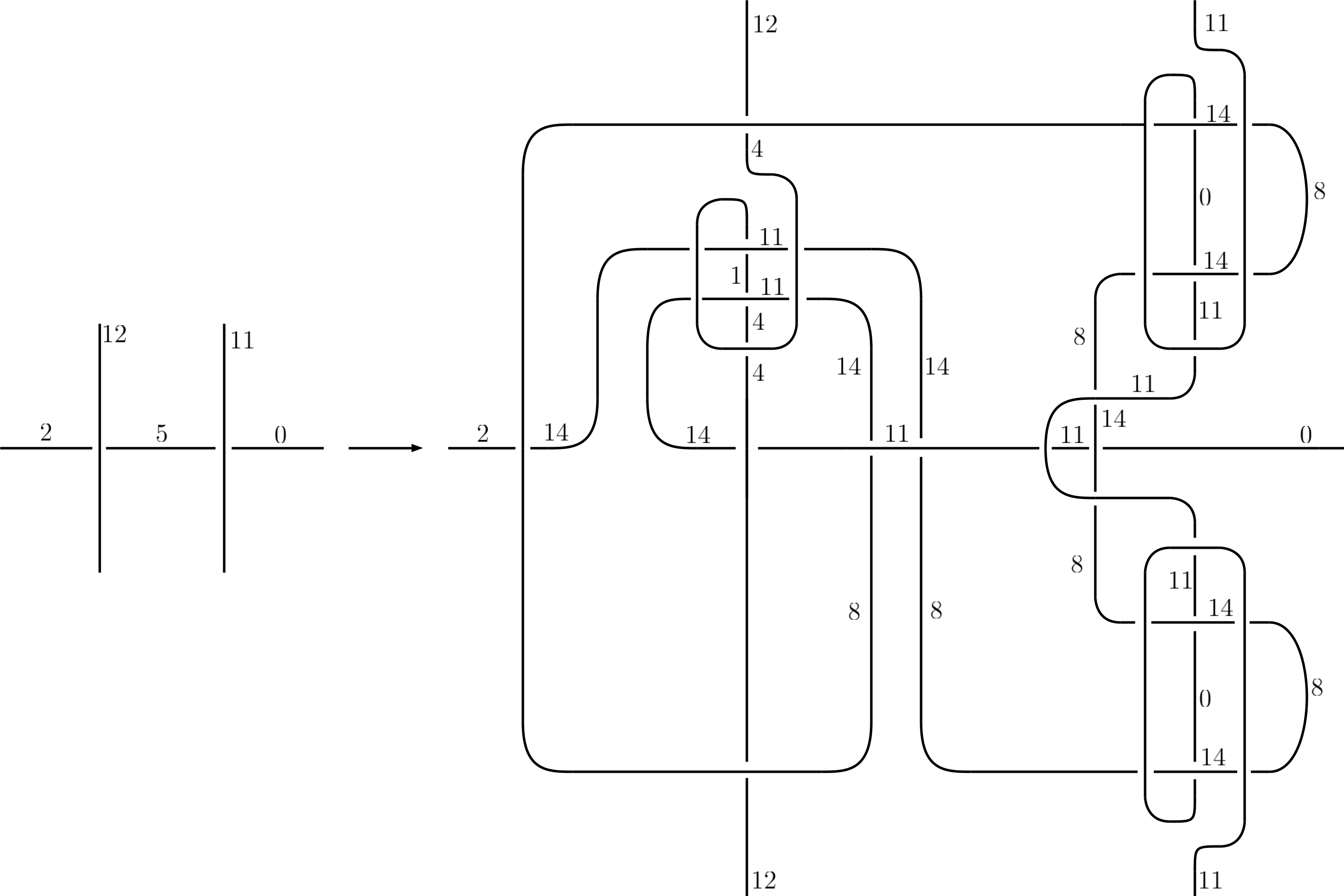} 
	\caption{\label{Fig.97}}
\end{figure}

\begin{figure}[H]
\centering
	\includegraphics[scale=0.1]{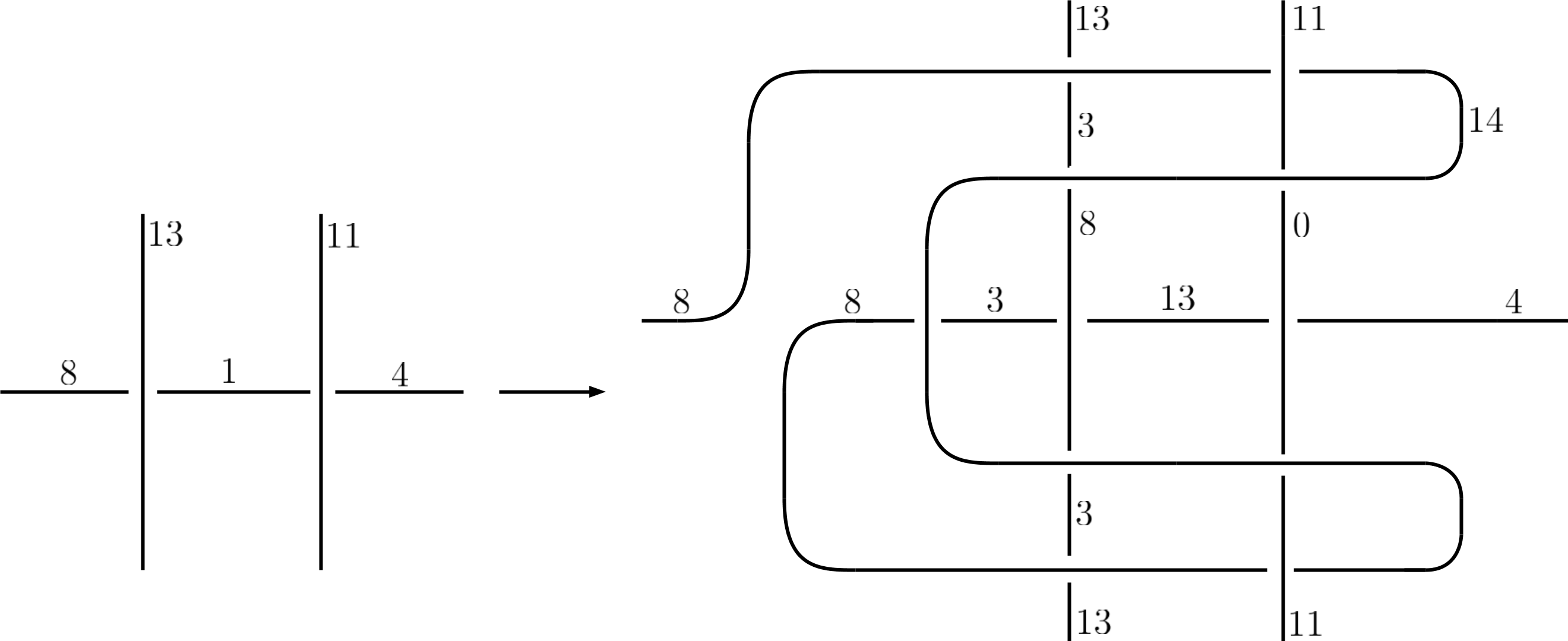} 
	\caption{\label{Fig.98}}
\end{figure}

\begin{figure}[H]
\centering
	\includegraphics[scale=0.101]{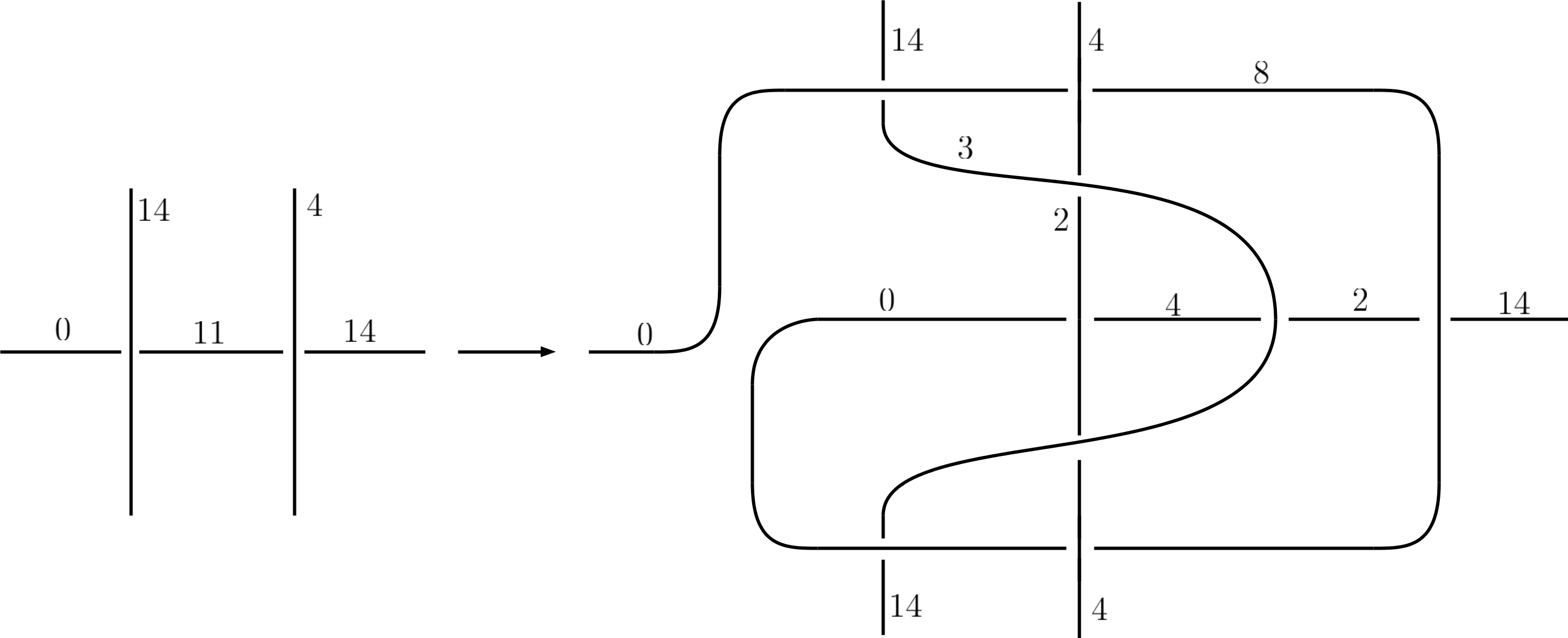} 
	\caption{\label{Fig.99}}
\end{figure}

\begin{figure}[H]
\centering
	\includegraphics[scale=0.12]{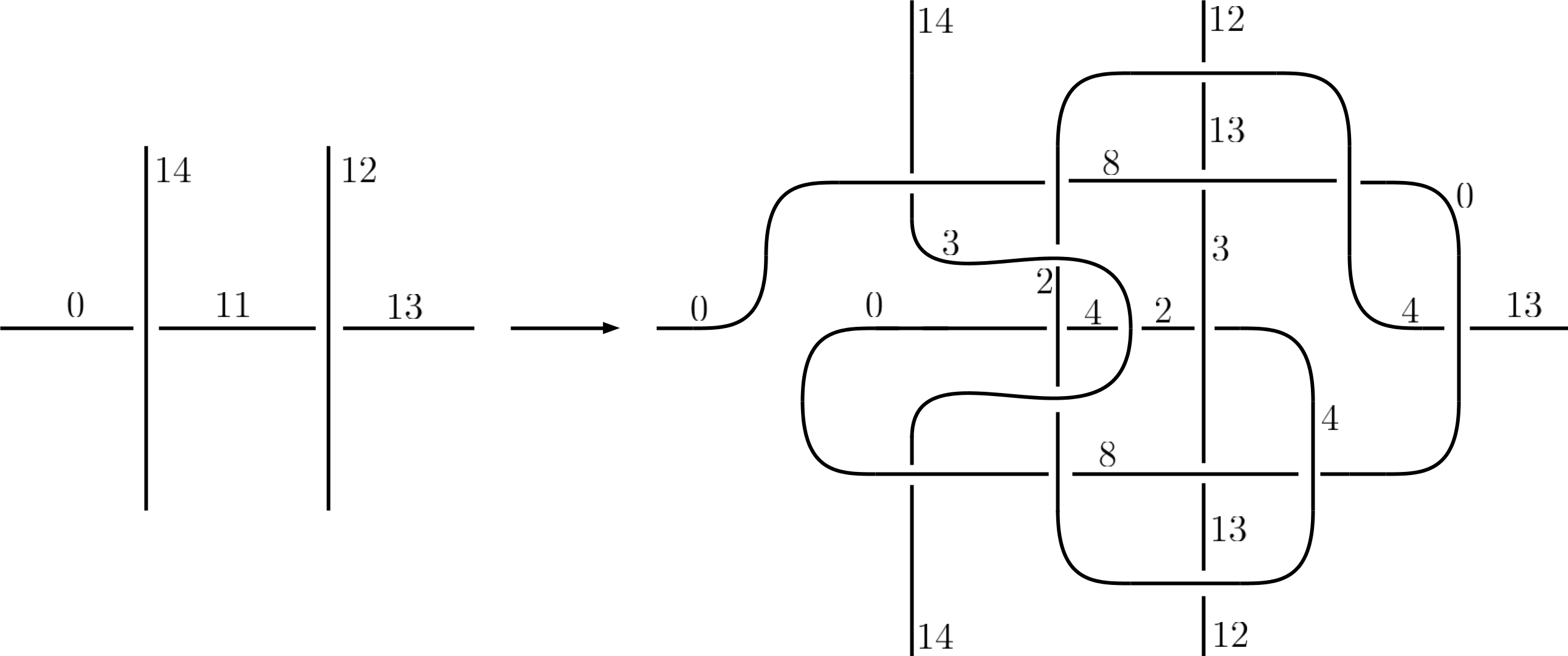} 
	\caption{\label{Fig.100}}
\end{figure}

\begin{figure}[H]
\centering
	\includegraphics[scale=0.12]{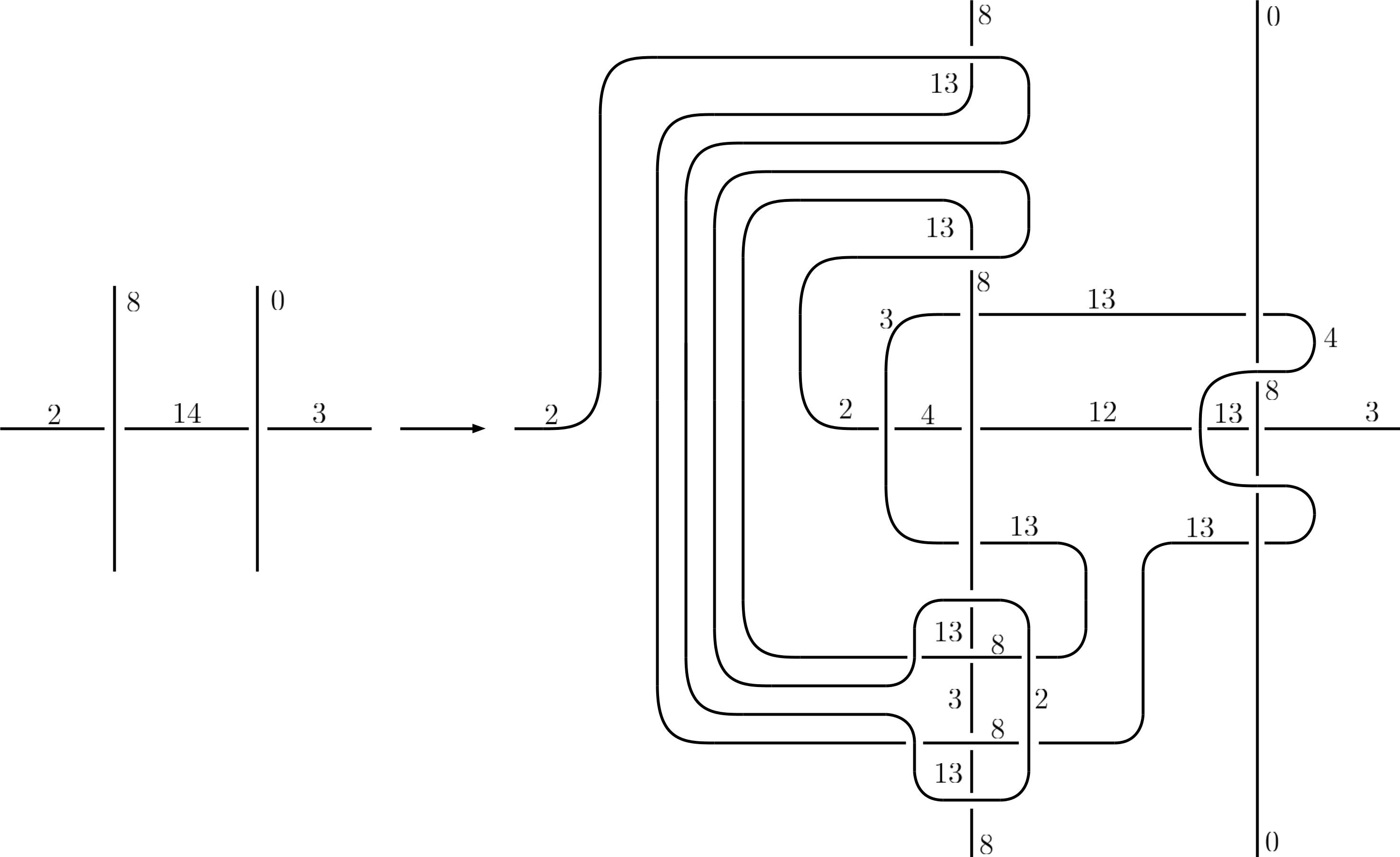} 
	\caption{\label{Fig.101}}
\end{figure}
\begin{figure}[H]
\centering
	  \includegraphics[scale=0.11]{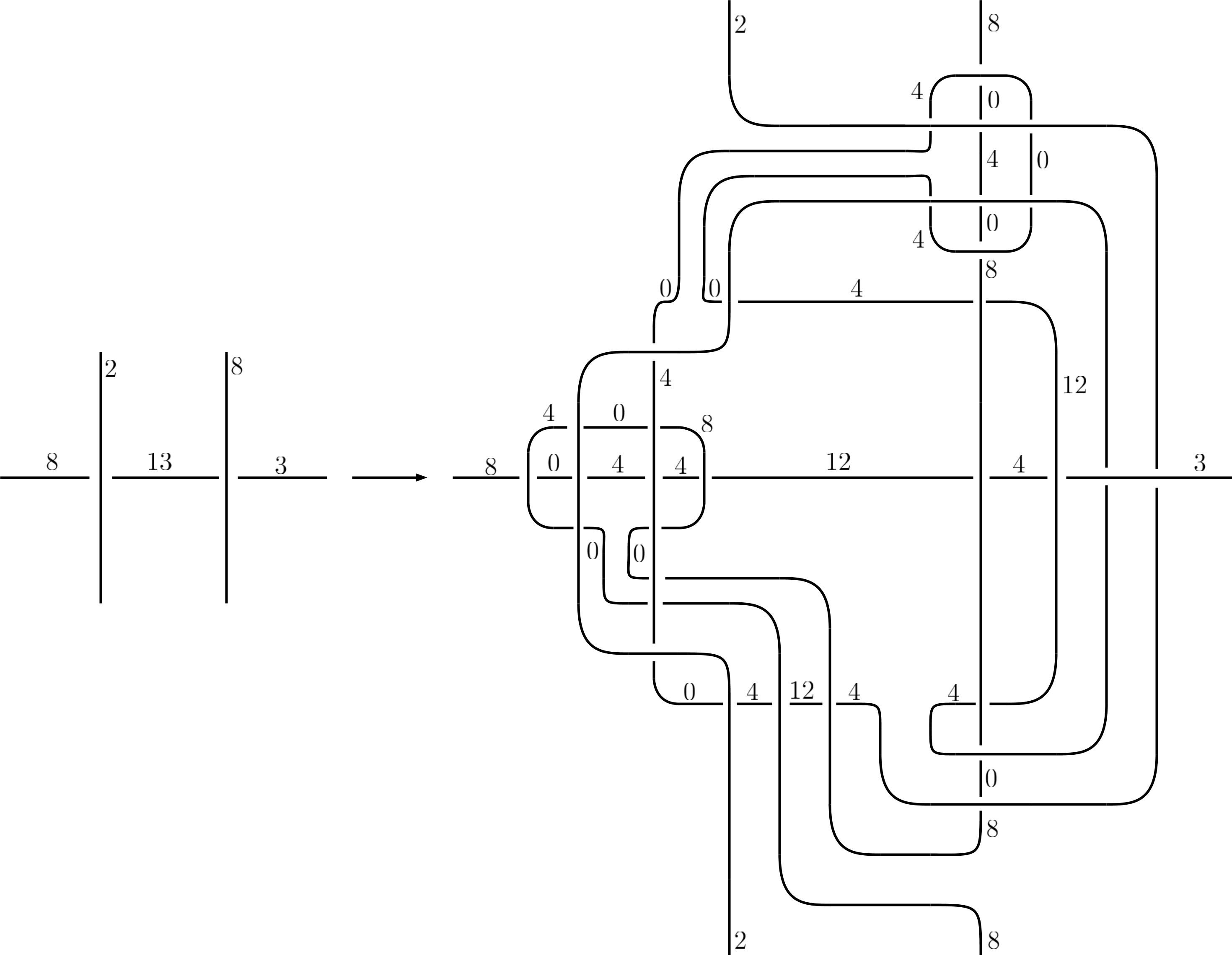}
	\caption{\label{Fig.103}}
\end{figure}
\begin{figure}[H]
\centering
	\includegraphics[scale=0.1]{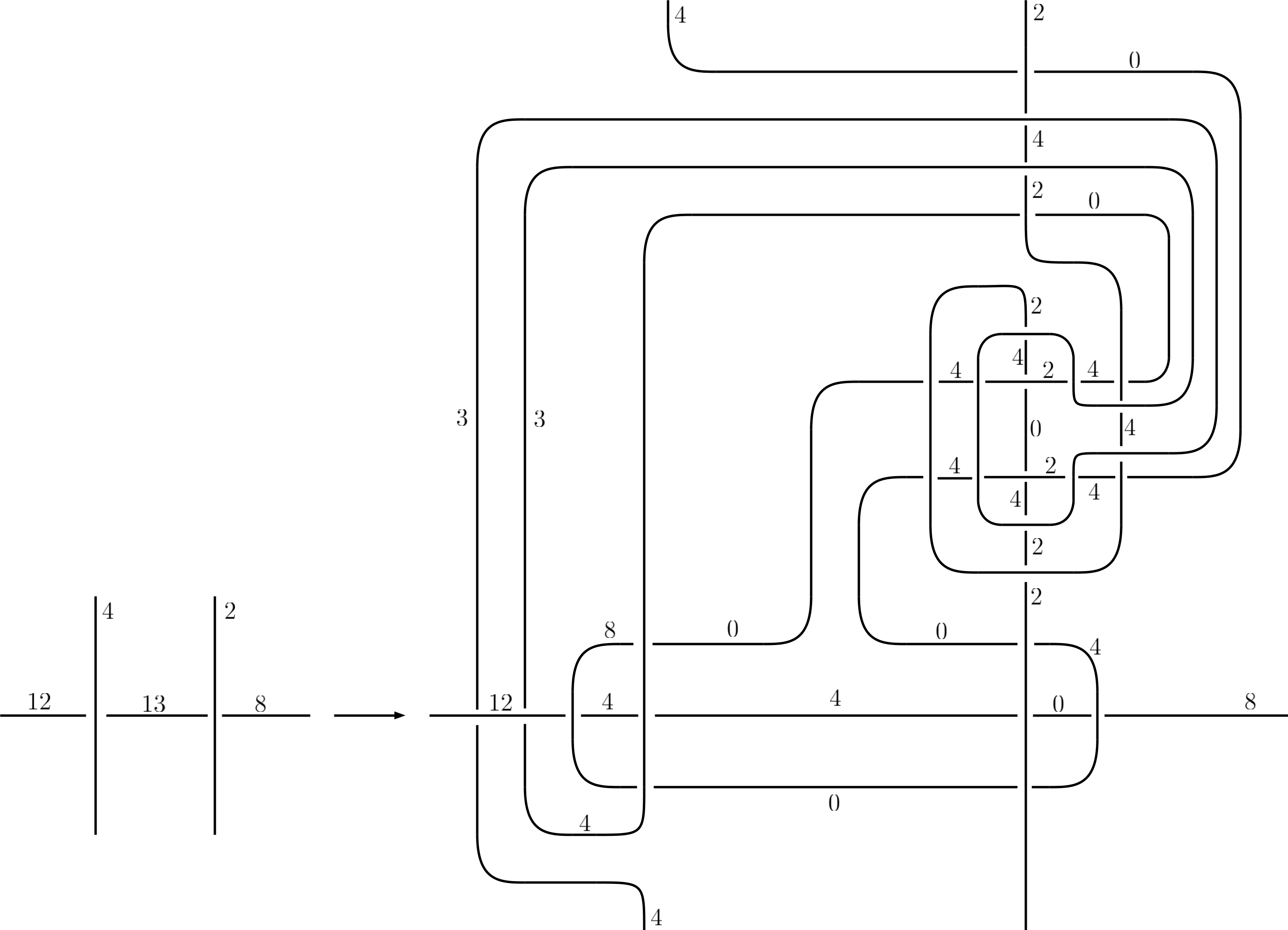}
	\caption{\label{Fig.102}}
\end{figure}
\end{proof}

\end{document}